\documentclass{amsproc}

\usepackage{cleveref}
\usepackage{mathrsfs}
\newtheorem{theorem}{Theorem}[section]
\newtheorem{lemma}[theorem]{Lemma}
\newtheorem{notation}[theorem]{Notation}
\newtheorem{convention}[theorem]{Convention}
\newtheorem*{notation*}{Notation}
\newtheorem*{convention*}{Convention}
\newtheorem{prop}[theorem]{Proposition}

\theoremstyle{definition}
\newtheorem{definition}[theorem]{Definition}
\newtheorem{example}[theorem]{Example}
\newtheorem{exercise}[theorem]{Exercise}

\theoremstyle{remark}
\newtheorem{remark}[theorem]{Remark}

\numberwithin{equation}{section}

\newcommand{\abs}[1]{\lvert#1\rvert}

\usepackage{tikz,xcolor} 
\usepackage{tikz-cd}
\usetikzlibrary{math} 
\usetikzlibrary{3d} 
\usepackage{tikz-3dplot}
\usetikzlibrary{fit}
\usetikzlibrary{positioning}
\usetikzlibrary{calc}

\usepackage{braket}
\usepackage{todonotes}
\usepackage{amssymb}
\usepackage[normalem]{ulem}

\usepackage{dynkin-diagrams}
\usepackage{accents}
\usepackage{mathtools}
\usepackage{url}

\DeclareMathOperator{\GL}{GL}%
\DeclareMathOperator{\Div}{Div}%
\DeclareMathOperator{\dom}{dom}%
\DeclareMathOperator{\ind}{ind}%
\DeclareMathOperator{\Eff}{Eff}%
\DeclareMathOperator{\Fix}{Fix}
\DeclareMathOperator{\ord}{ord}
\DeclareMathOperator{\supp}{supp}
\DeclareMathOperator{\mult}{mult}
\DeclareMathOperator{\crem}{c}
\DeclareMathOperator{\topo}{top}
\DeclareMathOperator{\codim}{codim}
\DeclareMathOperator{\graph}{graph}
\DeclareMathOperator{\Bl}{Bl}

\DeclareMathOperator{\trdeg}{trdeg}
\DeclareMathOperator{\rk}{rk}
\DeclareMathOperator{\id}{id}
\DeclareMathOperator{\Bir}{Bir}
\DeclareMathOperator{\Bs}{Bs}

\DeclareMathOperator{\coeff}{coeff}
\DeclareMathOperator{\red}{red}
\DeclareMathOperator{\divi}{div}
\DeclareMathOperator{\stab}{stab}
\DeclareMathOperator{\Span}{Span}
\DeclareMathOperator{\Hom}{Hom}

\newcommand{\BA}{\mathbb{A}}
\newcommand{\BC}{\mathbb{C}}
\newcommand{\BN}{\mathbb{N}}
\newcommand{\BZ}{\mathbb{Z}}
\newcommand{\BQ}{\mathbb{Q}}
\newcommand{\BR}{\mathbb{R}}
\newcommand{\BP}{\mathbb{P}}
\newcommand{\BF}{\mathbb{F}}

\newcommand{\OC}{\mathscr{C}}
\newcommand{\OP}{\mathscr{P}}
\newcommand{\OQ}{\mathscr{Q}}
\newcommand{\OR}{\mathscr{R}}
\newcommand{\OOE}{\mathscr{E}}
\newcommand{\OO}{\mathscr{O}}
\newcommand{\Fm}{\mathfrak{m}}

\DeclareMathOperator{\Cr}{Cr}
\DeclareMathOperator{\Aut}{Aut}
\DeclareMathOperator{\Pic}{Pic}
\newcommand{\pain}[1]{\operatorname{P}_{\mathrm{#1}}}

\newcommand{\CU}{\mathcal{U}}
\newcommand{\CW}{\mathcal{W}}
\newcommand{\nought}[1]{\accentset{\circ}{#1}}

\newcommand{\defeq}{\vcentcolon=}

      \numberwithin{equation}{section}
 
\begin{document}

\title{Classical Algebraic Geometry and Discrete Integrable Systems}

\author{Gessica Alecci} 
\address{Politecnico di Torino, Corso Duca degli Abruzzi 24, 10129, Torino (Italy)}
\email{gessica.alecci@polito.it} 

\author{Michele Graffeo}
\address{SISSA, Via Bonomea 265, 34136, Trieste (Italy)}

\email{mgraffeo@sissa.it}

\author{Alexander Stokes}
\address{Waseda Institute for Advanced Study, Waseda University, 1-21-1 Nishi Waseda Shinjuku-ku Tokyo 169-0051 (Japan)}
\email{stokes@aoni.waseda.jp}
\thanks{We thank Giorgio Gubbiotti for many useful discussions. We thank the anonymous reviewer for the helpful comments and suggestions, which greatly improved the notes. A.~S. was supported by Japan Society for the Promotion of Science (JSPS) KAKENHI Grant Number 24K22843. G.~A. thanks SISSA for the hospitality while writing these notes.  G.~A. and M.~G. are members of the GNSAGA - INdAM}
\date{May 15, 2025 and, in revised form, XXXX.}



\begin{abstract}
The aim of these notes is to present an accessible overview of some topics in classical algebraic geometry which have applications to aspects of discrete integrable systems. 
Precisely, we focus on surface theory on the algebraic geometry side, which is applied to differential and discrete Painlev\'e equations on the integrable systems side.
Along the way we also discuss the theory of resolution of indeterminacies, which is applied to the cohomological computation of algebraic entropy of birational transformations of projective spaces, which is closely related to the integrability of the discrete systems they define.

\end{abstract}

\maketitle


\section*{Introduction}
Classical algebraic geometry interacts with the theory of integrable systems in many ways.
In these notes we focus on some topics that have appeared frequently in SIDE conferences over the years, namely those related to singularity confinement and algebraic entropy, as well as the Okamoto--Sakai theory of spaces of initial conditions for differential and  difference Painlev\'e equations.
These have deep connections to objects of study in classical algebraic geometry. 
For instance, the theory of rational elliptic surfaces provides many of the foundations for the Okamoto--Sakai theory. 
Also, singularity confinement and algebraic entropy are deeply related to notions of regularisation of birational transformations of complex projective spaces. 
The aim of these lectures is to present a hands-on  introduction to some of the tools from classical algebraic geometry that are needed to face problems coming from (discrete) integrable systems.

In \Cref{sec:AG}, we  discuss with explicit formulas and examples the resolution of indeterminacies of maps through the blow-up procedure and its generalisations, and provide the reader with the necessary vocabulary to understand Sakai's description of discrete Painlev\'e equations in terms of generalised Halphen surfaces.

In \Cref{subsec:varieties}, we introduce the basic notions of varieties and morphisms. Then, in \Cref{subsec:ratmap,subsec:linebundledivisors}, we demonstrate how to relate projective varieties via rational maps, and explain how linear systems can be used to produce rational maps that take values in projective spaces.  In \Cref{subsec:intersection,subsec:canonical}, we begin to focus on surface theory by introducing the intersection pairing on smooth, projective, complex algebraic surfaces, as well as their canonical bundles. Finally, in \Cref{subset:ratsurf}, we present the first examples of rational surfaces.

In \Cref{sec:2}, we demonstrate the calculation of the action on cohomology of a map, which can be used to compute its algebraic entropy. Precisely, in \Cref{subset:degreebir} we present the main definition and properties of the degree of a birational transformation of $\BP^n$ and of algebraic entropy and then, in 
\Cref{subsec:2example}, we provide an explicit example of computation in dimension three.
 
In \Cref{sec:sec3}, we shift our focus to applications of the concepts from  \Cref{subsec:varieties,subsec:ratmap,subsec:linebundledivisors} to specific situations in integrable systems, namely the descriptions of differential and discrete Painlev\'e equations in terms of rational surfaces associated with affine root systems.
In \Cref{subsec:okamotospace} we will explain Okamoto's construction of spaces of initial conditions for the Painlev\'e differential equations, and in \Cref{example:wp} demonstrate the explicit calculations involved. 

Then we will move on to the more general Sakai framework for discrete and differential Painlev\'e equations in terms of generalised Halphen surfaces, beginning in \Cref{subsec:QRT} with methods for constructing spaces of initial conditions for discrete systems defined by birational maps.
In \Cref{subsec:generalisedHalphensurfaces}  we use the terminology established in \Cref{sec:AG} to describe the rational surfaces appearing in the Sakai framework, give an account of their classification in \Cref{subsec:classificationofSakaisurfaces} and finally give the definition of discrete Painlev\'e equations in terms of symmetries of these surfaces in \Cref{subsec:symmetries}.
We conclude by illustrating the general theory in the example about surfaces of type $D_5^{(1)}$ in \Cref{sec:KNY}.

\section{Algebraic geometry}\label{sec:AG}
\subsection{Quasi-projective varieties}\label{subsec:varieties}
We work over the field of complex numbers $\BC$. We denote by $R=\BC[x_0,\ldots,x_n]$ the polynomial ring in $n+1$ variables, without mentioning the dependence on $n$. We also denote by $\BP^n=\BP(\BC^{n+1})$ the $n$-dimensional projective space. Notice that, to ease the reader, we omit the field $\BC$ from the notation. Similarly, we denote the $n$-dimensional complex affine space by $\BA^n$, but when dealing with differential equations we will switch to the more appropriate vector space notation $\BC^n$ and vice versa, depending on our convenience.

Given a subset $X\subset \BP^n$, we denote by $I(X)\subset R$ the ideal
\[
I(X)=\left(\Set{f\in R | f \mbox{ is homogeneous and } f(x)=0,\ \forall x\in X}\right)\subset R.
\]
Clearly, $I(X)$ is a homogeneous ideal by definition. Moreover, the ideal $I(X)$ is \textit{radical}, i.e. if $f^m\in I(X)$ for some $f\in R$ and some $m\ge 0$, then $f\in I(X)$. Finally, as a consequence of \textit{Hilbert's basis theorem} \cite[Theorem 1.2]{EISENBUD}, the ideal $I(X)$ is finitely generated.

\begin{definition}
    A   subset $X\subset \BP^n$ is a \textit{projective variety} if  
    \[
    X=\Set{p\in \BP^n| f (p)=0,\ \forall f\in I(X)}\!.
    \]
    Given a set of generators $\Set{f_1,\ldots,f_s}$ of the ideal $I(X)$, we write sometimes $X=V(f_i\ |\ i=1,\ldots,s)$ to keep track of the $f_i$'s. We say that $X $ is a \textit{hypersurface} if $X=V(f)$, for some homogeneous $f\in R$.
\end{definition}

\begin{remark}\label{rem:chow}
    As a consequence of \textit{Chow's theorem} \cite[Chap I.3.1]{GRIHARR}, any complex (analytic) variety  holomorphically embedded in $\BP^n$ admits the structure of a projective variety. Recall that there is a bijective correspondence (\textit{GAGA}, \cite{GAGA}), preserving exactness, cohomology, and all classical constructions between algebraic and analytic coherent sheaves. Therefore, it is not restrictive to consider a closed analytic subvariety of $\BP^n$ as an algebraic one. 
\end{remark}

We endow the projective space $\BP^n$ with the structure of a \textit{Zariski topological space}. Precisely, we declare a subset $\CU\subset \BP^n$ open if its complement $\BP^n\setminus \CU$ is a projective variety.
\begin{definition}
    A \textit{quasi-projective variety}, or simply a \textit{variety}, is  $X\cap 
\CU \subset \BP^n$ for some projective variety $X\subset \BP^n$ and some open subset $\CU\subset\BP^n$. We endow quasi-projective varieties with the Zariski topology induced by the ambient space $\BP^n$. 
    A quasi-projective variety $X$ is \textit{irreducible} if it does not have two proper and disjoint open subsets. Finally, a \textit{subvariety} is a subset of a variety which is a variety itself. 
\end{definition}
\begin{remark}     Notice that projective varieties are precisely the
quasi-projective varieties that are closed in the Zariski topology, i.e. those having $\mathcal{U} = \BP^n$.

It is worth mentioning that often, in the literature, the term projective variety is used to name irreducible Zariski closed subsets of the projective space and similarly for quasi-projective varieties. Here, we adopt the (less usual) terminology from \cite{GEOFSCHEME} not requiring irreducibility, because in many instances we work with not necessarily irreducible geometrical objects.
\end{remark}

\begin{notation}\label{rmk:coorchart}
    Recall that the projective space $\BP^n$ is covered by the \textit{coordinate atlas} $\mathscr U=\Set{\CU_i}_{i=0}^n$ where $\CU_i=\Set{[x_0:\cdots:x_n]\in\BP^n|x_i\not=0}\cong\BA^n$. We use this atlas to endow  subvarieties of $\BP^n$ with an atlas.  If many sets of variables are involved, we write $\CU_{x_i}$ in place of $\CU_i$.

    Similarly we have a coordinate atlas on $\BP^1\times \BP^1$. Precisely, we put 
    \[
    \mathcal{U}_{i,j}=\Set{([x_0:x_1],[y_0:y_1])\in\BP^1\times\BP^1| x_i\not=0,\ y_j\not=0}\!.
    \]
    In order to ease the notation we put 
    \[
    x=\frac{x_1}{x_0},\    X=\frac{x_0}{x_1},\    y=\frac{y_1}{y_0},\    Y=\frac{y_0}{y_1}. 
    \]
    Therefore, we get the coordinate charts
    \begin{equation*}
    \begin{aligned}
        &\mathcal{U}_{0,0} \cong \BA^2_{(x,y)}, &&\qquad\mathcal{U}_{1,0} \cong \BA^2_{(X,y)}, \\  
        &\mathcal{U}_{0,1}  \cong \BA^2_{(x,Y)}, &&\qquad\mathcal{U}_{1,1}  \cong \BA^2_{(X,Y)}. 
        \end{aligned}
    \end{equation*}   
    
\end{notation}

\begin{definition}
    Let $X\subset \BP^n$ be a quasi-projective variety. A \textit{regular function} on $X$ is a function $f:X\rightarrow\BC$ such that, for every point $p\in X$, there exists an open neighborhood $\CU\subset X$ of $p$ and two homogeneous polynomials $g,h\in R$ of the same degree such that $V(h)\cap \CU=\varnothing$ and $f|_\CU\equiv g/h$.   We denote by $\BC[X]$ the set of regular functions on $X$.
\end{definition}
 
\begin{definition}  Let $X\subset \BP^n$ and $Y\subset\BP^m$ be two quasi-projective varieties. A continuous map $\varphi:X\to Y$ is a \textit{morphism}, if $f\circ \varphi:\varphi^{-1}(\CU) \to \BC$ is a regular function   for every $\CU\subset Y$ open and every $f\in\BC[\CU]$. 

\begin{remark}\label{rem:morphismring}
    Notice that the set $\BC[X]$ admits a structure of ring and that we have $\BC[X]\cong \BC$ whenever $X$ is a projective variety.
\end{remark} 

Composition of morphisms is defined in the usual way.  We say that a morphism $\varphi:X\to Y$ is \textit{dominant} if its \textit{image} $\varphi(X)$ is dense in $Y$. Sometimes, in this case we will also say that $X$ \textit{dominates} $Y$. It is an \textit{isomorphism} if there exists a morphism $\psi:Y\to X$ such that $\psi\circ \varphi \equiv \id_X $ and $\varphi\circ \psi\equiv \id_Y$.
\end{definition}

\begin{remark}Let $X\subset \BP^n$ and $Y\subset\BP^m$ be two quasi-projective varieties.    Given $m+1$ homogeneous polynomials $f_0,\ldots,f_m\in R$ of the same degree such that
    \begin{itemize}
        \item $X\cap  V(f_0,\ldots,f_m)=\varnothing$, and
        \item  $[f_0(p):\cdots:f_m(p)]\in Y$, for all $p\in X$,
    \end{itemize}
   one can define a morphism $f : X\to Y$.
\end{remark}

Many classical geometrical objects admit structures of projective varieties, i.e. they can be holomorphically embedded in projective spaces, see \Cref{rem:chow}. A basic example is provided by \textit{Segre embeddings} of products of projective spaces, which we explain in \Cref{exa:1.1}. 
In \Cref{exa:1.2} we discuss the two-dimensional case in more detail.

\begin{example} \label{exa:1.1}

Fix some $m,n\ge 1$ and put homogeneous coordinates $x_i$, $y_j$, $z_{i,j}$, for $i=0,\ldots,n$ and $j=0,\ldots,m$, on $\BP^n,\BP^m$ and $\BP^{mn+m+n}$ respectively. Let us denote the points in $\BP^{mn+m+n}$ as matrices
\[\begin{bmatrix}
    z_{0,0}&\cdots&z_{0,m}\\
    \vdots&\ddots&\vdots\\
    z_{n,0}&\cdots&z_{n,m}\\
\end{bmatrix}\in \BP^{nm+n+m}.\]
Then, the map 
\[
\begin{tikzcd}[row sep=tiny]
\BP^n\times\BP^m\arrow[r,"s_{n,m}"]&\BP^{nm+n+m}\\
([x_0:\cdots:x_n],[y_0:\cdots:y_m])\arrow[r,mapsto]& \begin{bmatrix}
    x_{0}y_0&\cdots&x_{0}y_m\\
    \vdots&\ddots&\vdots\\
    x_{n}y_0&\cdots&x_{n}y_m\\
\end{bmatrix}
\end{tikzcd}
\]
is an isomorphism between $\BP^n\times\BP^m$ and the projective variety 
\[
S_{n,m}=\Set{[M]\in\BP^{mn+n+m}|\rk M\le 1}\!.
\]
The morphism $s_{n,m}$ is the \textit{Segre $(n,m)$-embedding}, see \cite[Sec. 5.1]{Shafarevich} for more details. On the open subset $S_{n,m}\cap\CU_{z_{i,j}}$, its inverse has the form
\[
\begin{tikzcd}[row sep=tiny]
S_{n,m}\cap\CU_{z_{i,j}}\arrow[r]&\BP^n\times\BP^m\\
\begin{bmatrix}
    z_{0,0}&\cdots&z_{0,m}\\
    \vdots&\ddots&\vdots\\
    z_{n,0}&\cdots&z_{n,m}\\
\end{bmatrix}\arrow[r,mapsto]& ([z_{0,j}:\cdots:z_{n,j}],[z_{i,0}:\cdots:z_{i,m}]).
\end{tikzcd}
\]

\end{example}

\begin{remark}
    We remark that, since Segre embeddings endow products of projective spaces with structures of projective varieties, many of the properties we will state for projective spaces and their subvarieties in the next sections can be extended to products.
\end{remark}

\begin{example} \label{exa:1.2}In this example, we focus on the Segre embedding $s_{1,1}$, as it will be relevant in the rest of these notes.     Consider the quadric hypersurface $Q=V(z_0z_3-z_1z_2)\subset \BP^3$. Then, the association
    \[
    \begin{tikzcd}[row sep=tiny]
        Q\arrow[r]& \BP^1\times \BP^1\\
        {[}z_0:z_1:z_2:z_3{]}\arrow[r,mapsto]&({[}z_0:z_1{]},{[}z_0:z_2{]})
    \end{tikzcd}
    \]
    is an isomorphism as shown in \Cref{exa:1.1}.  Therefore, since all the smooth quadrics of $\BP^3$ are $\BP\GL(4,\BC)$-equivalent, they are all isomorphic to $\BP^1\times \BP^1$. Notice also that the quasi-projective variety $Q\setminus V(z_0)$ agrees with the coordinate chart 
    \[
    \CU_{0,0}=\left\{ ([x_0:x_1],[y_0:y_1])\in\BP^1\times\BP^1\ |\ x_0,y_0\not=0 \right\}\cong\BA^2.
    \] 

    The other coordinate charts $\CU_{i,j}$ for $(i,j)\in\Set{0,1}^{\times 2}$ of $\BP^1\times\BP^1$ are easily recovered in a similar way.
\end{example}

\subsection{Rational maps}\label{subsec:ratmap} 
Morphisms in algebraic geometry are very rigid and hence not suitable for classification, see \Cref{rem:morphismring}. Instead, algebraic geometers use rational maps and classify varieties up to birational equivalence.

\begin{definition}
  Let $X,Y$ be two irreducible quasi-projective varieties. A \textit{rational map} $\phi:X\dashrightarrow Y$ is the datum of a pair $(\CU,\varphi)$, where $\CU\subset X$ is an open subset and $\varphi:\CU\rightarrow Y$ is a morphism not   extendable to any proper open subset $\CU\subsetneq \CU' \subset X$. 
     
     We say that $\CU$ is the \textit{domain} of $\phi$  and we denote it by $\CU=\dom (\phi)$. On the other hand, the \textit{indeterminacy locus} of $\phi $ is its  complement  $\ind(\phi)=X\setminus \dom(\phi)$. The \textit{image} of $\phi$ is $\phi(X)=\varphi(\CU)$, i.e. the image of the morphism $\varphi$ corresponding to $\phi$. The map $\phi$ is \textit{dominant} if its image $\phi(X)\subset Y$ is dense.

     We denote the set of rational maps between $X$ and $Y$ by $\BC(X,Y)$. If $Y=\BA^1$, we say that $\phi$ is a \textit{rational function} and we put $\BC(X)=\BC(X,\BA^1)$.
 \end{definition}

 \begin{remark}
 We remark that the composition of rational maps is in general defined only for dominant maps. Let us discuss briefly the definition of composition.
 
Consider two rational maps $\phi\in\BC(X,Y),\theta\in\BC(Y,Z)$ with $\phi $ dominant, corresponding to morphisms $\varphi:\CU\to Y$ and $\vartheta:\mathcal V\to Z$ respectively. The composition $\theta\circ\phi\in\BC(X,Z)$ is the rational map corresponding to a morphism $\varsigma:\CU'\to Z$ where 
 \begin{itemize}
     \item $\CU'\supset \varphi^{-1}(\varphi(\CU)\cap \mathcal  V)$,
     \item $\varsigma|_{ \varphi^{-1}(\varphi(\CU)\cap \mathcal V)}\equiv \vartheta\circ \varphi $,
     \item $\varsigma$ can not be extended to any proper open subset $\CU'\subsetneq \CU''\subset X$.
 \end{itemize}
 Note that the composition is well-defined since the locus where two morphisms agree is closed.
 \end{remark} 
 \begin{definition}
A rational map $\phi\in\BC(X,Y)$ between irreducible varieties is \textit{birational} if there exists a rational map $\psi\in\BC(Y,X)$  such that $\psi\circ \phi \equiv \id_X $ and $\phi\circ \psi\equiv \id_Y$. 
 \end{definition}
\begin{remark}\label{rem:prop ratmap}  
   Given two irreducible quasi-projective varieties $X\subset\BP^n$, $Y\subset \BP^m$ and $m+1$ homogeneous polynomials $f_0,\ldots,f_m\in R$ of the same degree such that
    \begin{itemize}
        \item $X\not\subset V(f_0,\ldots,f_m)$,
        \item   $[f_0(p):\cdots:f_m(p)]\in Y$, for all $p\in X\setminus V(f_0,\ldots,f_m)$,
    \end{itemize}
    one can define an element   $\phi\in\BC(X,Y)$.
\end{remark}

\begin{remark}
The set $\BC(X)$ admits a ring structure induced by the natural ring structure on $\BA^1$. Moreover, since the variety $X$ is irreducible, the ring $\BC(X)$ is actually a field. Precisely, the field $\BC(X)$ is a transcendental extension of the base field $\BC$. In this setting,  the \textit{dimension} of $X $ is
    \[
    \dim X=\trdeg_{\BC}\BC(X),
    \]
    where $\trdeg_{\BC}\BC(X)$ is the minimal number of (transcendental) generators of $\BC(X)$ as a $\BC$-algebra, see \cite[Chap. II.8]{EISENBUD} for more details. 
    
    We do not expand on the purely algebraic definition of dimension as it is classical, technical and out of the scope of these notes. In turn, we give a geometrical interpretation of this notion. The dimension of a quasi-projective variety $X$ is
    \[
    \dim X=\max \Set{\!k\in\BZ _{\ge0}  |  \mbox{ there exists a } \pi:X\dashrightarrow \BP^k \mbox{ dominant}\!}\!,
    \] 
   see \cite[Section 6.1]{Shafarevich}. We refer to varieties of dimension $1, 2, n$, as curves, surfaces and $n$-folds respectively.
\end{remark}

\begin{remark}\label{rem:irr-rational}
For any open subset $\CU\subset X$ of an irreducible variety $X$ we have $\BC(\CU)\cong\BC(X)$ and consequently $\dim \CU=\dim X$. This follows from the fact that, by definition of Zariski topology, all the non-empty open subsets of an irreducible variety are dense.
\end{remark}

    Given a dominant morphism $\varphi:X\to Y$ of irreducible varieties, there is a well-defined notion of \textit{pull-back $\varphi^*$   of rational functions}. We have
 \begin{equation}\label{eqn:pull-back}
         \begin{tikzcd}[row sep = tiny]
        \BC(Y) \arrow[r,"\varphi^*"] &\BC(X)\\
         f  \arrow[r,mapsto] &{ f\circ\varphi }.
    \end{tikzcd}
 \end{equation}
Clearly, this definition does not extend to non-dominant morphisms as the image of $\varphi$ might be contained in the indeterminacy locus of some $f\in\BC(Y)$. 

If $X$ and $Y $ have the same dimension, the pull-back map $\varphi^*$ defines an algebraic field extension $\varphi^*(\BC(Y))\subset\BC(X)$. In this setting, the \textit{degree} $\deg \varphi$ of the morphism $\varphi$ is defined to be the index $[\BC(X):\varphi^*(\BC(Y))]$ of the extension, i.e.
\[
\deg \varphi=[\BC(X):\varphi^*(\BC(Y))].
\]

Geometrically, this translates into the fact that there is an open subset $\CU\subset   X$ such that $\varphi|_\CU:\CU\to\varphi(\CU)$ is a topological cover of degree $d$.

\begin{example}
  Consider the irreducible subvariety $X\subset \BP^2$ defined by $X=V(x_0x_1-x_2^2)$ and consider the following association,
    \[
    \begin{tikzcd}[row sep = tiny]
        X\arrow[r,"\pi"]& \BP^1_{[y_0:y_1]}\\
        {[x_0:x_1:x_2]}\arrow[r,mapsto] & {[x_0:x_1]},
    \end{tikzcd}
    \]
    see \Cref{ex:projections} for a geometrical description of the map $\pi$. This is a morphism\footnote{In particular, it is a rational map.} between irreducible curves. Let us  compute its degree. As observed in \Cref{rem:irr-rational}, we can perform the computation after restricting to the open dense subsets $ \mathcal{X}=X\cap \mathcal{U}_{x_0}\subset X$ and $\mathcal{U}_{y_0}\subset \BP^1$. First we notice that $\mathcal{X}\cong\BA^1_x$, with $x=\frac{x_2}{x_0}$. Locally, the map $\pi$ takes the form
    \[
    \begin{tikzcd}[row sep = tiny]
        \mathcal{X}\arrow[r,"\pi|_{\mathcal{X}}"]& \mathcal{U}_{y_0}\\
        x\arrow[r,mapsto] & x^2.
    \end{tikzcd}
    \]
    As a consequence, we get $\pi^*\BC(\BP^1)=\BC(x^2)\subset \BC(x)=\BC(X)$, which implies $\deg\pi=2$. 
\end{example} 

\begin{definition}
  Let $X=V(f_1,\ldots,f_s)\subset\BP^n$ be an irreducible variety of dimension $\dim X=d$. Consider a coordinate chart $\mathcal X_i=X\cap \CU_i$, for some $i=0,\ldots,n$, and a  point $p\in\mathcal X_i$, see \Cref{rmk:coorchart}. Fix affine coordinates $y_j=x_j/x_i$, for $j\in\Set{0,\ldots,n}\setminus\Set{i}$ on $\CU_i\cong\BA^n$ and put $y_i=1$. Then, the   point $p$ is a \textit{smooth point of} $X$ if the Jacobian matrix 
    \[
    J_{X,p}=\left(\!\left(\frac{\partial}{\partial y_j}{f}_k(y_0,\ldots,y_n)\right)(p)\  |\     k=1,\ldots,s,\ j\in\Set{0,\ldots,n}\setminus\Set{i}\!\right)\!
    \]
    has  rank $\rk J_{X,p}=n-d $, and it is a \textit{singular point} otherwise. The variety $X$ is \textit{smooth} if it has no singular point and it is \textit{singular} otherwise. 

    In the case when $X$ is not irreducible, suppose instead that $X=\bigcup_{i=1}^sV$, for $s\ge2$, is the decomposition of $X$ into irreducible components, and consider a point $p\in X$. If there is a unique index $1\le i \le s$ such that $p\in V_i$ we say that $p$ is a smooth point of $X$ if and only if it is a smooth point of $V_i$. Otherwise we declare it singular point.
\end{definition} 

As a consequence of the principal ideal theorem  \cite[Theorem 10.1]{EISENBUD}, we have the following proposition.
\begin{prop}\label{prop:codim2ind}
    Let $X,Y$ be two irreducible smooth projective varieties and let $\phi\in\BC(X,Y)$ be a rational map. Then, $\codim  \ind(\phi) \ge 2$.
\end{prop}

\begin{remark}\label{rem:indRatOnSurf}
    As a consequence of \Cref{prop:codim2ind}, if $X$ is a smooth projective curve, a rational map $\phi\in\BC(X,Y)$  is a morphism. Moreover, a rational map whose domain is a smooth projective surface is not defined at   finitely many points.
\end{remark}

\begin{definition}\label{def:birati}
    We say that two irreducible varieties $X,Y$ are \textit{birational} to each other, in symbols $X\sim Y$, if there is a birational map $X\dashrightarrow Y$. We denote by $\Bir(X)$ the set of birational transformations of $X$, i.e.
    \[
    \Bir(X)=\Set{\!\phi:X\dashrightarrow X |\phi\mbox{ is birational}\!}\!.
    \]
    A variety is \textit{rational} if it is birational to a projective space.
\end{definition}

\begin{remark}
    Note that $\Bir(X)$ is a group with respect to the composition, see \Cref{rem:prop ratmap}. Moreover, being birational is an equivalence relation for irreducible quasi-projective varieties.
\end{remark}

As expressed in the following theorem, the field of rational functions is a complete invariant for birational equivalence classes of varieties, see \cite[Corollary I.4.5]{Hartshorne77}. 

\begin{theorem}
Two irreducible varieties $X,Y$ are birational to each other if and only if $\BC(X)\cong \BC(Y)$.
\end{theorem}
Fix homogeneous polynomials $f_0,\ldots,f_m\in R$ of the same degree,   generating a radical ideal $(f_0,\ldots,f_m)\subset R$. Consider the rational map  
\[
\begin{tikzcd}[row sep = tiny]
    \BP^n \arrow[r,"\phi"]&\BP^m\\
    p\arrow[r,mapsto] & {[f_0(p):\cdots:f_m(p)]}.
\end{tikzcd}
\]
The \textit{graph} of $\phi$ is the closed subset
\[
\graph(\phi)=\overline{\Set{(x,\phi(x))|x\in \dom(\phi)}}\subset \BP^n\times \BP^m,
\]
 fitting in the following commutative diagram
\begin{equation}\label{eq:resmap}
\begin{matrix}
    \begin{tikzpicture}[scale =2 ] 
    \node at (0,-0.05) {$\graph(\phi)$};
    \node at (1,-1) {${\BP^m}.$};
    \node at (-1,-1) {${\BP^n}$};
    \node[below] at (0,-1) {$\phi$};

    \draw[->] (-0.1,-0.2) -- (-0.9,-0.8);
    \draw[->] (0.1,-0.2) -- (0.9,-0.8);
    \draw[dashed,->] (-0.8,-1) -- (0.8,-1);
    \node[right] at (0.5,-0.5) {\small $\pi_{\BP^m}|_{\graph(\phi)}$};
    \node[left] at (-0.5,-0.5) {\small $\pi_{\BP^n}|_{\graph(\phi)}$};
\end{tikzpicture}
\end{matrix}
\end{equation}
In this setting, we call the diagram \eqref{eq:resmap} the \textit{resolution of indeterminacies} of $\phi$.
Denote by $Z=\ind(\phi)=V(f_0,\ldots,f_m)$ the indeterminacy locus of $\phi$, see \Cref{rem:prop ratmap}. Then, the graph $\graph (\phi)$ is called the \textit{blow-up} of $\BP^n$ with \textit{centre} $Z$ and it is denoted by $\graph(\phi)=\Bl_Z\BP^n$. In this context, the morphism $\pi_{\BP^n}|_{\graph(\phi)}$ is the \textit{blow-up morphism} and the preimage $\pi_{\BP^n}|_{\graph(\phi)}^{-1}(Z)$ is the \textit{exceptional locus}.

If $X\subset \BP^n$ is a quasi-projective subvariety such that $X\not\subset Z$, the \textit{strict transform} $\widehat{X}$ of $X$ is the closure of the preimage 
\[
\widehat{X}=\overline{\pi_{\BP^n}|_{\graph(\phi)}^{-1}(X\setminus Z)} \subset\Bl_Z\BP^n.
\]
Then, if $Z\subset X\subset \BP^n$ and $I_X+I_Z\subset R$ is a radical ideal, we say that the strict transform $\widehat{X}$ is the \textit{blow-up} of $X$ with centre $Z$ and we denote it by $\widehat{X}=\Bl_ZX$.

Finally, if 
\[\varphi: X_s\xrightarrow{\varepsilon_s}X_{s-1}\xrightarrow{\varepsilon_{s-1}}\cdots X_1\xrightarrow{\varepsilon_1}X\]
is a sequence of blow-ups and $p\in X$ is any point, we say that the points in the preimage $\varphi^{-1}(p)\subset X_s$ are \textit{points infinitely near to} $p$. Similarly, if $Y\subset X$ is a subvariety all the irreducible components of $\varphi^{-1}(Y)$ are said to be \textit{infinitely near to} $Y$.
 
\begin{remark}
Let $X\subset \BP^n$ be a quasi-projective variety and let $Z\subset X$ be a closed subset. Consider the blow-up morphism $\varepsilon:\Bl_ZX\to X$ and denote by $E=\varepsilon^{-1}(Z)$ the exceptional locus. Then, the restriction $\varepsilon|_{\Bl_ZX\setminus \varepsilon^{-1}(Z)}:\Bl_ZX\setminus \varepsilon^{-1}(Z)\to X\setminus Z$ is an isomorphism. Moreover, the exceptional locus has codimension 1.
\end{remark}
 
We present below  some basic examples of rational maps and blow-ups. We mostly focus on the two-dimensional setting which will be of our interest in the rest of the notes. Precisely, we discuss projections, blow-ups of  surfaces at a few points and the standard Cremona transformation of the projective plane.  
\begin{example}[Projections]\label{ex:projections} In this example we present  simplest instance of a rational map, namely the projection. Consider two projective subspaces $H,K\subset\BP^n$ such that $H\cong\BP^k$, $K\cong\BP^{n-k-1}$ and $H\cap K=\varnothing$. Without loss of generality, we assume $K=V(x_0,\ldots,x_k)$ and $H=V(x_{k+1},\ldots,x_n)$. Then, the \textit{projection onto $H$ with centre $K$} is the rational map\footnote{ We omit the dependence on $H$ from the notation for the rational map as it will not play any role in what follows.}
\[
\begin{tikzcd}[row sep = tiny]
    \BP^n\arrow[r,dashed,"\pi_{K}"]& H\\
    {[}x_0:\cdots:x_n{]}\arrow[r,mapsto]&{[}x_0:\cdots:x_k{]}.
\end{tikzcd}
\]
Note that $\pi_{K}$ is not defined along $K$, i.e. the indeterminacy locus of $\pi_K$ is $\ind(\pi_{K})=K$. Geometrically, we associate to any $p\in\BP^n\setminus K$ the unique intersection point $\pi_{K}(p)= \langle p,K\rangle\cap H$.  We recall two useful properties of projections.
\begin{itemize}
    \item Any linear subspace $W\subset\BP^n$  such that $K\subset W$ and $W\cong\BP^{n-k}$ is contracted to the point  $p=\pi_K(W)=W\cap  H$.
    \item If $L\subset \BP^n$ is a line  such that $L\not\subset K$ then $\pi_K(L) = p $ is a point if and only if $L\cap K\not=\varnothing$. 
\end{itemize}
 
\end{example}

\begin{example}[Blow-up at a point]\label{exa:plow1} We describe now the projection $\pi_K$ in the case $n=2$ and $\dim K=0$. Without loss of generality we put $K=\Set{e_2}$ and we consider the projection
\[ 
\begin{tikzcd}[row sep = tiny]
\BP^2\arrow[r,dashed,"\pi_{e_2}"]& \BP^1\\
    {[}x_0:x_1:x_2{]}\arrow[r,mapsto]&{[}x_0:x_1{]},
\end{tikzcd}
\]
    from $\BP^2$ with centre the coordinate point $e_2=[0:0:1]$. Then, the blow-up with centre $e_2$ is
    \begin{align*}
\Bl_{e_2}\BP^2&={\graph(\pi_{e_2})}\\&=\overline{\Set{(p,q)\in\dom(\pi_{e_2})\times \BP^1 | \pi_{e_2}(p)=q}}\subset \BP^2\times\BP^1\\
&=\overline{\Set{({[x_0:x_1:x_2]},{[y_0:y_1]})\in\dom(\pi_{e_2})\times \BP^1 | \det\begin{pmatrix}
    x_0&x_1\\y_0&y_1
\end{pmatrix}=0}}\subset \BP^2\times\BP^1\\
&={\Set{({[x_0:x_1:x_2]},{[y_0:y_1]})\in\BP^2\times \BP^1 | x_0y_1-x_1y_0=0}}\subset \BP^2\times\BP^1.
    \end{align*}
    We stress that there is a commutative diagram 
    \[ 
\begin{tikzpicture} [scale = 2]
    \node[right] at (0.2,0) {$\subset \BP^2\times \BP^1$};
    \node at (0,0) {$\Bl_{e_2}\BP^2$};
    \node at (1,-1) {$\BP^1,$};
    \node at (-1,-1) {$\BP^2$};
    \node[below] at (0,-1) {$\pi_{e_2}$};

    \draw[->] (-0.1,-0.2) -- (-0.9,-0.8);
    \draw[->] (0.1,-0.2) -- (0.9,-0.8);
    \draw[dashed,->] (-0.75,-1) -- (0.75,-1);
    \node[right] at (0.5,-0.5) { $\pi_1|_{\Bl_{e_2}\BP^2}$};
    \node[left] at (-0.5,-0.5) { $\pi_2|_{\Bl_{e_2}\BP^2}$};
\end{tikzpicture}
    \]
    where $\pi_2|_{\Bl_{e_2}\BP^2}$ is the blow-up morphism and $\pi_1|_{\Bl_{e_2}\BP^2}$ is a $\BP^1$-fibration as described in \Cref{ex:projections}. 

We give now an explicit affine atlas of the blow-up $\Bl_{e_2}\BP^2$.  Since the blow-up is an isomorphism over $\CU_{x_0},\CU_{x_1}$,   it is enough to give an affine cover  of $\Bl_{e_2}\CU_{x_2}$, see \Cref{rmk:coorchart} for the notation. This consists of the following two charts 
\[
 \CW^{(i)}=(\Bl_{e_2}{\CU_{x_2}})\cap( \CU_{x_2}\times \CU_{y_i})\subset \CU_{x_2}\times \BP^1,
\]
for $i=0,1$. Precisely, if we fix affine coordinates
\[
(X_0,X_1)=\left(\frac{x_0}{x_2},\frac{x_1}{x_2}\right)
\]
on $\CU_{x_2}$, then we have 
    \begin{equation} \label{eq:blow-upchart2}    
        \CW^{(0)}  = \Set{\!((X_0,X_1),[y_0:y_1])\in \CU_{x_2}\times\BP^1 | X_0y_1-X_1y_0=0, ~   
        y_0\neq 0}  \cong \BA^2_{(U,V)}, 
    \end{equation}
    and
    \begin{equation} \label{eq:blow-upchart1} 
        \CW^{(1)}  = \Set{ ((X_0,X_1),[y_0:y_1])\in \CU_{x_2}\times\BP^1 | X_0y_1-X_1y_0=0, ~   
        y_1\neq 0}  \cong \BA^2_{(u,v)}, 
    \end{equation} 
    where
    \[  (U,V ) = \left(\frac{y_1}{y_0} 
        , X_0\right)\qquad \mbox{ and }\qquad ( u ,v)=\left( \frac{y_0}{y_1},X_1\right)\!. \] 
    Note that the local equations of the exceptional curve $E\subset \Bl_{e_2}\BP^2$ in the coordinates $(u,v)$ and $(U,V)$ are $v=0$ and $V=0$ respectively. 
    In terms of the affine coordinates $(X_0,X_1)$ for $\CU_{x_2}$, in which $e_2$ is given by $(X_0,X_1)=(0,0)$, the restriction of the projection $\pi_{e_2}$ to the blow-up, is written in the 
    new coordinates as
    \[
    \begin{tikzcd}[row sep=tiny]
         \CW^{(0)} \cong \BA^2_{(U,V)}\arrow[r,"\pi_{e_2}|_{\CW_0}"]&\CU_{x_2}\\
         (U,V) \arrow[r,mapsto]&  (V,U V),
    \end{tikzcd}\qquad\qquad
     \begin{tikzcd}[row sep=tiny]
         \CW^{(1)} \cong \BA^2_{(u,v)}\arrow[r,"\pi_{e_2}|_{\CW_1}"]&\CU_{x_2}\\
         (u,v) \arrow[r,mapsto]& (u v, v).
    \end{tikzcd}
    \] 
\end{example}

We explain now the blow-up at a point of any smooth surfaces following the ideas in \Cref{exa:plow1}. 

\begin{example}\label{exa:plow11} 
Let $S$ be a quasi-projective surface and $p\in S$ a smooth point. Let $C_1,C_2\subset S$ be two irreducible distinct curves intersecting \textit{transversally} at $p$, i.e. they are smooth at $p$ and their Jacobian matrices are linearly independent. Let $p\in \CU\subset S$ be a smooth and connected open neighbourhood such that $\mathcal{C}_{i }=\CU\cap C_i$ is a smooth curve for $i=1,2$, $\mathcal{C}_{1 }\cap \mathcal{C}_{2 }=\Set{p} $ and\footnote{This can always be achieved paying the price of restricting the open $\CU$.} $\mathcal{C}_{i }=V(f_i)$ with $f_i\in\BC[\CU]$, for $i=1,2$.  
Then, the quasi-projective surface
\[
\widehat \CU=\Set{(q,[\lambda_0:\lambda_1])\in  \CU\times \BP^1| \lambda_0f_1(q)-\lambda_1f_2(q)=0 }\subset \CU\times \BP^1
\]
    is the blow-up\footnote{Formally, it is the strict transform under the blow-up of the ambient space.} of $\CU$ with centre $p$ and the blow-up morphism is $\pi_{\widehat \CU}=\pi_\CU|_{\widehat \CU}$. In particular, we have $E=\pi_{\widehat   \CU}^{-1}(  p)\cong \BP^1$ and  $\pi_{\widehat \CU}|_{\widehat \CU\setminus E}:\widehat \CU\setminus E\to \CU\setminus p$ is an isomorphism.
    
    In order to construct $\Bl_pS$, we first write $S=\CU\cup \mathcal{V}$, where $\mathcal{V}=S\setminus p$. Then, we have $\Bl_pS=\widehat{\CU}\cup \mathcal{V}$ with the obvious gluing. 
\end{example}

\Cref{exa:plow1,exa:plow11} are particularly meaningful in the surface setting. This is expressed in \Cref{prop:factor}. 

\begin{prop}[{\cite[Corollary II.12]{BEAUVILLE}}]\label{prop:factor}
Let $S,S'$ be two smooth irreducible surfaces  and let  $\phi\in\BC(S,S')$ be a birational map. Then, there exists a third smooth  surface $\widetilde{S}$ and a commutative diagram
\[
\begin{tikzcd}
    &\widetilde{S}\arrow[dl,"\varphi"']\arrow[dr,"\varphi'"]&\\
    S\arrow[rr,dashed,"\phi"]&&S',
\end{tikzcd}
\]
    such that the morphisms $\varphi$ and $\varphi'$ are compositions of blow-ups at one point and isomorphisms.
\end{prop}

In  \Cref{exa:2points,exa:blow3} we interpret the surface $\widetilde{S}$ of \Cref{prop:factor} as a blow-up of $S$ and $S'$ at the same time, thus endowing it with two  different \textit{blowing-down structures}. We will discuss this notion in \Cref{subsec:changeblowstruc}, see \Cref{def:geobasis}.

\begin{example}[Blow-up at two points]\label{exa:2points}
We construct now the blow-up of the projective plane at two points as the graph of a birational map given by a pair of projections.  Consider the rational map
\[
\begin{tikzcd}[row sep=tiny]
\BP^2\arrow[r,dashed,"\phi"]&\BP^1\times\BP^1\\    {[x_0:x_1:x_2]}\arrow[r,mapsto] &({[x_0:x_1],[x_1:x_2]}),
\end{tikzcd}
\]
given as a pair of projections.
It is birational   with inverse 
\[
\begin{tikzcd}[row sep=tiny]
\BP^1\times\BP^1\arrow[r,dashed,"\theta"]&\BP^2\\     ({[y_0:y_1]},{[z_0:z_1]})\arrow[r,mapsto]&
{[y_0z_0:y_1z_0:y_1z_1]} .
\end{tikzcd}
\]
The two maps have the following indeterminacy loci
    \[
    \ind(\phi)=\Set{\!e_0,e_2\!}\quad \mbox{ and }\quad\ind(\theta)=\Set{\!([1:0],[0:1])\!}\!.
    \]
    On the one hand, the map $\phi$ contracts the line $V(x_1)=\langle e_0,e_2\rangle$ onto the point $([1:0],[0:1])$ while the map $\theta$ contracts the two lines $L_0=\Set{[1:0]}\times\BP^1$ and $L_2=\BP^1\times \Set{[0:1]}$, passing through $([1:0],[0:1])$, onto $e_0$ and $e_2$ respectively.
    \Cref{fig:blow2pt} depicts the construction.
   \begin{figure}[ht!]
       \centering 
       \begin{tikzcd}[row sep = tiny]
          & \begin{tikzpicture}[scale=0.7]
       \draw (-1.5,0)--(0,-1)--(1.5,0)--(0,1)--cycle;
           \node at (0,-1.3) {\tiny $\Bl_{p,q}\BP^2\cong \Bl_r (\BP^1)^{\times2}$};
           \draw[thick , color=blue] (-1,0)--(1,0); 
  \draw[thick , color=red] (-0.8,-0.4) --(-0.6,0.4) ;
  \draw[thick , color=red] (0.8,-0.4) --(0.6,0.4) ;
       \end{tikzpicture}\arrow[rd]\arrow[ld]\\
       \begin{tikzpicture}[scale=0.7]
       \draw (-1.5,0)--(0,-1)--(1.5,0)--(0,1)--cycle;
           \node at (0,-0.7) {\small $\BP^2$};
           \draw[thick , color=blue] (-1,0)--(1,0);
  \fill[color=red] (0.7,0) circle (2pt);
  \fill[color=red] (-0.7,0) circle (2pt);
       \end{tikzpicture}\arrow[rr,dashed]&&
       \begin{tikzpicture}[scale=0.7]
       \draw (0.9,0.8)--(-0.1,-0.8);
       \draw (0.8,0.8)--(-0.2,-0.8);
       \draw (0.7,0.8)--(-0.3,-0.8);
       \draw (0.6,0.8)--(-0.4,-0.8);
       \draw[color = red, thick] (0.5,0.8)--(-0.5,-0.8);
       \draw (0.4,0.8)--(-0.6,-0.8);
       \draw (0.3,0.8)--(-0.7,-0.8);
       \draw (0.2,0.8)--(-0.8,-0.8);
       
       \draw (-0.9,0.8)--(0.1,-0.8);
       \draw (-0.8,0.8)--(0.2,-0.8);
       \draw (-0.7,0.8)--(0.3,-0.8);
       \draw (-0.6,0.8)--(0.4,-0.8);
       \draw[color = red, thick] (-0.5,0.8)--(0.5,-0.8);
       \draw (-0.4,0.8)--(0.6,-0.8);
       \draw (-0.3,0.8)--(0.7,-0.8);
       \draw (-0.2,0.8)--(0.8,-0.8);
  \fill[color= blue ] (0,0) circle (2pt); 
            \node at (1.3,0) {\small $\BP^1\times \BP^1$};
       \end{tikzpicture}
       \end{tikzcd}
       \caption{Pictorial description of the construction in \Cref{exa:2points}.}
       \label{fig:blow2pt}
   \end{figure}
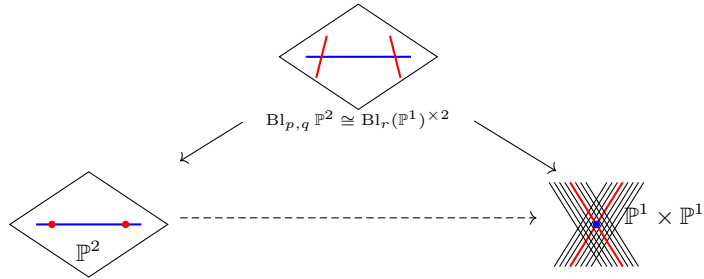
    As a consequence, we get the isomorphism 
    \[
    \Bl_{p,q}\BP^2\cong\Bl_r(\BP^1\times \BP^1),
    \]
    for any choice of distinct points $p,q\in\BP^2$ and of $r\in\BP^1\times \BP^1$.
\end{example}

\begin{example}\label{exa:CREMONA}
    In this example we present a birational involution of the projective space, namely the \textit{standard Cremona transformation}. This is the birational map $\crem_n\in\Bir(\BP^n) $ defined by 
\begin{equation}
        \begin{tikzcd}[row sep = tiny]
    \BP^n\arrow[r,dashed , "\crem_n"] &\BP^n\\
    {[x_0:\cdots:x_n]}\arrow[r,mapsto]&{\left[ \frac{1}{x_0}:\cdots:\frac{1}{x_n} \right]} &[-1cm] ={\left[ x_1\cdots x_n:\cdots:x_0\cdots x_{n-1} \right]} .
\end{tikzcd}
        \label{eq:cremM}
\end{equation} 
Its indeterminacy locus is
 
\[
\ind(\crem_n) =\underset{{0\le i<j\le n}}{\bigcup}V({x_i,x_j})\subset\BP^n.
\]
 Note that, the Cremona map  contracts the $i$-th coordinate hyperplane $V({x_i})$ 
to the $i$-th coordinate point $e_i$, for $i=0,\ldots,n$. Moreover, $\crem_n$ is an involution, i.e. $\crem_n^2\equiv \id_{\BP^n}$. Strictly speaking we have
\[
\crem_n^2([x_0:\cdots:x_n])=[x_0^2x_1\cdots x_n:\cdots:x_0\cdots x_{n-1}x_n^2]
\]
and, since we work in the projective setting, we can remove the common factor from the entries of $\crem_n^2$ to get the identity.
\end{example}

\begin{example}[Blow-up at three points]\label{exa:blow3}
    In this example we focus on the Cremona transformation $\crem_2\in\Bir(\BP^2)$ of the projective plane. Consider two copies of the projective plane with homogeneous coordinates $x_0,x_1,x_2$ and  $x_0',x_1',x_2'$ respectively. Similarly, denote by $e_i,L_i=V(x_i)$ and $e_i',L'_i=V(x_i')$, for $i=0,1,2$, the coordinate points and lines on the first and the second copy of $\BP^2$ respectively. Then, the resolution of indeterminacies of $\crem_2$ is the commutative diagram
\begin{equation}
        \label{eq:C2}
        \begin{tikzcd}[row sep=tiny]
        &B\arrow[dr,"\varepsilon'"]\arrow[dl,"\varepsilon"']\\
                \BP^2_{{[}x_0:x_1:x_2{]}} \arrow[dashed]{rr}{\crem_2} & &\BP^2_{{[}x_0':x_1':x_2'{]}} \\
                {[}x_0:x_1:x_2{]} \arrow[mapsto]{rr} && \left[\frac{1}{x_0}:\frac{1}{x_1}:\frac{1}{x_2}\right]\!,
        \end{tikzcd}
\end{equation}
where   $B={\graph(\crem_2)}$ denotes  the graph of $ \crem_2$. The transformation $\crem_2 $ is an involution and it contracts the coordinate line $L_i$ to the coordinate point $e_i'$, for $i=0,1,2$. On the other hand, the inverse $\crem_2^{-1}$ contracts the coordinate line $L_i'$ to the coordinate point $e_i$, for $i=0,1,2$. We stress that $\crem_2$ is an involution in the sense that, if we identify $x_i=x_i'$, for $i=0,1,2$, we get $\crem_2^2=\id_{\BP^2}$. 

Consider the set $\Set{\!\widehat L_0,\widehat L_1,\widehat L_2,\widehat L_0',\widehat L_1',\widehat L_2'\!}$ consisting of the strict transforms, via $\varepsilon$ and $\varepsilon'$, of the lines $L_i,L_i'$, for $i=0,1,2$.  Then, the morphism $\varepsilon$  contracts only the triple $\Set{\!\widehat L_0,\widehat L_1,\widehat L_2\!}$, while $\varepsilon'$ contracts $\Set{\!\widehat L_0',\widehat L_1',\widehat L_2'\!}$ so realising $B$ as the blow-up of $\BP^2$ at 3 points in two different ways, see \Cref{Cremonaresolution2} for a graphical description of the construction. As anticipated, this is an example of two different  {blowing-down structures}, see \Cref{def:geobasis}.

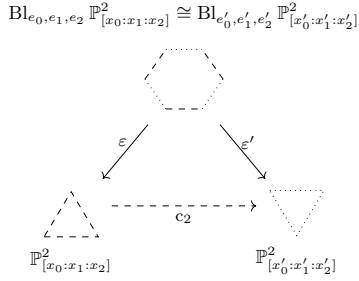
\begin{figure}[ht!]
        \centering
        \scalebox{0.8}{\begin{tikzpicture}[scale=1.5]
                \node at (-1.25,-1) {$\underset{\mbox{\small $\BP^2_{[x_0:x_1:x_2]}$}}{\begin{matrix}
\begin{tikzpicture}[scale=1.5]
        \draw [dashed] (-0.3,-0.5)--(0.3,-0.5)--(0,0)--(-0.3,-0.5);
\end{tikzpicture} 
            \end{matrix}}$};
    \node at (1.25,-1) {$\underset{\mbox{\small $\BP^2_{[x_0':x_1':x_2']}$}}{\begin{matrix}
                            \begin{tikzpicture}[scale=1.5]
                                    \draw [dotted] (-0.3,0.5)--(0.3,0.5)--(0,0)--(-0.3,0.5);
                            \end{tikzpicture} 
            \end{matrix}}$};
    \node at (0,1) {\begin{tikzpicture}[scale=1.5]
    \draw [dashed] (-0.2,-0.33)--(0.2,-0.33)(0.433,0)--(0.2,0.33)(-0.433,0)--(-0.2,0.33);
                    \node at (0,0.7) {\small$\Bl_{e_0,e_1,e_2}\BP^2_{[x_0:x_1:x_2]}\cong\Bl_{e_0',e_1',e_2'}\BP^2_{[x_0':x_1':x_2']}$};
                    \draw [dotted] (-0.2,0.33)--(0.2,0.33)(0.433,0)--(0.2,-0.33)(-0.433,0)--(-0.2,-0.33);
    \end{tikzpicture}};
    \draw[dashed,->] (-0.8,-0.7)--(0.8,-0.7);
    \draw[->] (0.4,0.2)--(0.9,-0.4);
    \node[right] at (0.55,0) {\footnotesize $\varepsilon'$};
    
    \draw[->] (-0.4,0.2)--(-0.9,-0.4);
    \node[left] at (-0.55,0) {\footnotesize $\varepsilon$};
    \node[below] at (0,-0.7) {\footnotesize $\crem_2$};
        \end{tikzpicture}}
        \caption{The resolution of the indeterminacies of the standard 
        Cremona transformation in dimension 2.}
        \label{Cremonaresolution2}
\end{figure}
 
\end{example}

\begin{exercise} 
Consider the standard Cremona transformation $\crem_2\in\Bir(\BP^2)$ presented in \Cref{exa:blow3}. 
\begin{itemize}
    \item Realise the blow-up $\Bl_{e_0,e_1,e_2}\BP^2$ as a closed subset of  $\BP^2\times\BP^2$, see \Cref{Cremonaresolution2}.
    \item Realise the blow-up $\Bl_{e_0,e_1,e_2}\BP^2$ as a closed subset of   $\BP^2\times\BP^1\times\BP^1\times\BP^1$.\\
    \textbf{Hint:} Imitate the strategy in \Cref{exa:2points} by considering the triple of projections $\phi=(\pi_{e_0},\pi_{e_1},\pi_{e_2})$.
    \item Show that the restriction of the canonical projection
    \[ \pi:\BP^2\times\BP^1\times\BP^1\times\BP^1\to  \BP^1\times\BP^1\times\BP^1
    \]
    to $\Bl_{e_0,e_1,e_2}\BP^2$ is an embedding.\\
    \textbf{Hint:} Adopt the strategy suggested in the previous item and show that the image of the restriction of $\pi$ to $\Bl_{e_0,e_1,e_2}\BP^2$ is defined by a trilinear equation. Then, show that $\pi|_{\Bl_{e_0,e_1,e_2}\BP^2}$ establishes an embedding of the blow-up in $\BP^1\times\BP^1\times\BP^1$.
    \item In the previous items, find the equations of the lines $ \widehat L_0,\widehat L_1,\widehat L_2,\widehat L_0',\widehat L_1',\widehat L_2' $.
\end{itemize}

\end{exercise}

As proved by \textit{Max Noether} and \textit{Guido Castelnuovo}, the standard Cremona transformation plays a special role in dimension two. 

\begin{theorem}[Noether-Castelnuovo, \cite{Noether1870,Castelnuovo1901}]
   The group $\Bir(\BP^2)$ is generated by $\BP\GL(3,\BC)$ and the standard Cremona transformation $\crem_2$.
\end{theorem}

It is worth mentioning that in dimension higher than two, the problem of finding generators and relations of $\Bir(\BP^n)$ is highly non-trivial and still open already for $\BP^3$.

\begin{exercise}\label{exe:CREMONA}
Let $T=\underset{0\le i<j\le 3}{\bigcup}V({x_i,x_j})\subset\BP^3$ be  the \textit{coordinate tetrahedron}, i.e. the union of the coordinate lines of $\BP^3$. Denote by $X$ the blow-up $X=\Bl_T\BP^3$.
\begin{itemize}
    \item Realise $X$ as a closed subset of $\BP^3\times\BP^3$.
    \item Show that $X$ has 12 singular points.
    \item Find all the irreducible components of the exceptional locus $E_T$. \\\textbf{Hint:} There are 10 of them.
\end{itemize}
    
\end{exercise}

\subsection{Line bundles and divisors}\label{subsec:linebundledivisors} 
In this subsection we present some basic facts from the theory of divisors and line bundles on smooth varieties.

Let $X$ be a smooth quasi-projective variety. Denote by $\Pic(X)$ the \textit{Picard group} of $X$, i.e. the  group of isomorphism classes of line bundles on $X$, 
\[
\Pic(X)=\Set{\!\mbox{line bundles}\!}/\mbox{isomorphisms}, \]
with operation given by tensor product $\otimes$, identity element given by the class of the structure sheaf $\OO_X$, i.e. the class of the trivial line bundle, and inverse given by the dual $^\vee$.

Recall that a \textit{prime divisor} $V\subset X$ is an irreducible closed subvariety of codimension 1. Then, the \textit{group of divisors} $\Div(X)$ on $X$ is the free abelian group generated by the prime divisors of $X$, i.e. 
\[ \Div(X)=\left\langle\Set{ \!V\subset X |  V \mbox{ is a prime divisor}\!  }\right\rangle_{\BZ}.
\]
By definition, the elements of $\Div(X)$ are formal sums, with integral coefficients, of prime divisors of $X$. In this setting, if $D=\sum_{i=1}^s\alpha_iV_i$, we say that the \textit{support} of $D$ is the closed hypersurface $\supp(D)=\bigcup_{i=1}^s V_i\subset X$. The \textit{cone of effective divisors} is
\[
\Eff(X)=\left\langle\Set{ \!V\subset X |  V \mbox{ is a prime divisor}  \!}\right\rangle_{\BZ_{> 0}},
\]
i.e. the set of formal sums of prime divisors with positive coefficients.  The notion of effective divisor allows the group $\Div(X)$ to be endowed with a   \textit{partial order}. Given two divisors $D,D'\in\Div(X)$, we say that 
\begin{equation}\label{eq:partialordereffectivedivisors}
D\prec D'\mbox{ if }D'-D \in \Eff(X).
\end{equation} 

Recall that, by the \textit{principal ideal theorem} \cite[Theorem 10.1]{EISENBUD}, a divisor $D=\sum_{i=1}^t\alpha_i V_i\in\Div(X)$ on a smooth variety is uniquely determined by its \textit{local data}. That is a set of pairs $\Set{\left(\CU^{(k)},\prod_{i=1}^t   (f_i^{(k)})^{\alpha_i}\right)}_{k=1}^N$, where $\Set{\CU^{(k)}}_{k=1}^N$ is an open cover of $X$ and, for $k=1,\ldots,N$, the element $f^{(k)}_i\in\BC\left[\CU^{(k)}\right]$ is a local equation for $V_i$ on $\CU^{(k)}$, for $i=1,\ldots,t$, i.e. such that  $V_i\cap \CU^{(k)}=V(f_i^{(k)})$. 

We now briefly present the relationship between line bundles and divisors. For the sake of brevity, we will keep the discussion to a minimum, see \cite[Section 1.1]{GRIHARR} for more details.

Given a holomorphic section $s\in H^0(X,L)$ of a line bundle $L$ on $X$, its \textit{divisor of zeroes} is defined as follows. Let $V\subset X$ be a prime divisor and let $f\in \BC[\CU]$ be a local equation for $V$ on some open subset $\CU\subset X$ with $\CU\cap V\not=\emptyset$, i.e. $I(V\cap \CU) =(f)\subsetneq \BC[\CU]$. Then, the \textit{order of $s$ along $V$} is
\[
\ord_{V}(s)=\max\Set{k\ge 0 | \frac{s|_{\CU}}{f^k} \mbox{ is holomorphic on $\CU$}}\in\BZ_{\ge 0}.
\]
Note that this number is independent of the choice of the open $\CU\subset X$ and on the local equation $f\in\BC[\CU]$. In this setting, the divisor associated to $s$ is
\begin{equation}\label{eq:divass}
    \divi(s)=\sum_{\mbox{\tiny $V$ prime divisor}} \ord_V(s)\cdot V.
\end{equation}

It is worth noting that the sum in \eqref{eq:divass} is finite since a holomorphic function has positive order along only finitely many prime divisors.  

On the other hand, given an effective divisor $D\in \Eff(X)$, one can cook up\footnote{The association boils down to the fact that the local data of $D$ naturally induce local data and a section of a line bundle.} a line bundle $\OO_X(D)$ and a section $s\in H^0(X,\OO_X(D))$ such that $\divi(s)=D$. Moreover, this association can be extended by $\BZ$-linearity to the whole $\Div(X)$ via the rules
\begin{equation}\label{eq:picDiv}
    \begin{cases}
    \OO_X(D+D')=\OO_X(D)\otimes \OO_X(D'),\\
    \OO_X(-D)=\OO_X(D)^\vee.
\end{cases}
\end{equation}

\begin{definition}
    Two divisors $D,D'\in\Div(X)$ are \textit{linearly equivalent} if the associated line bundles are isomorphic. In symbols, we write $D\sim D'$ if $\OO_X(D)\cong \OO_X(D')$. 
\end{definition}
In other words, two divisors are linearly equivalent if they correspond to the same element in $\Pic(X)$. This gives an isomorphism
\begin{equation}\label{eq:Pic<->Div}
    \Pic(X)\cong\Div(X)/\sim.
\end{equation}
In what follows, we shall implicitly  make intensive use of the identification in \eqref{eq:Pic<->Div}.

\begin{example}\label{exa:PicPn} Recall that $\Pic(\BP^n)\cong\BZ$ is generated by the class of a line bundle $\OO_{\BP^n}(1)$ whose local sections on the coordinate charts $\CU_i$, for $i=0,\ldots,n$, are ratios of the form $\ell/x_i$ for some homogeneous form $\ell\in R$ of degree one. 

Let $H\in\Div(\BP^n)$ be a hyperplane. Then $H$ corresponds, via the identification \eqref{eq:Pic<->Div}, to the line bundle  $\OO_{\BP^n}(1)$. Moreover every hypersurface $X\subset \BP^n$ of degree $d$ is linearly equivalent to $dH$ and it  corresponds to the line bundle $\OO_{\BP^n}(1)^{\otimes d}=\OO_{\BP^n}(d)$, i.e. $\OO_{\BP^n}(X)\cong\OO_{\BP^n}(d H)\cong \OO_{\BP^n}(d)$, see \eqref{eq:picDiv}. 
\end{example}

\begin{notation} \label{not:linearequiv}
Since we consider constructions  independent of linear equivalence, with a slight abuse of notation we refer to the elements of $\Pic(X)$ as line bundles. If $L\in\Pic(X)$ is a line bundle on $X$ and $D\in\Div(X)$ is a divisor, we denote by $L(D)$ the line bundle $L\otimes\OO_X(D)$. We also abuse notation by denoting  a divisor and its class in $\Div(X)/\sim$ by the same symbol.
\end{notation}

A key tool to deal with line bundles and divisors is the pull-back.

\begin{definition}
Let $\varphi:X\to X'$ be a surjective morphism of smooth projective varieties. And let $D\in \Div(X')$ be a divisor with local data  $\Set{\left(\CU^{(k)},\prod_{i=1}^t  g_i^{(k)}\right)}_{k=1}^N$, for some open cover $\Set{\CU^{(k)}}_{k=1}^N$ of $X'$. Then, the \textit{pull-back divisor} $\varphi^*D\in\Div X$ is the divisor having local data $\left\{\left(\varphi^{-1}\left(\CU^{(k)}\right),\prod_{i=1}^t  g_i^{(k)}\circ{\varphi|}_{\varphi^{-1}\left(\CU^{(k)}\right)}\right)\right\}_{k=1}^N$. 
\end{definition} 

\begin{remark}\label{rmk:genPull}
 Given any morphism $\varphi:X\to X'$ of smooth varieties and a line bundle $L\in\Pic(X')$, the \textit{pull-back} $\varphi^*L$ is always a well-defined element of $\Pic(X)$. This observation, together with the isomorphism \eqref{eq:Pic<->Div} allows one to extend the notion of pull-back of divisors to all morphisms.

Given a closed subset $F\subset S$ of codimension $\codim F\ge 2$,  we clearly have an isomorphism $\Pic(X)\cong\Pic(X\setminus F)$ consisting in taking closures in $X$ in one direction and intersections with $X\setminus F$ in the other. As a consequence, for a rational map $\phi:X\dashrightarrow X'$ of smooth varieties, it  always makes sense to consider the \textit{pull-back} $\phi^*L\in\Pic(X)$ of a line bundle $L\in\Pic(X') $ or similarly of a divisor.
\end{remark} 

The relation between line bundles and divisors reflects on the well-known relation between rational maps with target projective spaces and the so-called \textit{linear systems without fixed part}, see \cite[Sections 1.4.1]{GRIHARR} for more details. We report this relation explicitly here.

\begin{definition}
    Let $X$ be a smooth quasi-projective variety and let $D\in\Eff(X)$ be an effective divisor. Then, the \textit{complete linear system associated to} $D$ is 
    \[
    |D|=\Set{D'\in\Eff(X)| D'\sim D}\!.
    \]
\end{definition}

\begin{remark}
 Given an effective divisor $D\in \Eff(X)$,   we can canonically identify the complete linear system $|D|$ with the projective space $\BP H^0(X,\OO_X(D))$ via the association between line bundles and divisors explained at the beginning of this subsection. 

 In this setting, a \textit{linear system} $\mathfrak{P}$ is any linear subspace $\mathfrak{P}\subset |D|$ of a complete linear system. A linear system is a \textit{pencil}, \textit{net} or \textit{web} according to its dimension if this is $1$, $2$, or $3$, respectively.
\end{remark}
\begin{definition}
    Let $\mathfrak{P}$ be a linear system on a smooth variety $X$. The \textit{fixed part} of $\mathfrak{P}$ is the greatest effective divisor $\Fix(\mathfrak{P})\in\Eff(X)$, with respect to $\prec$, such that $D-\Fix(\mathfrak{P})\in\Eff(X)$ for all $D\in \mathfrak{P}$.     A point $p\in X$ is a \textit{base point} for $\mathfrak{P}$ if $p\in\supp(D)$, for all $D\in \mathfrak{P} $. The \textit{base locus} of $\mathfrak{P}$ is the set of its base points, i.e.
    \[
    \Bs(\mathfrak{P})= \bigcap_{D\in \mathfrak{P}} \supp D .
    \]
\end{definition}
  
\begin{example}\label{pencillines}
    Let $L\subset \BP^2$ be a line. Then, we have $|L|={\BP^2}^\vee$. Now, fix a point $p\in \BP^2$. Then, the pencil $\mathfrak{P}_p$ of lines through $p$ is a (non-complete) linear system $\mathfrak{P}_p\subset |L|$, with $\Bs(\mathfrak{P}_p)=\Set{p}$.
\end{example}

As anticipated, linear systems having no fixed part are associated to rational maps to projective spaces. The idea behind this correspondence is that, if $p\in X\setminus\Bs(\mathfrak{P})$, then the locus $\Set{D\in \mathfrak{P}|p\in\supp D }$ is a hyperplane in $\mathfrak{P}$, i.e. a point of $\mathfrak{P}^\vee$. This defines a rational map with value in $\mathfrak{P}^\vee$ whose domain is the complement $X\setminus \Bs(\mathfrak{P})$ of the base locus.  For this reason, we only consider linear systems satisfying the necessary condition  $\Fix(\mathfrak{P})=\varnothing$, see \Cref{prop:codim2ind}. The next proposition shows that this condition is also sufficient, see \cite[Section II.6]{BEAUVILLE} for the surface case and \cite[Theorem II.7.1]{Hartshorne77} for a more general statement. 

\begin{prop} Let $X$ be a smooth surface. Then, there is a bijection 
\begin{center}
    \begin{tikzpicture}
    \draw[<->] (3.2,1)--(4.7,1);
    \draw[<-|] (3.2,0)--(4.7,0);
    \draw[|->] (3.2,-1)--(4.7,-1);
        \node at (7,1) { $\Set{\mathfrak{P}\mbox{ lin. sys.}| \begin{matrix}
        \dim \mathfrak{P}=n,\\ \Fix(\mathfrak{P})=\varnothing
    \end{matrix}} $};
        \node at (0,1) { $\Set{\phi\in\BC(X,\BP^n)|\phi(X)\not\subset H \ \forall[H]\in{\BP^n}^\vee }$};
        \node at (7,0) {$\mathfrak{P}$};
        \node at (0,0) {\footnotesize
    ${\left(\begin{tikzcd}[row sep=tiny]
        X\arrow[r,dashed,"\phi_\mathfrak{P}"]&\mathfrak{P}^\vee\\
        p\arrow[r,mapsto]&\Set{D\in \mathfrak{P}|p\in\supp D}
    \end{tikzcd}\right)}$}; 
        \node at (0,-1) {$\phi$};
        \node at (7,-1) {$|\phi^*\OO_{\BP^n}(1)|.$};
    \end{tikzpicture}
\end{center} 
\end{prop}
\begin{definition}
    A linear system $\mathfrak{P}$ on $X$ is \textit{very ample} if the associated rational map $\phi_{\mathfrak{P}}:X\dashrightarrow \mathfrak{P}^\vee$ is an embedding. It is \textit{ample} if there exists $k\in\BZ_{\ge 1}$, such that $k\mathfrak{P}$ is very ample.
\end{definition}

\begin{example}
 Let $L\subset \BP^2$ be a line. Then, the projection $\pi_{e_2}$ presented in \Cref{exa:1.2} corresponds to the pencil
\[
\mathfrak{P}_{e_2}=\Set{D\in \lvert L\rvert \mbox{ such that } e_2\in D }\!,
\]
see \Cref{pencillines}.
\end{example}

\subsection{Intersection pairing on smooth surfaces} \label{subsec:intersection} 
In this section we introduce the \textit{intersection pairing} on smooth surfaces, mostly following the presentations in \cite{Hartshorne77,BEAUVILLE}. This formula is a generalisation of the  well-known \textit{Bezout formula} for the intersection multiplicity of two distinct irreducible plane curves  to the more general intersections of divisors on any smooth projective surface. In the rest of the subsection we explain some facts about algebraic surfaces that we shall need in the  notes.

\begin{convention}
From now on, we will only consider irreducible surfaces. Therefore,  in order to ease the notation, we will implicitly assume all the surfaces to be irreducible without saying it explicitly.
\end{convention}

\begin{theorem}[Intersection pairing]\label{thm:intprod}
  Let $S$ be a smooth projective surface.  Then, there is a unique symmetric bilinear form
\[
\begin{tikzcd}[row sep =tiny]
    \Div(S)\times \Div(S)\arrow[r] &\BZ\\
    (C,D)\arrow[r,mapsto] & C.D,
\end{tikzcd}
\]
named {the} \textit{intersection pairing}, such that:
\begin{itemize}
    \item if $C,D\subset S$ are smooth curves meeting transversely then $C.D$ equals the cardinality of the set $ C\cap D $;
    \item the integer $C.D$ only depends on linear equivalence classes, i.e. if $C,D,D'\in\Div(S)$ are divisors then $C.D=C.D'$ whenever $D\sim D'$. 
\end{itemize} 
\end{theorem}

The push-forward of divisors is often involved in surface theory.

\begin{definition}
Let $\varphi:S\to S'$ be a degree $d$ (surjective) morphism of smooth projective surfaces. Let $V\subset S$ be a prime divisor. Then, the \textit{push-forward divisor} $\varphi_*V$ is
\begin{equation}\label{eq:push}
    \varphi_*V=\begin{cases}
    0 &\mbox{ if $\varphi(V)$ is a point,} \\
   \deg(\varphi|_V)\cdot \varphi(V) &\mbox{ otherwise}.
\end{cases}
\end{equation}
The \textit{push-forward} $\varphi_*:\Div(S)\to\Div(S')$ is the $\BZ$-linear extension of the association \eqref{eq:push} to the group $\Div(S)$.
\end{definition}

The interplay between pull-back and push-forward via a surjective degree $d$ morphism $\varphi:S\to S'$ of smooth projective surfaces is expressed by the equality
\[
\varphi_*\varphi^*D=d\cdot D,
\]
which holds for every $D\in\Div(S')$.

Consider a smooth quasi-projective surface $S$. Recall that, if $C=V(f)\subset S$ is a (not necessarily  reduced) curve and $p\in S $ is a point, the \textit{multiplicity of $C$ at $p $} is the integer
\[
\mult_p C = \min \Set{k\in\BZ_{\ge 0} | f\in\Fm^k_p\subset \OO_{S,p} }\!,
\] 
where $\OO_{S,p} $ is the \textit{stalk} of the structure sheaf $\OO_S$ at $p$, i.e. the local ring of germs of rational functions on $S$ which are holomorphic at $p$, and $\Fm_p\subset \OO_{S,p}$ its \textit{maximal ideal}, i.e. the ideal consisting of germs in $\OO_{S,p}$ vanishing at $p$.

Clearly, the notion of multiplicity is local, and it can naturally be extended to any curve on  a smooth surface. The following lemma, whose proof is a direct consequence of the defining properties of the intersection pairing, establishes a relation between the pull-back and the strict transform of a curve via blow-ups.
\begin{lemma}\label{lemma:strictmult}
Let $\varepsilon:\Bl_p S\to S$ be the blow-up of a smooth surface $S$ centred at a point $p\in S$ and let $E\subset \Bl_p S$ be the exceptional curve. Let also $C\subset S$ be an irreducible curve and let $\widehat{ C}\subset \Bl_p S$ be its strict transform.  Then, we have
\[
\varepsilon^* C = \widehat{ C} + (\mult_p C)\cdot E.
\]
\end{lemma}

We present now more direct consequences of the defining properties of the intersection pairing.
\begin{lemma}\label{lemma:eltransf}
Let $\varepsilon:\Bl_p S\rightarrow S$ be the blow-up of a smooth projective surface $S$  centred at a point $p\in S$. Let also $E=\varepsilon^{-1}(p)$ be the (rational) exceptional curve and $D,D'\in\Div(S)$ two divisors. Then, we have
\begin{itemize}
    \item $\varepsilon^*D.E=0$,
    \item $\varepsilon^*D.\varepsilon^*D'=D.D'$,
    \item $E^2=-1$.
\end{itemize}
Moreover, the association
\[
\begin{tikzcd}[row sep = tiny ]
\Pic(S)\oplus\BZ\arrow[r]&\Pic(\Bl_pS)\\
    (D,n)\arrow[r,mapsto]& \varepsilon^*D+nE,
\end{tikzcd}
\]
is an isomorphism. 
\end{lemma}
\begin{exercise}
    Prove \Cref{lemma:strictmult,lemma:eltransf}.\\
    \textbf{Hint}: In order to prove \Cref{lemma:strictmult}, write the map $\varepsilon$ in coordinates, compute the equation of the pull-back of $C$ and the multiplicity of the factor corresponding to $E$. For \Cref{lemma:eltransf}, use divisors linearly equivalent to $D$ and $D'$. For instance, you can ask that their supports do not contain $p$.
\end{exercise} 

\begin{remark} 
Let $\varepsilon:\Bl_p S\rightarrow S$ be the blow-up of a smooth projective surface $S$  centred at a point $p\in S$. Let also $p\in C\subset S$ be an irreducible (closed) curve passing through $p$, and smooth at $p$. Then, \Cref{lemma:eltransf} implies $\widehat{C}^2=C^2-1$. 
\end{remark}

\begin{definition} \label{def:exceptionalcurve}
    Let $S$ be a smooth quasi-projective surface and let $C\subset S$ be an irreducible curve. Then, a prime divisor $C$ is an \textit{exceptional curve} if there exists another smooth surface $S'$ and an isomorphism $\Bl_p S'\cong S$ identifying $C$ with the exceptional curve of the blow-up $\Bl_p S'$.
\end{definition} 

    As a consequence of \Cref{exa:plow1} and \Cref{lemma:eltransf}, when $S$ is projective, an exceptional curve $C\subset S$ satisfies $C\cong\BP^1$ and $C^2=-1$. The following celebrated result by \textit{Guido Castelnuovo} characterises exceptional curves, showing that these conditions are not just necessary but also sufficient.

\begin{theorem}[Castelnuovo's contractibility criterion {\cite[Theorem II.17]{BEAUVILLE}}]\label{thm:Castelnuovo}
Let $S$ be a smooth projective surface and let $C\subset S$ be an irreducible curve such that $C\cong\BP^1$ and $C^2=-1$, then $C $ is an exceptional curve.
\end{theorem}

\begin{remark}
The projectivity assumption on $S$ in \Cref{thm:Castelnuovo} is actually not necessary. Indeed, having self-intersection  $-1$ is a property of a curve $C\subset S$ that can be phrased in terms of the normal bundle of $C$ in $S$ and it can be checked in an analytic neighbourhood of $C$, see \cite[Subsection 1.4.2]{GRIHARR}. As a consequence, \Cref{thm:Castelnuovo} can be in principle stated for quasi-projective surfaces  and proved on any smooth compactification.
\end{remark}

\subsection{The canonical bundle}\label{subsec:canonical} We introduce now the notion of canonical bundle and show how to relate the canonical bundles of smooth varieties via morphisms.

\begin{definition}
Let $X$ be a smooth quasi-projective variety of dimension $n$. The \textit{canonical bundle}  of $X$ is the line bundle $\omega_X$ whose local sections on an open subset $\CU\subset X$ consist of holomorphic $n$-forms on $\CU$. A \textit{canonical divisor} on $X$ is a divisor $K_X\in\Div(X)$ corresponding to $\omega_X$ via the association \eqref{eq:Pic<->Div}. An \textit{anti-canonical divisor} on $X$ is a divisor $-K_X$ corresponding to the dual bundle $\omega_X^{\vee}$ of $\omega_X$.
\end{definition}
Note that we will often refer to $K_X$ as \textit{``the canonical divisor"}. This convention is common in the literature. We stress that the abuse of the notation is justified by the fact that the constructions we consider are independent of linear equivalence.
\begin{example} 
The canonical bundle of the projective space is $\omega_{\BP^n}=\OO_{\BP^n}(-n-1)$, see \Cref{exa:PicPn}, and the anti-canonical bundle is $\omega_{\BP^n}^{\vee}=\OO_{\BP^n}(n+1)$.
\end{example}

An important feature of the canonical bundle is that it provides the so-called \textit{Serre duality}, which is a powerful tool for computing the cohomology of line bundles or vector bundles in general, see \cite[Section III.7]{Hartshorne77}.
\begin{theorem}[Serre duality]\label{thm:SErredual}
    Let $X$ be a smooth projective variety and let $L\in\Pic(X)$ be a line bundle. Then, we have $H^n(X,\omega_X)\cong\BC$. Moreover, for all $i=0,\ldots,n$, the cup product pairing
    \[
\begin{tikzcd}
        H^i(X,L)\times H^{n-i}(X,\omega_X\otimes L^{-1})\arrow[r]&H^n(X,\omega_X)\cong\BC,
\end{tikzcd}
    \]
    defines a perfect pairing. 
\end{theorem}

\begin{remark}
    The takeaway  from \Cref{thm:SErredual} is that there is an isomorphism $H^i(X,L)\cong H^{n-i}(X,\omega_X\otimes L^{-1})^\vee$, which translates in an equality between the dimensions of the two complex vector spaces.
\end{remark}

Let $X$ be a smooth projective variety of dimension $n$. Recall that the \textit{arithmetic genus} $p_a(X)$ and the \textit{geometric genus} $p_g(X)$  are the integers
\[
p_a=(-1)^n(\chi(\OO_X)-1)\mbox{ and }p_g(X)=\dim _{\BC} H^0(X,\omega_X),
\]
where $\chi(L)=\sum_{i=0}^{\dim X} (-1)^i\dim_{\BC} H^i(X,L)$ denotes the
\textit{Euler characteristic}  of a line bundle $L\in\Pic(X)$.

We   discuss now the \textit{adjunction formula} for morphisms between smooth varieties. This formula establishes a clear relation between the canonical bundle of a smooth $n$-fold and the canonical bundle of a smooth prime divisor, see \cite[Subsection 1.1.3]{GRIHARR}.
\begin{theorem}[Adjunction formula] \label{thm:adjunctionformula}
Let $X$ be a smooth quasi-projective variety. Let $Y\subset X$ be a closed smooth hypersurface. Then, we have
\[
\omega_Y=\omega_X(Y)|_Y.
\]
    In terms of divisors this can be written as $K_Y=(K_X+Y)|_Y$.
\end{theorem}

In dimension two, {the} adjunction formula is equivalent to the \textit{genus formula}, see \cite[Section I.15]{BEAUVILLE}.  That is
\begin{equation} \label{eq:genusformula}
    p_a(C)=1+\frac{1}{2}(C^2+C.K_S),
\end{equation}
for any smooth surface $S$ and any closed irreducible curve $C\subset S$. Notice that we do not need to assume $C$ to be smooth for the genus formula.  
\begin{example}
Let $S$ be a surface with trivial canonical bundle, i.e. $\omega_S\cong \OO_S$. Assume\footnote{This is a non-trivial assumption. For example \textit{abelian surfaces} have trivial canonical bundle and do not contain any rational curve.} that there exists a curve $C\subset S$ such that $C\cong\BP^1$. Then, we have $C^2=-2$. Indeed, the genus formula gives
\[
0=1+\frac{C^2}{2}.
\]
Infinitely many examples of surface with trivial canonical bundle are provided by smooth quartic hypersurfaces of $\BP^3$. Indeed, by the adjunction formula, we have
\[
\omega_S=\omega_{\BP^3}(S)|_S=(\OO_{\BP^3}(-4)\otimes\OO_{\BP^3}(4))|_S=\OO_S.
\]
Finally, to give an explicit example, the so-called \textit{quartic Fermat hypersurface}  $V(x_0^4-x_1^4+x_2^4-x_3^4)$ is smooth and it contains the line $V(x_0-x_1,x_2-x_3)$.
\end{example}

\begin{remark}
    When $C$ is a smooth curve we have $p_a(C)=p_g(C)=g(C)$, where $g(C)$ is the \textit{topological genus} of $C$. On the other hand, thanks to Serre duality, when $S$ is a smooth surface the difference
    \[
    q(S)=p_a(S)-p_g(S)=\dim_{\BC}H^1(S,\OO_S),
    \]
    plays an important role. In fact, knowing the integers $q(S)$, $p_a(S)$ and $p_g(S)$ implies knowing $\chi(\OO_S)$, but this is a finer information.    The  number $q(S)$ is called the \textit{irregularity} of $S$. Another infinite family of invariants of $S$ is given by its \textit{plurigenera} $P_k(S)$, for $k\ge1$. These are the integers defined by
    \[
    P_k(S)= \dim _{\BC} H^0(S,\omega_S^{\otimes k})\in\BZ.
    \]
\end{remark}
Note that $P_1(S)=p_g(S)$. The following proposition highlights the importance of the numerical invariants $ q(S)$ and $P_k(S)$, for $k\ge1$, in the birational classification of surfaces.
\begin{prop}[{\cite[Proposition III.20]{BEAUVILLE}}]\label{prop:brinva}
    The integers $q(S)$ and $P_k(S)$, for $k\ge1$, are birational invariants of $S$.
\end{prop}

 We present now the \textit{blow-up formula} relating the canonical bundle of the blow-up of a smooth surface $S$ at a point $p$ and the pull-back of the canonical bundle of $S$.
\begin{theorem}[Blow-up formula] \label{th:blow-upformula}
    Let $\varepsilon:\Bl_p S\rightarrow S$ be the blow-up of a smooth quasi-projective surface $S$  centred at a point $p\in S$. Then, we have
    \[
    \omega_{\Bl_p S}=\varepsilon^*\omega_S(E).
    \]
    In terms of divisors, this can be written as $K_{\Bl_p S}=\varepsilon^*K_S+E$.
\end{theorem}

We conclude this subsection with a celebrated result by \textit{Max Noether} relating the topological and complex structures of a smooth projective surface $S$.
\begin{theorem}[Noether's formula, {\cite[Section I.14]{BEAUVILLE}}]
\label{thm:noetherformula}
Let $S$ be a smooth projective surface. Then, we have
\[
\chi(\OO_S)=\frac{1}{12}(K_S^2+\chi_{\topo}(S)),
\]
where $\chi_{\topo}(S)$ denotes the topological Euler characteristic of $S$.
\end{theorem}
\subsection{Rational surfaces}\label{subset:ratsurf} 

As anticipated in \Cref{subsec:ratmap}, usually algebraic geometers classify quasi-projective varieties up to birational equivalence. In dimension two, there is a well-defined notion of minimality which makes the classification more transparent.

\begin{definition}
Given a smooth projective  surface $S$, we denote by  $B(S)$ the collection of isomorphism classes of smooth projective surfaces birational to $S$.  The surface $S$ is \textit{minimal} if every birational morphism $S\to S'$ to another smooth surface $S'$ is an isomorphism.
\end{definition}

\begin{remark}
A smooth projective surface $S$ is minimal if and only if it does not contain any exceptional curve. Moreover, every surface dominates a minimal surface. Indeed, if $S$ contains an exceptional curve, this can be contracted via a dominant morphism  and this operation drops the rank of $\Pic(S)$ by one, see \Cref{lemma:eltransf}. By iterating this process, one eventually ends up with a minimal surface.
\end{remark}

\begin{definition}
Let $C$ be a smooth irreducible curve. A smooth surface $S$ is \textit{ruled over $C$} if it is birational to $C\times \BP^1 $. It is \textit{rational}\footnote{Note that this definition of rational surface is consistent with \Cref{def:birati}.} if $C\cong\BP^1$. The surface $S$ is  \textit{geometrically ruled over $C$} if it admits a morphism $\begin{tikzcd}
    S\arrow[r,"\pi_S"]&C
\end{tikzcd}$ with all fibres isomorphic to $\BP^1$.
\end{definition}

As a consequence of the Noether--Enriques Theorem \cite[Theorem II.4]{BEAUVILLE}, every surface $S$ geometrically ruled over a smooth  curve $C$ is a Zariski locally trivial $\BP^1$-fibration over $C$. In other words, there exist a Zariski open cover $\mathscr U_C=\Set{\mathcal C_i}_{i=1}^s$ of $C$ and isomorphisms $\varphi_i:\pi_S^{-1}(\mathcal C_i)\to \mathcal C_i\times\BP^1$, making the following diagram commutative
\[
\begin{tikzcd}
    \pi_S^{-1}(\mathcal C_i)\arrow[rr,"\varphi_i"]\arrow[rd,"\pi_S|_{\pi_S^{-1}(\mathcal C_i)}"']&& \mathcal C_i\times\BP^1\arrow[ld,"\pi_1"]\\
    &\mathcal C_i,
\end{tikzcd} 
\]
for all $i=1,\ldots,s$. As a consequence, all the geometrically ruled surfaces $S$ over $C$ are projectivisations of rank 2 vector bundles, i.e. $S\cong\BP_C E$, for some rank 2 vector bundle $E$ on $C$. The following proposition characterises vector bundles giving the same surface.

\begin{prop}[{\cite[Proposition II.7]{BEAUVILLE}}] \label{prop:invundertens}Let $C$ be a smooth projective curve. Two geometrically ruled surfaces $\BP_CE$, $\BP_CE'$, are $C$-isomorphic if and only if there exists a line bundle $L\in\Pic(C)$ such that $E'\cong E\otimes L$. 
\end{prop}
\begin{remark}
Given a geometrically ruled surface $ 
  \pi: \BP _CE \to C $, we can always assume that $\pi$ has a \textit{section} $\sigma$, i.e. a morphism $ 
    \sigma:C\to\BP _CE  $ such that $\pi\circ\sigma \equiv \id_C$. Indeed, we can twist $E$ with some line bundle $L\in \Pic(C)$ in order to satisfy the requirement.    
\end{remark}

The Picard group and the intersection pairing on a geometrically ruled surface are well understood. We encode them in the next proposition.

\begin{prop}[{\cite[Proposition V 2.3]{Hartshorne77}}] 
    Let $\pi_S:S\rightarrow C$ be a geometrically ruled surface and let $\sigma:C\rightarrow S$ be a section. Denote by $C_0\in\Div(S)$ the divisor defined by $\sigma$, i.e. $C_0=\sigma(C)$, and by $c_0,f\in H^2(S,\BZ)$ the cohomology classes of $C_0$ and of a fibre respectively. Then, the following isomorphisms hold
    \[
    \Pic(S)\cong \langle C_0\rangle_{\BZ}\oplus\pi_S^* \Pic(C),
    \]
    and
    \[
    H^2(S,\BZ)\cong \langle c_0\rangle_{\BZ}\oplus\langle f\rangle_{\BZ}.
    \]
    Moreover, we have $F^2=0$ and $F.C_0=1$.
\end{prop}

By a celebrated result of \textit{Alexander Grothendieck}, every vector bundle $E$ on $\BP^1$ decomposes as a direct sum of line bundles, see \cite{GrotP1}.  As a consequence of this result and of \Cref{prop:invundertens}, rational geometrically ruled surfaces are indexed by non-negative integers. We present them in the next definition. 

\begin{definition}[Hirzebruch surface]
Let $k\ge 0$ be a non-negative integer, the  \textit{$k$-th Hirzebruch surface} $\BF_k$ is the projectivisation of the rank-two vector bundle $\OO_{\BP^1}\oplus\OO_{\BP^1}(k)$, i.e.
\[
\BF_k\cong \BP_{\BP^1}(\OO_{\BP^1}\oplus\OO_{\BP^1}(k)).
\]
\end{definition} 

\begin{example}
    The first example of a Hirzebruch surface is $\BF_0$, i.e. the projectivisation of the trivial rank 2 vector bundle $\OO_{\BP^1}^{\oplus2}$. Therefore we have $\BF_0\cong \BP^1\times \BP^1$. On the other hand, the first Hirzebruch surface is $\BF_1\cong\Bl_p\BP^2$. In the language of \Cref{exa:plow1}, the fibration over $\BP^1$ corresponds to the morphism $\pi_1|_{\Bl_{e_2}\BP^2}$.
\end{example}

For the sake of completeness, in the next theorem we give the classical characterisation of minimal surfaces.
 
\begin{theorem}[{\cite[Theorems III.10,V.10,V.19]{BEAUVILLE}}]
Let $S$ be a smooth projective surface. Then, 
\begin{itemize}
    \item if $S$ is ruled and not rational the minimal surfaces in $B(S)$ are the geometrically ruled surfaces;
    \item if $S$ is rational the minimal surfaces in $B(S)$ are $\BP^2$ and the Hirzebruch surfaces $\BF_k$ for $k\not=1$;
    \item if $S$ is not ruled, there is a unique minimal surface in $B(S)$.
\end{itemize}
\end{theorem}

It is well known that every smooth projective curve can be holomorphically embedded in $\BP^3$, see \cite[Corollary IV 3.7]{Hartshorne77} and \Cref{rem:chow}. Similarly, every smooth projective surface can be embedded in $\BP^5$, see \cite[Proposition IV.5]{BEAUVILLE}. In \Cref{exa:realiHirz} we present Hirzebruch surfaces as projective subvarieties of $\BP^5$.

\begin{example}\label{exa:realiHirz}  In this example we construct an explicit embedding $\BF_k\hookrightarrow \BP^5$, for $k\ge 0$. Technically, we realise $\BF_k$ as a hypersurface in $\BP^2\times\BP^1$  which is then embedded in $\BP^5$ via the Segre $(2,1)$-embedding $s_{2,1}$, see \Cref{exa:1.1}.
First,  observe that we have
\[
\BF_0\cong\Set{([x_0:x_1:x_2],[\lambda_0:\lambda_1])\in\BP^2\times\BP^1|x_0-x_1=0}\subset \BP^2\times\BP^1,
\]
and
\[
\BF_1\cong \Set{([x_0:x_1:x_2],[\lambda_0:\lambda_1])\in\BP^2\times\BP^1|x_0\lambda_0-x_1\lambda_1=0}\subset \BP^2\times\BP^1,
\]
see \Cref{exa:plow1}. We promote this pattern to a sequence of surfaces. Define, for $k\ge 0$,
\[
S_k=\Set{([x_0:x_1:x_2],[\lambda_0:\lambda_1])\in\BP^2\times\BP^1|x_0\lambda_0^k-x_1\lambda_1^k=0}\subset \BP^2\times\BP^1.
\]
Then, we have $\BF_k\cong S_k$. This can be easily checked on a coordinate atlas for $S_k$. Precisely, the surface $S_k$ is contained in four   among the six charts coming from coordinate charts of $\BP^2\times\BP^1$, so providing four charts each isomorphic to an affine plane $\BA^2$. Then, a direct computation shows that these charts glue according to the transition functions of the projective bundle $\BF_k$. 
  
\end{example}

\begin{remark} 

    The fibration $\BF_k\to\BP^1$ has two distinguished sections $C_+,C_-$ such that $C_+^2=-C_-^2=k$. More details on Hirzebruch surfaces can be found in  \cite[Chapter V]{Hartshorne77}.
    \label{rem:sechirz}
\end{remark}
Any two Hirzebruch surfaces are birational to each other. The birational tranformations relating them are sequences of  the so-called \textit{elementary transformations}.

\Cref{fig:eltransf}   describes the elementary transformation relating $\BF_k$ and $\BF_{k+1}$. Each curve is labeled by its self-intersection. Notice that at each step self-intersections are computed via  \Cref{lemma:eltransf}.
\begin{figure}[ht!]
    \centering
    \scalebox{0.7}{\begin{tikzpicture}[basept/.style={circle, draw=black!100, fill=black!100, thick, inner sep=0pt,minimum size=1.2mm},scale=0.6]    
     \begin{scope}[xshift = -4cm, yshift=+1cm]
        \node at (0,2.5) {$\BF_k$};
        \draw[thick] (-2,-2)--(-2,2) node [midway, left] {\small $0$}
            (2,-2)--(2,2) node [midway, right] {\small $0$}
            (-2.5,1.5)--(2.5,1.5) node [midway, above] {\small $k$}
            (-2.5,-1.5)--(2.5,-1.5) node [midway, below] {\small $-k$}
      ;
      
	   \node at (-2,1.5) [basept,label={[xshift=-16pt, yshift = -3 pt] \small{$(\infty,0)$}}] {};
    \node at (2,1.5) [basept,label={[xshift=18pt, yshift = -3 pt] \small{$(\infty,\infty)$}}] {};
    \node at (2,-1.5) [basept,label={[xshift=18pt, yshift = -16 pt] \small{$(0,\infty)$}}] {};
    \node at (-2,-1.5) [basept,label={[xshift=-14pt, yshift = -16 pt] \small{$(0,0)$}}] {};
    \end{scope}

    \begin{scope}[xshift = +4cm, yshift=+1cm]
    \node at (0,2.5) {$\BF_{k+1}$};
    \draw[thick] (-2,-2)--(-2,2) node [midway, left] {\small $0$}
            (2,-2)--(2,2) node [midway, right] {\small $0$}
            (-2.5,1.5)--(2.5,1.5) node [midway, above] {\small $k+1$}
            (-2.5,-1.5)--(2.5,-1.5) node [midway, below] {\small $-k-1$}
      ;
      
	   \node at (-2,1.5) [basept,label={[xshift=-16pt, yshift = -3 pt] \small{$(\infty,0)$}}] {};
    \node at (2,1.5) [basept,label={[xshift=18pt, yshift = -3 pt] \small{$(\infty,\infty)$}}] {};
    \node at (2,-1.5) [basept,label={[xshift=18pt, yshift = -16 pt] \small{$(0,\infty)$}}] {};
    \node at (-2,-1.5) [basept,label={[xshift=-14pt, yshift = -16 pt] \small{$(0,0)$}}] {};
    \end{scope}
    
    \draw [->] (1.5,-2.75)--(2.5,-1.75) node[pos=0.5, below right] {$\text{Bl}_{(\infty,\infty)}\BF_{k+1}$};

    \begin{scope}[xshift = 0cm, yshift=-5.5cm]
    \draw[thick] (-2,-2)--(-2,2) node [midway, left] {\small $0$} 
            (-2.5,1.5)--(2.5,1.5) node [midway, above] {\small $k$}
            (-2.5,-1.5)--(2.5,-1.5) node [midway, below] {\small $-k-1$}
      ;
    \draw[thick,red] (2,2)  .. controls (2,0.8) .. (1.3,-0.2) node [midway, right] {\small $-1$};
    \draw[thick,red] (2,-2)  .. controls (2,-0.8) .. (1.3,0.2) node [midway, right] {\small $-1$};
       
    \end{scope}
 \draw [->] (-1.5,-2.75)--(-2.5,-1.75) node[pos=0.5, below left] {$\text{Bl}_{(0,\infty)}\BF_{k}$};
\end{tikzpicture}}
    \caption{The elementary transformation relating $\BF_k$ and $\BF_{k+1}$.}
    \label{fig:eltransf}
\end{figure}
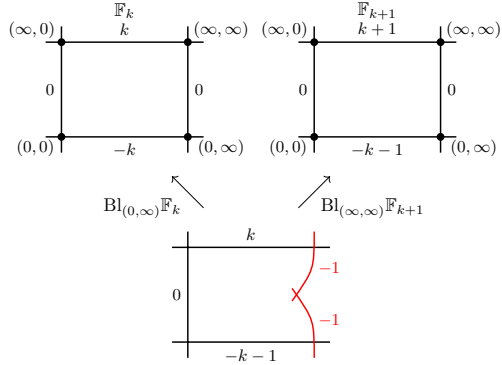

We conclude this section by stating Castelnuovo's rationality criterion and by giving a useful consequence that we shall use in what follows.

\begin{theorem}[Castelnuovo's rationality criterion, {\cite[Theorem V.1]{BEAUVILLE}}]\label{thm:CASTrat}
   Let $S$ be a smooth projective surface with $q(S)=P_2(S)=0$. Then, $S$ is rational.
\end{theorem}

\begin{prop}\label{prop:ratCOhomo}
Let $S$ be a smooth projective rational surface. Then, we have
\begin{equation}
    \label{eq:ratpichomo}\Pic(S)\cong H^2(S,\BZ).
\end{equation}
\end{prop}
\begin{proof}
    Consider the exponential sequence
    \begin{equation}\label{eq:expo}
        \begin{tikzcd}[row sep = tiny] 0\arrow[r]&\BZ\arrow[r]&\OO_S\arrow[r,"\exp"]&\OO_S^*\arrow[r]&0\\
        &&f\arrow[r,mapsto]&e^{2\pi i f}.
        \end{tikzcd}
    \end{equation}
    Then, since $H^0(S,\OO_S)\cong\BC$ and $H^0(S,\OO_S^*)\cong\BC^*$, we can split the long exact sequence in cohomology coming from \eqref{eq:expo} and we get
    \[
    \begin{tikzcd}
    0\arrow[r]&H^0(S,\BZ)\arrow[r]&H^0(S,\OO_S)\arrow[r]&H^0(S,\OO_S^*)\arrow[r]&0,
    \end{tikzcd}
    \]
    and 
     \[
   \begin{tikzcd}[column sep =tiny]
        \cdots\arrow[r]&H^1(S,\OO_S)\arrow[r]&H^1(S,\OO_S^*)\arrow[r]&H^2(S,\BZ)\arrow[r]&H^2(S,\OO_S)\arrow[r]&\cdots.
   \end{tikzcd}
    \]
    Recall\footnote{The isomorphism associates to each line bundle the 1-cocycle given by its transition functions, see \Cref{subsec:linebundledivisors} and  \cite[Section 1.1.2]{GRIHARR} for more details.} that the isomorphism $\Pic(S)\cong H^1(S,\OO_S^*)$ holds in general. Moreover, as a consequence of \Cref{thm:SErredual}, \Cref{prop:brinva} and \Cref{thm:CASTrat}, we have $H^1(S,\OO_S)\cong H^2(S,\OO_S)\cong H^0(S,\omega_S)=0$, which implies the statement.
\end{proof}
\begin{remark}
For the rest of these notes, we will mostly consider rational surfaces, frequently making use of the isomorphism \eqref{eq:ratpichomo}. To simplify the notation, we will use the same symbol to denote both line bundles and their cohomology classes.
\end{remark}
\section{Algebraic Entropy}\label{sec:2}
In this section we introduce the notion of integrability with respect to algebraic entropy. Then, we introduce three  alternative versions of the notion of a \emph{space of initial conditions} for  a given map $\Phi\in\Bir(\BP^{n})$ suitable for the computation of algebraic entropy. 
This concept is an analogue of Okamoto's construction
\cite{Okamoto1979,Okamoto1977} for the continuous Painlev\'e equations
\cite{InceBook},  formulated in the case of discrete Painlev\'e equations by Sakai in \cite{Sakai2001}.

\subsection{Degree of birational maps and algebraic entropy}\label{subset:degreebir}  We start from the definition of degree of a birational transformation of the projective space.

\begin{definition}[\cite{BellonViallet1999}]
        \label{def:degree}
        Given a birational map $\Phi\in\Bir(\BP^{n})$
        \begin{equation}
                    \label{eq:birpol}
               \begin{tikzcd}[row sep = tiny]
               \BP^n\arrow[r,dashed,"\Phi"]&\BP^n\\
                {[x_0:\cdots:x_{n}] }\arrow[r,mapsto] &{\left[f_{0}(x_0,\ldots,x_{n}):\cdots :f_{n}(x_0,\ldots,x_{n}) \right]},
               \end{tikzcd}
        \end{equation}
        such that its (homogeneous) polynomial entries $f_{i}\in R$ are devoid of common factors,
        that is $\gcd(f_{0},\ldots,f_{n})=1$, we define its \emph{degree} to
        be:
        \begin{equation}
                \label{eq:dphi}
                d^\Phi = \deg f_{i},
                \quad \text{ for any $i=0,\ldots, n$}.
        \end{equation}
        In the same way, for all $k\in\BN$ we define $d_k^\Phi $ as the \emph{degree} 
        of the $k$-th iterate to be
        \begin{equation}
                \label{eq:dn}
                d^\Phi_{k} = d^{\Phi^{\circ k}}\!.
        \end{equation}
\end{definition}

\begin{remark}
        \label{rem:degree}
        We make the following observations.
        \begin{enumerate}
            \item The degree of a birational map is invariant under conjugation by 
                projectivities
                (see \cite{BellonViallet1999,HietarintaViallet1998}) and the numbers $d^{\Phi}$ and $d_{k}^{\Phi}$ are
                uniquely determined by the birational map $\Phi\in\Bir(\BP^{n})$.
            \item  \Cref{def:degree} is not the usual definition of degree
                in algebraic geometry, see \Cref{subsec:ratmap}, and in particular all birational maps have degree 1 in the sense of \Cref{subsec:ratmap}, if considered as dominant rational maps between varieties of the same dimension. 
                The notion of degree in \Cref{def:degree} is used to measure the growth of complexity of birational maps under iteration, in the same spirit as the notion of intersection complexity due to \textit{Arnol'd} \cite{arnold}.
            \item It is crucial in \Cref{def:degree} to require that the
                polynomial entries have no common factors. For a given birational map $\Phi \in\Bir(\BP^{n})$, 
                after some iterations common
                factors can appear and they must be removed, see \Cref{exa:CREMONA}. This process
                has geometric meaning which we will discuss later in this section.
        \end{enumerate}
\end{remark}

Having specified the notion of degree of a birational map, we give
the definition of algebraic entropy, which measures the growth of the complexity of a birational map after iterations.

\begin{definition}[\cite{BellonViallet1999}]
        \label{def:algent} 
        The \emph{algebraic entropy} of a birational map $\Phi\in\Bir(\BP^{n})$ is the 
        following limit:
        \begin{equation}
                \label{eq:algentdef}
                S_{\Phi} = \lim_{k\to\infty}\frac{1}{k}\log d_{k}^\Phi.
        \end{equation} 
\end{definition}
It is worth mentioning that the related notion of the (first) dynamical degree is equivalent to the algebraic entropy in the special case of a birational self-map of $\BP^2$, and that this notion, going back at least to \textit{Russakovskii} and \textit{Shiffman} \cite{RussakovskiiShiffman} and to \textit{Friedland} \cite{Friedland} following his work with \textit{Milnor} \cite{friedland1989dynamical}, is widely used in the algebraic dynamics community, see e.g.  \cite{DillerFavre2001,BedfordKim2004,BedfordKim2008,BEDKIM}. 
\begin{remark}[\cite{BellonViallet1999, GubbiottiASIDE16,
                GrammaticosHalburdRamaniViallet2009}]
        \label{rem:algent}
We also remark that the algebraic entropy has the following properties:
        \begin{itemize}
            \item by the properties of birational maps and the subadditivity
                of the logarithm, using Fekete's lemma \cite{Fekete1923},
                the algebraic entropy always exists;
            \item the algebraic entropy is non-negative and bounded from above: $0\le S_\Phi \le \log d^\Phi$;
            \item the algebraic entropy is invariant with respect to birational
                conjugation. That is, given two birational maps
                $\Phi,\Theta\in\Bir(\BP^{n})$, we have $S_{\Phi}
                = S_{\Theta^{-1}\circ\Phi\circ\Theta}$;
            \item if $d_{k}^\Phi$ is sub-exponential as $k\to\infty$, e.g.       polynomial, then
                $S_{\Phi}=0$, while, if $d_{k}^\Phi \sim a^{k}$ for some $a\in\BR_{> 0}$, 
                then $S_{\Phi} = \log a$.
        \end{itemize}
\end{remark}

We are now in a position to define  integrability
with respect to algebraic entropy.

\begin{definition}[\cite{BellonViallet1999,HietarintaBook}]
        \label{def:integrability}
        A birational map $\Phi\in\Bir(\BP^{n})$ is
        \emph{integrable according to the algebraic entropy} if $S_{\Phi}
        = 0$. If $S_{\Phi}>0$ the map is said to be \emph{non-integrable}
        or \emph{chaotic}. Moreover, if $d_{k}^\Phi\sim k$ as $k\to \infty$ the
        map is said to have \emph{linear degree growth}. 
        Finally, if $d_{k}^\Phi$ is periodic
        the map is said to have \emph{periodic degree growth}. 
\end{definition}

\begin{remark}
       	\label{rem:integrability}
        Most of the known integrable maps are such that $d_{k}^{\Phi}\sim k^{2}$ as
        $k\to\infty$. From \cite{Bellon1999}, it is known that if 
        $\Phi\in\Bir(\BP^n)$ preserves a fibration of $\BP^n$ by elliptic curves on which $\Phi$ induces translation with respect to the corresponding group structure, 
        then the degree growth is quadratic.
        From~\cite{DillerFavre2001}, it is known that in $\BP^2$ the only sub-exponential 
        behaviours are quadratic, linear, and periodic. The first is associated
        with the preservation of an elliptic fibration, the second
        with the preservation of a rational fibration, the latter with a power of
        the map being isotopic to the identity.
        In $\BP^n$ with $n>2$, it is possible that $d_{k}^{\Phi}\sim k^{\alpha}$ as $k\to\infty$
        with $\alpha>2$. For instance, in \cite{AnglesMaillarViallet2002,
                LaFortuneetal2001, GJTV_class, JoshiViallet2017,Viallet2019} maps with cubic  growth were presented.  
\end{remark}

In \Cref{rmk:genPull} we presented a notion of pull-back of divisors via a rational function. 
We now briefly present the analogous cohomological notion, see \cite{CarsteaTakenawa2019JPhysA,FULTON} for more details.

\begin{definition}[Pull-back in cohomology] 
Let $X,Y$ be two smooth irreducible projective varieties and let $\Phi\in\BC(X,Y)$ be a birational map. Denote by $Z=\graph(\Phi)$ the graph of $\Phi$ and consider the  commutative diagram 
\begin{equation}
\begin{tikzpicture}[scale = 2] 
    \node at (0,0) {$Z$};
    \node at (1,-1) {$Y.$};
    \node at (-1,-1) {$X$};
    \node[above] at (0,-1) {\small$\Phi$};

    \draw[->] (-0.1,-0.2) -- (-0.9,-0.8);
    \draw[->] (0.1,-0.2) -- (0.9,-0.8);
    \draw[dashed,->] (-0.8,-1) -- (0.8,-1);
    \node[right] at (0.5,-0.5) {\small $\pi_Y|_{Z}$};
    \node[left] at (-0.5,-0.5) {\small $\pi_X|_{Z}$};
\end{tikzpicture}
\end{equation}
Then, \textit{the pull-back} $\Phi^*$ is the map
\begin{equation}
    \Phi^*\colon H^2(Y,\BZ)\rightarrow H^2(X,\BZ),
\end{equation}
defined as $\Phi^*=\left(\pi_X|_{Z}\right)_*\circ\left(\pi_Y|_{Z}\right)^*$, where the pull-back and the push-forward on the right-hand side are the usual inverse and direct images of cycles via morphisms.  On the other hand, the push-forward $\Phi_*$ is the pull-back $\Phi_*=(\Phi^{-1})^*$ of the inverse of $\Phi$.
\end{definition}

\begin{remark}
    As explained in \cite{CarsteaTakenawa2019JPhysA}, it is possible to perform the actual computation via an auxiliary smooth variety $\widetilde{Z}$ instead of the possibly singular graph of $\Phi$. This is possible thanks to the celebrated result of \textit{Hironaka} on resolution of singularities  (see \cite{Hironaka1964}).
\end{remark}

 In order to introduce
the concept of \emph{ space
of initial conditions} we need to give the following definition.

\begin{definition}
        \label{def:anstable}
        A rational map $\Phi$ from a smooth projective variety $X$  to itself is   \emph{algebraically stable} if the equality $(\Phi^*)^k = (\Phi^k)^*$ holds for all $k\ge0$.
\end{definition}

\begin{remark}
        \label{rem:singularities}
        The concept of algebraic stability is related to the one of
        singularity confinement. Indeed, heuristically algebraic stability means that the
        indeterminacies of the iterations of the map and its inverse 
        behave in a controlled way: they either form finite 
        or periodic patterns. 
        In practical terms, for any divisor $D$ contracted by $\Phi$, there exists a positive integer $k$ and divisor $D'$ such that $\Phi^{\circ k}(D)=D'$. The sequence of subvarieties encountered under iteration of $\Phi$ from $D$ is the \textit{singularity pattern} of $\Phi$ starting from $D$. The term \textit{confinement} refers to the return of $D$ to a divisor after finitely many steps.
        Specifying to the case of
        interest, i.e. birational transformations of $\BP^{n}$, a singularity pattern will be of the
        following form:
        \begin{equation}
                \label{eq:singpatt}
                D \xrightarrow{\Phi} \gamma_{1} \xrightarrow{\Phi}
                \gamma_{2}\xrightarrow{\Phi} \cdots \xrightarrow{\Phi}
                \gamma_{K} \xrightarrow{\Phi} D',
        \end{equation}    
        where $D$, $D'$ are divisors and
        $\gamma_{i}$ are varieties of codimension greater than one. Finite concatenations of
        patterns of the form \eqref{eq:singpatt} can repeat periodically as
        long as the number of centres $\gamma_{i}$ stays finite (this
        last requirement can be false for linearisable equations
        \cite{Ablowitz_et_al2000,Takenawaatel2003,HayHoweseNakazonoShi2015}).
        Following \cite{BellonViallet1999,Viallet2015}, we can detect the divisors contracted by the map $\Phi$ and its inverse. Precisely, denoting by $\Psi\in\Bir(\BP^{n})$
        the inverse of $\Phi$, the following relations hold:
        \begin{equation}
                \label{eq:kappadef}
                \Psi \circ \Phi\equiv \kappa \cdot \id_{\BP^n},
                \quad
                \Phi \circ \Psi\equiv \lambda \cdot\id_{\BP^n},
        \end{equation}
        for some   $ 
                \kappa,\lambda\in R$. The polynomials $\kappa$ and $\lambda$ admit possibly trivial 
        factorisations of the form:
        \begin{equation}
                \label{eq:kappafact}
                \kappa = \prod_{i=1}^{K_{\kappa}} \kappa_{i}^{d_{\kappa,i}},
                \quad
                \lambda = \prod_{i=1}^{K_{\lambda}} {\lambda}_{i}^{d_{\lambda,i}},
        \end{equation}
        where $\kappa_i \neq \kappa_j$ and $\lambda_i \neq \lambda_j$ for $i\neq j$.
         {The only (prime) divisors that can be contracted to subvarieties of higher
        codimension by $\Phi$ are then the hypersurfaces:
        \begin{equation}
                \label{eq:singvarphi}
                \mathrm{K}_i =  \{\kappa_{i}=0\},
                \quad\mbox{for }i=1,\ldots,K_\kappa,
        \end{equation}
        while $\Psi$ can only contract the hypersurfaces:
        \begin{equation}
                \label{eq:singvarpsi}
                \Lambda_j = \{\lambda_{j}=0\}\quad\mbox{for }j=1,\ldots,K_\lambda.
        \end{equation}}
        In \Cref{fig:singularities} we present a possible blow-down blow-up 
        sequence in $\BP^{3}$: the surface $D$ is mapped to a curve $\gamma$ and then to a point $p$, 
        but after four steps the singularity is confined 
        and a new surface $D'$ is
        found. This is a graphical representation of a sequence as in \eqref{eq:singpatt}.
\end{remark}

\begin{figure}[htb]
        \centering
        \scalebox{0.7}{\begin{tikzpicture}[scale=0.3]
                \draw[thick] (1,0)--(5,4)--(5,12)--(1,8)--(1,0);
                \node[below right] at (3,2) {$D$};
                
                \draw[->] (5.8,6) -- (8.2,6);
                \draw (9,0).. controls (7,5) and (11,9) ..(9.1,12);
                \draw[->] (9.2,6) -- (12-0.5,6);
                
                \node[below] at (9,0) {$\gamma$};
                \filldraw (12,6) circle (3pt);
                \node[above] at (12,6) {$p$};
                \draw[->] (12.4,6) -- (14.6,6);
                \filldraw (15,6) circle (3pt);
                \node[above] at (15,6) {$p'$};
                
                \draw[->,dashed] (15.5,6)--(18-0.6,6);
                \draw (18,0).. controls (16,5) and (20,9) ..(18,12);
                \node[below] at (18,0) {$\gamma'$};
                
                \draw[->, dashed] (18.3,6) -- (20.7,6);
                \draw[thick] (21.5,0)--(25.5,4)--(25.5,12)--(21.5,8)--(21.5,0);
                \node[below right] at (23.5,2) {$D'$};
        \end{tikzpicture}}
        \caption{A possible blow-down blow-up sequence in $\BP^{3}$.}
        \label{fig:singularities}
\end{figure}
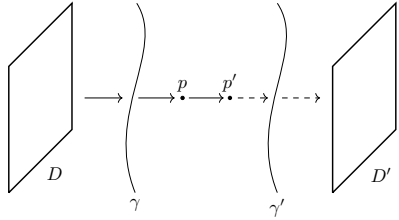

The following result allows us to characterise algebraically stable maps
by studying their indeterminacy loci, see \Cref{rem:singularities}.

\begin{prop}[\cite{Bayraktar2013,CarsteaTakenawa2019JPhysA,BedfordKim2004,BedfordKim2008}]
        \label{prop:anstable}
         {Let $X$ be a smooth projective variety and let
        $\Phi\in\Bir( X)$ be a birational map with indeterminacy locus
        $ \ind\Phi$. Then, the map $\Phi$ is algebraically stable if
        and only if there do not exist a positive integer $k$ and a prime 
        divisor $E$ on $X$  such that $\Phi (E \smallsetminus  \ind\Phi )
        \subset  \ind( \Phi^k)$.}
\end{prop} 
\begin{definition}
        \label{def:spaceofinitialcond}
    Let $\Phi\in\Bir(\BP^n)$ be a birational transformation. A birational projective morphism $\varepsilon\colon B\to\BP^{n}$ with $B$ a smooth variety and
        the lifted (birational) maps denoted by $\widetilde{\Phi},\widetilde{\Phi}^{-1}\in\Bir (B)$ is a  
        \begin{itemize}
            \item \textit{space of initial conditions in the algebraic stability sense} if the lifted (birational) maps $\widetilde{\Phi},\widetilde{\Phi}^{-1}\in\Bir (B)$ are algebraically stable,
            \item \textit{space of initial conditions in the pseudo-automorphism sense} if the lifted (birational) map $\widetilde{\Phi} \in\Bir (B)$ is a pseudo-automorphism, i.e. an automorphism in codimension 1,
            \item \textit{space of initial conditions in the automorphism sense} if the lifted (birational) map $\widetilde{\Phi} \in\Bir (B)$ is an automorphism.
        \end{itemize}

A space of initial condition in the algebraic stability / pseudo-automorphism / automorphism sense $B$ is minimal if any birational morphism $\varphi : B\to B'$ to another space of initial conditions in the same sense  is an isomorphism.

    For the remainder of this section, whenever not specified we say space of initial conditions to indicate space of initial conditions in the algebraic stability sense.
\end{definition}

\begin{remark}
    The notion of space of initial conditions in the algebraic stability sense is used in \cite{GraffeoGubbiotti} and  \cite{CarsteaTakenawaAlgebraicallyStable} where $B$   is called an \textit{`algebraically stable variety'}.  
    Spaces of initial conditions in the pseudo-automorphism sense when $n\geq 2$ have been considered, for instance, in \cite{BEDKIM,CarsteaTakenawa2019JPhysA}. 
     Note that for $n=2$ the pseudo-automorphism sense and automorphism sense are the same. 
     Note also for $n=2$ that the results of \textit{Diller} and \textit{Favre} \cite{DillerFavre2001} imply that any $\Phi\in\Bir(\BP^2)$ is birationally conjugate to an algebraically stable self-map of a smooth rational surface. 
     In his work, \textit{Sakai} uses space of initial conditions in the isomorphism sense (a more general one allowing for non-autonomous equations), as does \textit{Mase} \cite{Mase}. 
    It is worth mentioning that the last notion, of lifting a birational map to an automorphism, is hardly used in the context of integrable systems when $n>2$, but there are some cases where they are of interest, see the Ph.D. thesis \cite{Kusnetsova2022} of \textit{Alexandra Kuznetsova} and references therein.
\end{remark}

Spaces of initial conditions are constructed via blow-ups. In the context of computation of algebraic entropy, the idea is that if a divisor $D\subset\BP^n$  is contracted by $\Phi\in\Bir(\BP^n)$ to a subvariety $\gamma$ of codimension greater than 1, then after blow-up $\gamma$ is turned into a divisor. However, the lift of $\Phi$ to the blow-up can in principle still contract $D$ on some subvariety of the  exceptional  divisor over $\gamma$. This behaviour requires one to perform iterated blow-ups, i.e. to blow-up infinitely near subvarieties of $\BP^n$, until $D$ is no longer contracted, see \Cref{sec:sec3} for explicit examples of iterated blow-ups.

Assume\footnote{We will just blow-up reduced points and no iterated blow-ups will be performed in this section, as the general case is more intricate and beyond our purpose.} now that the (irreducible) subvarieties $\gamma_i$, for $i=1,\ldots,K$, of codimension  greater than one encountered in the singularity pattern \eqref{eq:singpatt} of some map $\Phi$ are disjoint, i.e. $\gamma_i\cap\gamma_j=\varnothing$ for $i\not=j$, irreducible, smooth and all lie in $\BP^n$, that is no iterated blow-up is required. 
As a consequence of \Cref{rem:singularities} and of the properties of the blow-up (see \cite[Proposition IV-22]{GEOFSCHEME}), we have that
\begin{equation}
        \label{eq:blow-ups}
        B = \Bl_{\underset{i=1}{\overset{K}{\cup}}\gamma_{i}}\BP^{n}
\end{equation}
is a space of initial conditions for $\Phi$. Denoting by $F_{i}$, for $i=1,\ldots,K$, the components of the exceptional locus of $\varepsilon$, we attach to $B$ its second cohomology group (see \cite[Section 4.6.2]{GRIHARR}):
\begin{equation}
        \label{eq:H2B}
        H^{2}(B,\BZ) =
        \langle \varepsilon^{*}H , F_{1},\ldots, F_{K}\rangle_{\BZ}.
\end{equation}
Then, the action of $(\Phi^{-1})^{*}$ on $H^{2}(B,\BZ)$ is linear and
the coefficient of the pull-back of $\varepsilon^{*}H$ via $\Phi$ agrees with the \emph{degree of $\Phi$} in the sense of \Cref{eq:dn}.
So, following \cite{Takenawa2001JPhyA,BedfordKim2004}, from the algebraic stability 
condition we get that:
\begin{equation}
        \label{eq:dncoeff}
        d_{k}^\Phi  = \coeff\left( ((\Phi^{-1})^{*})^{k}\varepsilon^* H, \varepsilon^{*}H \right) = \coeff\left( (\Phi^{*})^{-k}\varepsilon^* H, \varepsilon^{*}H \right),
\end{equation}
that is, we converted the problem of finding a closed form
expression for $d_{k}^\Phi$ to a problem in linear algebra over
the $\BZ$-module $H^{2}(B,\BZ)$. 

\subsection{A working example: the Cremona-Cubes group}\label{subsec:2example} We present now an explicit example of a discrete integrable system in dimension 3 following \cite{GraffeoGubbiotti}, see \cite{alonso2025discretepainleveequationspencils2,alonso2025discretepainleveequationspencils,alonso2023dynamicaldegreesbirationalmaps,alonso2023three} for other 3-dimensional examples and more details.

Let $\crem_3\in\Bir(\BP^3)$ be the standard Cremona transformation, see \Cref{exa:CREMONA} and \Cref{exe:CREMONA}.

\begin{notation}
 In this subsection we adopt the unusual choice $[x_1:\cdots:x_4]$ to denote the homogeneous coordinates on $\BP^3$. We also denote by $e_i,H_i$, for $i=1,\ldots,4$, the coordinate points and hyperplanes of $\BP^3$. 
\end{notation}

\begin{remark}
    \label{puntifissi} We have
    \begin{equation}
        \label{eq:fixc3}
        \Fix\crem_3 = \OP \cup \OQ,
        \quad 
         \OP=\Set{p_1,p_2,p_3,p_4},
        \,
        \OQ=\Set{q_1,q_2,q_3,q_4},
    \end{equation}
    where:
    \begin{equation}
        \begin{aligned}
                p_1&=[1:-1:-1:-1],  \\
                p_2&=[-1:1:-1:-1], \\
                p_3&=[-1:-1:1:-1],  \\
                p_4&=[-1:-1:-1:1], 
        \end{aligned}
        \quad
        \begin{aligned}
                q_1&=[1:-1:-1:1],  \\
                q_{2}&=[-1:1:-1:1],  \\
                q_{3}&=[1:1:-1:-1],  \\
                q_{4}&=[1:1:1:1].
        \end{aligned}
    \end{equation}
    These eight points correspond to two four-tuples of lines of 
    $\BC^4$ orthogonal with respect to the standard scalar product. In particular, these are four-tuples of points in general position.  We highlight that the points in $\Fix \crem_3$ can be interpreted as the vertices of a  cube in the affine space, as depicted in \Cref{CUBE}. Here, by vertices of a cube we mean the base locus of a general net of quadrics of $\BP^3$ (see  \cite[App. B.5.2]{HuntGeomQuot1996} and \cite[Section 1.5.2]{Dolgachev2012book}). Note that we are considering only general nets in order to have a 0-dimensional\footnote{The dimension of the base locus may jump in some special cases, an example being the twisted cubic.} reduced base locus. 
\end{remark}

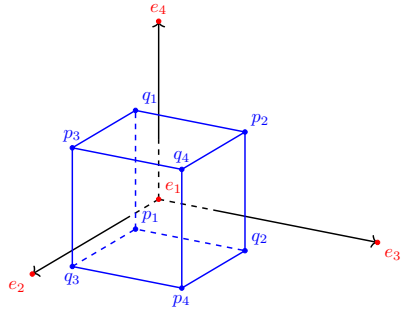
\begin{figure}[htb]
	\centering
    \tdplotsetmaincoords{70}{120}
    \scalebox{0.7}{\begin{tikzpicture}[scale=1.2,tdplot_main_coords]
        \draw[color=black,thick,dashed] (0,0,0) -- (1,0,0);
        \draw[color=black,thick,->] (1,0,0) -- (4,0,0) node[anchor=north east,red]{$e_2$};
        \draw[color=black,thick,dashed] (0,0,0) -- (0,1,0);
        \draw[color=black,thick,->] (0,1,0) -- (0,4,0) node[anchor=north west,red]{$e_3$};
        \draw[color=black,thick,dashed] (0,0,0) -- (0,0,1);
        \draw[color=black,thick,->] (0,0,1) -- (0,0,3) node[anchor=south,red]{$e_4$};

        \coordinate (p1) at (-1,-1,-1);
        \coordinate (p2) at (-1,1,1);
        \coordinate (p3) at (1,-1,1);
        \coordinate (p4) at (1,1,-1);
        
        \coordinate (q1) at (-1,-1,1);
        \coordinate (q2) at (-1,1,-1);
        \coordinate (q3) at (1,-1,-1);
        \coordinate (q4) at (1,1,1);
                
        \draw[blue,thick] (q1)--(p2)--(q4)--(p3)--cycle;
        \draw[blue,thick] (q4)--(p4)--(q2)--(p2);
        \draw[blue,thick] (p3)--(q3)--(p4);
        \draw[blue,thick,dashed] (q3)--(p1)--(q2);
        \draw[blue,thick,dashed] (p1)--(q1);
        
        \node[blue,anchor=south west] at (p1) {$p_1$};
        \draw[fill = blue,draw=blue] (p1) circle (1pt);
        \node[blue,anchor=south west] at (p2) {$p_2$};
        \draw[fill = blue,draw=blue] (p2) circle (1pt);
        \node[blue,anchor=south] at (p3) {$p_3$};
        \draw[fill = blue,draw=blue] (p3) circle (1pt);
        \node[blue,anchor=north] at (p4) {$p_4$};
        \draw[fill = blue,draw=blue] (p4) circle (1pt);
        
        \node[blue,anchor=south west] at (q1) {$q_1$};
        \draw[fill = blue,draw=blue] (q1) circle (1pt);
        \node[blue,anchor=south west] at (q2) {$q_2$};
        \draw[fill = blue,draw=blue] (q2) circle (1pt);
        \node[blue,anchor=north] at (q3) {$q_3$};
        \draw[fill = blue,draw=blue] (q3) circle (1pt);
        \node[blue,anchor=south] at (q4) {$q_4$};
        \draw[fill = blue,draw=blue] (q4) circle (1pt);
        
        \node[red,anchor=south west] at (0,0,0) {$e_1$};
        \draw[fill = red,draw=red] (0,0,0) circle (1pt);
        \draw[fill = red,draw=red] (4,0,0) circle (1pt);
        \draw[fill = red,draw=red] (0,4,0) circle (1pt);
        \draw[fill = red,draw=red] (0,0,3) circle (1pt);
    \end{tikzpicture}  }
    
	\caption{The configuration, in the chart $\mathcal{U}_{x_1}\subset \BP^3$, of the points in $\OR$.}
	\label{CUBE}
\end{figure}

\begin{example}
        \label{punticonfinement} 
        If we are interested in spaces of initial conditions for $\crem_3$, 
        we do not need to work with a resolution of singularities of the singular variety $X$ constructed in  \Cref{exe:CREMONA},
        and it is enough to consider the variety
        \begin{equation}
        {B}=\Bl_{\OOE}\BP^3, \quad
                \OOE=\Set{\!e_1,e_2,e_3,e_4\!}\!,    
        \end{equation}
        where
        \begin{equation}
            \label{eq:edef}
            \begin{aligned}
                e_1&=[1:0:0:0],  
                \\
                e_2&=[0:1:0:0],  
                \\
                e_3&=[0:0:1:0],  
                \\
                e_4&=[0:0:0:1].     
            \end{aligned}
        \end{equation}
        Indeed, the only divisorial contractions of $\crem_3$ consist of contractions over 
        one of the $e_i$'s and the map induced by $\crem_3$ on ${B}$ is algebraically stable.  Let us denote by $E_{i}$
the exceptional divisor over the coordinate point $e_{i}$, for
$i=1,\ldots,4$. In this setting, one can choose (see \cite[Section 4.6.2]{GRIHARR} or \Cref{lemma:eltransf})
the following basis of the second singular cohomology group of ${B}$:
\begin{equation}
	\label{eq:H2cremona} H^{2}( {B},\BZ ) = \langle
	\varepsilon^*H,E_1,E_2,E_3,E_4\rangle_{\BZ},
\end{equation} 
where $H$ is the class of a hyperplane in $\BP^3$ and, by abuse
of notation, we have denoted by the same symbols the exceptional
divisors $E_i$, for $i=1,\ldots,4$, and their cohomology classes.
         
Then, the action of the standard Cremona transformation on the
second cohomology group $H^{2}({B},\BZ)$ is expressed, in terms of the
basis \eqref{eq:H2cremona}, by the following matrix
\begin{equation}
    \label{CREMONACTION} 
(\crem_3^{-1})^* =   \crem_3^*=
    \begin{pmatrix} 
        3 & 1 & 1 & 1 & 1
        \\
	-2 & 0 & -1 & -1 & -1
        \\
        -2 & -1 & 0 & -1 & -1
        \\
        -2 & -1 & -1 & 0 & -1
        \\ 
        -2 & -1 & -1 &- 1 & 0
    \end{pmatrix}\!,
\end{equation} 
see for instance \cite{CarsteaTakenawa2019JPhysA} 
{or \cite[Eq. (3.1)]{BedfordKim2004} evaluated at $d=3$}.
\end{example}

\begin{definition}\label{specialpoints}
        We will denote\footnote{The letter $\OR$ stands for \textit{Reye}.} by $\OR\subset \BP^3$ the
        finite subset containing all the points appearing in
        \Cref{puntifissi,punticonfinement}, i.e.
        \begin{equation}
                \label{eq:setSdef}
                \OR=\OOE \cup\OP\cup\OQ.
        \end{equation}
\end{definition}

In what follows, we compute the algebraic entropy of maps of the
form $\Phi=g\circ \crem_3$ where $g\in \BP\GL(4,\BC)$ is a projectivity of finite order
that acts on the set $\OR$.

\begin{definition}\label{def:cubegroup}
        We will call the \textit{Cremona-cubes group} the subgroup 
        $\OC$ of $\BP\GL(4,\BC)$ defined by:
        \begin{equation}
                \label{eq:CremonaGroup}
                \OC=
                \Set{g\in\BP\GL(4,\BC)|g\cdot \OR\subseteq \OR}\!.
        \end{equation}
\end{definition} 

\begin{remark}
  We remark that the condition $g\cdot \OR\subseteq \OR$ is equivalent to $g\cdot \OR= \OR$. However, we have opted for this more common presentation.  We also remark that since $\OR $ contains five-tuples of points in general position, we have $\stab_{\langle g \rangle} 
        (\OR)=\langle\id_{\BP\GL
        (4,\BC)}\rangle$,  for any $g\in\OC$ (see \cite[Section 1.3]{PARDINI}).   
        This implies that all the elements
        $g\in\OC$ have finite order. Indeed, suppose that there exists a $g\in\OC$ of infinite order. In particular, for any integer $k>1$, $g^k$ is not the inverse of $g$. Now, since $g$ acts on the finite set $\OR$, there is an integer $k>1$ such that $g|_{\OR}\equiv g^k|_{\OR}$. This implies that $g^{1-k}$ would be a non-trivial element in $\stab_{\langle g \rangle} 
        (\OR)$. 
\end{remark}

The following result tells us that, within $\OR$, the three subsets $\OOE$,
$\OP$, and $\OQ$ are mapped between themselves {as a whole}.

\begin{lemma}\label{azioneindotta}
        The action of $\OC$ on $\OR$ induces an action of
        $\OC$ on the set $\Set{\OOE,\OP,\OQ}$.
\end{lemma}

\begin{proof}
        First notice that, if a line $L$ in $\BP^3$ contains at least two points from $\OR$, then it contains either three  collinear points each
        belonging to one of the sets $\OOE,\OP$ and $\OQ$ or two points
        from the same collection $\OOE,\OP$ or $\OQ$.
        
        We now proceed by contradiction. Suppose, without loss of generality,
        that the projectivity $g$ sends the point $e_1$ to the point $p_1$
        and the point  $e_2$ to the point $q_2$, i.e.
        \begin{equation}
                \label{eq:p1p2proof}
                \begin{tikzcd}[row sep=tiny]
                        e_1 \arrow[mapsto]{r}{g} & p_1 \\
                        e_2\arrow[mapsto]{r}{g} & q_2 .
                \end{tikzcd}
        \end{equation}
        Let $L_{12}$ be the line through $p_1$ and $q_2$ and let $e_j$ be
        the third intersection point in $L_{12}\cap\OR$, i.e.:
        \begin{equation}
                \label{eq:L59intersect}
                L_{12}\cap\OR=\Set{\!e_j,p_1,q_2\!}\!.
        \end{equation}
        Then, we get 
        \begin{equation}
                g^{-1 }\cdot e_j \in (g^{-1}(L_{12})\smallsetminus\{e_1,e_2\})\cap\OR.
                \label{eq:gpapplied}
        \end{equation} 
        Which is a contradiction.
\end{proof}

Now, we characterise the elements of $\OC$ as belonging to three
different classes depending on their action on the set $\Set{\!\OOE,\OP,\OQ\!}$. 
The following Lemma is crucial in this characterisation.

\begin{lemma}
        \label{possibilemma}
        Let $g\in\OC\subset \BP\GL(4,\BC)$ be an element
        of the Cremona-cubes group. Then, there is a matrix
        $\widetilde{g}\in\GL(4,\BC)$ representing $g$ whose entries
        belong to $\{-1,0,1\}$. Moreover, $g$ falls in one of the
        following cases.
        \begin{itemize}
                \item[\textit{(A)}] Both,  the columns and rows of $g$ represent the points in $\OP$ (or in $\OQ$).
                \item[\textit{(B)}] The matrix $g$ is a permutation matrix with signs.
                \item[\textit{(C)}] The columns of $g$ represent the points in $\OP$ and the rows represent the points in $\OQ$ (or viceversa).
        \end{itemize}
\end{lemma}

\begin{proof}
    The first part of the statement follows from the second, while the second
    part is  a direct consequence of \Cref{azioneindotta}. Indeed, as per
    \Cref{azioneindotta}, $g$ and $g^{-1}$ act on $\Set{\OOE,\OP,\OQ}$
    and, depending on the action on this set, we get \textit{(A)},
    \textit{(B)} or \textit{(C)}.
\end{proof}

\begin{remark}\label{orbits}
    As a consequence of \Cref{possibilemma}, we can divide the elements of
    the Cremona-cubes group according to the orbit $\langle g \rangle\cdot
    \OOE$ of $\OOE$ via $g$:
    \begin{equation}\label{eq:orbit}
            \langle g \rangle\cdot \OOE=\Set{g^k\cdot e_i | k\in\BN, \ 1\le i\le 4}\!.
    \end{equation}
    We have the following characterisation.
    \begin{itemize}
        \item An element $g \in \OC$ belongs to case \textit{(A)} in
            \Cref{possibilemma} if and only if 
            $\langle g \rangle\cdot \OOE=\OOE\cup \OP$ (or 
            $\langle g \rangle\cdot \OOE=\OOE\cup \OQ$), i.e. if and only if $\abs{\langle g\rangle\cdot \OOE} = 8$.
        \item An element $g \in \OC$ belongs to case \textit{(B)} in 
            \Cref{possibilemma} if and only if 
            $\langle g \rangle\cdot \OOE=\OOE$, i.e. if and only if $\abs{\langle g\rangle\cdot \OOE} = 4$.
        \item An element $g \in \OC$ belongs to case \textit{(C)} in 
            \Cref{possibilemma} if and only if 
            $\langle g \rangle\cdot \OOE=\OOE\cup \OP\cup \OQ=\mathscr{ R}$, i.e. if and only if $\abs{\langle g\rangle\cdot \OOE} = 12$. 
    \end{itemize}
    In what follows, we show that the orbit $\langle g \rangle
    \cdot \OOE$ \eqref{eq:orbit} plays a fundamental role in the
    confinement of singularities of the maps of the form $\Phi=g\circ
    \crem_{3}$ for $g\in\OC$.
\end{remark}

\begin{definition}
    We   say that an element $g\in \OC$ is of type \textit{(A)}
    (resp. of type \textit{(B)} or \textit{(C)}) if it belongs to
    the case \textit{(A)} (resp. \textit{(B)} or \textit{(C)}) in
    \Cref{possibilemma}.
\end{definition}

The following lemma investigates the relation between elements of the
Cremona-cubes group of different type.

\begin{lemma}\label{proprietagruppi}
        The following properties hold for the elements in $\OC$ (see \Cref{possibilemma}).
        \begin{itemize}
                \item Two elements of type \textit{(A)} (resp.  \textit{(C)}) differ by multiplication by a permutation matrix with sign having an even number of -1 (which is an element of $\OC$ of type \textit{(B)}).
                \item Two elements of type  \textit{(B)} differ by multiplication by a permutation matrix with signs, i.e. by an element of type \textit{(B)}. In particular, the elements of type \textit{(B)} form a subgroup of $\OC$ that we will denote by $\OC_{\mbox{\tiny\textit{(B)}}}$.
                \item An element of type \textit{(A)} differs by an element of type \textit{(C)} by multiplication by a permutation matrix with signs having an odd number of -1 (which is an element of $\OC$ of type \textit{(B)}). 
                \item The inverse of an element of type \textit{(A)} (resp. \textit{(B)} or \textit{(C)}) is of  the same type.
        \end{itemize}
\end{lemma}
\begin{proof}
        The proof of the first three points consists of a direct check, while the fourth point is a direct consequence of \Cref{orbits}.
\end{proof}

\begin{theorem}
    \label{thm:cardinality}
    The cardinality of $\OC $ is
    \begin{equation}
        |\OC|=576.
        \label{eq:cdim}
    \end{equation}
\end{theorem}

\begin{proof}
    Direct computation of the cardinality of $\OC$ using the
    computer software \texttt{Macaulay2} \cite{M2} with the package
    \texttt{InvariantRing} \cite{INVRING}.
\end{proof}

\begin{remark}
    \label{rem:identification}
    We observe that the Cremona-cubes group $\OC$ is
    isomorphic to $((A_4\times A_4)\rtimes \BZ/2\BZ)\rtimes
    \BZ/2\BZ$, where $A_4<S_4$ is the alternating subgroup in the symmetric group of four elements, i.e. the subgroup consisting of permutations with even order.  This identification is obtained using the function
    \texttt{StructureDescription} of the system for computational
    discrete algebra GAP \cite{GAP4}\footnote{Using the function
    \texttt{IdGroup} we see that  $\OC$ is the 8654-th finite
    group of order 576 of the finite groups database provided by
    GAP (see  $\texttt{SmallGroupInformation}$ \cite{GAP4}).}.

\end{remark}

\begin{theorem}
	\label{thm:main}
Let $\OOE=\Set{e_1,\ldots,e_4}$	be the set of coordinate points of $\BP^3$. Consider a birational map of the form $\Phi=g\circ\crem_{3}\in\Bir(\BP^{3})$ (or $\Phi=\crem_{3}\circ g$),  for some $g\in\OC$. 
    Then, there are three possibilities depending on the cardinality of the orbit 
    $\langle g\rangle\cdot \OOE$ of the points in $\OOE$, under the action of $g$.
	\begin{itemize}
		\item[\textit{(A)}] If $\abs{\langle g\rangle\cdot \OOE} = 8$ then the map is integrable
            in the sense of \Cref{def:integrability}, with  $d_n^\Phi\sim n^{2}$ as $n\to\infty$. 
		\item[\textit{(B)}] If $\abs{\langle g\rangle\cdot \OOE} = 4$  then the map has periodic degree growth 
            in the sense of \Cref{def:integrability}, with  $d_{n}^{\Phi}\in\Set{1,3}$. 
		\item[\textit{(C)}] If $\abs{\langle g\rangle\cdot \OOE} = 12$
            then the map is non-integrable in the sense of \Cref{def:integrability}, with $d_{n}^{\Phi}\sim \varphi^{2n}$, where $\varphi$ is the golden ratio. 
	\end{itemize}
\end{theorem}

\begin{proof}

We provide the proof in case \textit{(C)}. The other cases are similar.

Let $g\in\OC$ be an element of type \textit{(C)} and let $B_g=\Bl_{\mathscr R}\BP^3$ be the blow-up of $\BP^3$ centred at the orbit $\mathscr R=\langle g\rangle\cdot \OOE$ of $\OOE$ via $g$ and by $\varepsilon_g$ the blow-up morphism. We fix (see \cite[Section 4.6.2]{GRIHARR} or \Cref{lemma:eltransf}) the following basis of the second cohomology
group of $B_{g}$
\begin{equation}
    \label{eq:H2caseB} 
    H_{\mbox{\tiny\textit{(C)}}}^2 = 
    H^{2}\left(B_{g},\BZ \right) = 
    \langle \varepsilon_{g}^*H,E_1,E_2,E_3,E_4, P_1,P_2,P_3,P_4,Q_1,Q_2,Q_3,Q_4\rangle_{\BZ},
\end{equation}
where $H$ is the class of a hyperplane in $\BP^3$, and, for $i=1,\ldots,4$,  $E_i,P_i,Q_i$ are the cohomology classes of the exceptional divisors over the points $e_i,p_i$ and $q_i$ respectively. We
want to compute the action induced by $\Phi=g\circ \crem_3$ on the
cohomology group $ H_{\mbox{\tiny\textit{(C)}}}^2$. Equivalently, we
want to compute the matrix representing $\Phi_*$ 
with respect to the
basis \eqref{eq:H2caseB}. The action of the standard Cremona transformation on the elements $\varepsilon_{g}^*H,E_1,E_2,E_3,E_4$ agrees with \Cref{CREMONACTION}, while the elements $P_i,Q_i$,
for $i=1,\ldots,4$, are fixed by ${\crem_3}_*$ because they lie over the
fixed points of $\crem_3$. In terms of singular orbits \cite{BedfordKim2004} this implies that  there are four closed singular orbits of length three. 

Now, the linear map $g_*=(g^{-1})^*$ fixes
$\varepsilon_{g}^*H$ and it permutes the remaining elements of
the basis of the cohomology we have chosen.
As a consequence,
the matrix that represents $g_*$ with respect to the basis
\eqref{eq:H2caseB} has a block decomposition whose blocks correspond to permutations $\sigma_1,\sigma_2,\sigma_3\in\mathfrak{S}_4$, where $\mathfrak{S}_4$ denotes the symmetric group in four letters. 
In particular, the cyclic subgroup of
$\GL(H_{\mbox{\tiny\textit{(C)}}}^2,{\BZ})$ generated by $g_*$ induces
a transitive action on the set $\Set{\OOE,\OP,\OQ}$.

Summarising, the action of $\Phi_{*}=g_*\circ {\crem_3}_*$ on
$H_{\mbox{\tiny\textit{(C)}}}^2$ is \begin{equation}
	\label{eq:actionB} \begin{tikzcd}[row sep=tiny] \varepsilon_{g}^{*}H
		\arrow[mapsto]{r}{\Phi_*} & 3 \varepsilon_{g}^{*}H -
		2\sum_{j=1}^{4} P_{j},\\
  E_{i} \arrow[mapsto]{r}{\Phi_*}
		&\varepsilon_{g}^{*} H -
		\sum_{j\neq i}^{} P_{\sigma_{1}(j)}&\mbox{for }i=1,\ldots,4,\\ P_i
		\arrow[mapsto]{r}{\Phi_*} &Q_{\sigma_{2}(i)}&\mbox{for
		}i=1,\ldots,4,\\
		Q_{i} \arrow[mapsto]{r}{\Phi_*}& E_{\sigma_{3}(i)}&\mbox{for }i=1,\ldots,4,
	\end{tikzcd}
\end{equation} where $\sigma_1,\sigma_2,\sigma_3$ are the afore mentioned elements of
$\mathfrak{S}_4$ and correspond to the non-zero $4\times4$ blocks of the matrix
representing $g_*$ with respect to the basis in \eqref{eq:H2caseB}. The same action can be recovered from the four singular orbits using \cite[Eqs. (4.1, 4.3)]{BedfordKim2004}.

To conclude, we claim that the following formula holds for the map $\Phi=g\circ \crem_{3}\in\Bir( \BP^{3})$
    \begin{equation}
        \label{eq:caseBphinH} 
        (\Phi_{*})^{n}( \varepsilon_{g}^{*}H) 
            = d_{n} \varepsilon_{g}^{*}H - f_{n}\sum_{j=1}^{4} E_{j}
            - b_{n}\sum_{j=1}^{4} P_{j} - c_{n}\sum_{j=1}^{4} Q_{j},
    \end{equation} where the coefficients solve the following system of
    difference equations
    \begin{equation}
        \label{eq:dnen}
        \begin{gathered}
            d_{n}=d_n^\Phi = 3 d_{n-1} - 4 f_{n-1}, 
            \quad
            f_{n} = c_{n-1},
            \\
            b_{n} = 2 d_{n-1} - 3 f_{n-1},
            \quad
            c_{n} = b_{n-1},
        \end{gathered}
    \end{equation} 
    with initial conditions
    \begin{equation}
        \label{eq:ini} 
        d_{0}= 1, \,
        f_{0} = 0, \, 
        b_{0} = 0, \, 
        c_{0} = 0. \,
    \end{equation}
    This would imply that the map $\Phi$ has positive algebraic entropy given by
    \begin{equation}
        \label{eq:entropyB} 
        S_{\mbox{\tiny\textit{(C)}}} = 2 \log \varphi,
    \end{equation} 
    where $\varphi$ is the \emph{golden ratio}, i.e. the only positive solution 
    of the algebraic equation $\varphi^{2}=\varphi+1$.
    Thus, the map $\Phi$ is \emph{non-integrable} according to the algebraic entropy. 

  In order to prove the claim, we start by
    evaluating $\Phi_*$ on the sums $\sum_{j=1}^4E_j$, $\sum_{j=1}^4P_j$
    and $\sum_{j=1}^4Q_j$. Thanks to \Cref{eq:actionB}, we have
    \begin{equation}
        \label{evalB} 
        \begin{aligned}
            \Phi_{*}\left(\sum_{j=1}^{4}E_j\right) &=4\varepsilon_{g}^*H-3\sum_{j=1}^{4}P_j,&
            \\ \Phi_{*}\left(\sum_{j=1}^{4}P_j\right) &=
            \sum_{j=1}^{4}Q_{\sigma_2(j)}=\sum_{j=1}^{4}Q_{j},&
            \\ \Phi_{*}\left(\sum_{j=1}^{4}Q_j\right)
            &=
            \sum_{j=1}^{4}E_{\sigma_2(j)}=\sum_{j=1}^{4}E_{j}.
        \end{aligned}
    \end{equation} 
    We proceed now by induction on $n\ge 1$. The case $n=1$ is a direct
    computation. Suppose that \Cref{eq:caseBphinH} 
    is true for some $n\ge 1$ and we prove it for $n+1$. Then, we have
    \begin{equation}
            \label{eq:caseBphiHpcomput} \begin{aligned}
                    (\Phi_{*})^{n+1}( \varepsilon_{g}^{*}H )
                    &=\Phi_*\left[\left(\Phi_*\right)^n\left(\varepsilon_{g}^*H\right)\right]=
                    \\ &= \Phi_* \left[d_{n} \varepsilon_{g}^{*}H -
                    f_{n}\sum_{j=1}^{4} E_{j} - b_n\sum_{j=1}^{4}
                    P_{j} - c_n\sum_{j=1}^{4}
                    Q_{j}\right]=\\ &=\left(3d_n-4f_{n}
                    \right)\varepsilon_{g}^*H-c_n\sum_{j=1}^{4}E_j-\left( 2d_{n}-3f_{n}
                    \right)\sum_{j=1}^{4}P_j-b_n\sum_{j=1}^{4}Q_j,
            \end{aligned}
    \end{equation} where the third equality is a consequence of
    \Cref{evalB}.  On the other hand, we must have \begin{equation}
            \label{eq:caseBphinHp} 
            (\Phi_{*})^{n+1}( \varepsilon_{g}^{*}H ) = d_{n+1} \varepsilon_{g}^{*}H -
            f_{n+1}\sum_{j=1}^{4} E_{j} - b_{n+1}\sum_{j=1}^{4} P_{j} -
            c_{n+1}\sum_{j=1}^{4} Q_{j}.
    \end{equation} So, the condition is satisfied by equating
    the right-hand sides of \eqref{eq:caseBphiHpcomput} and
    \eqref{eq:caseBphinHp} and invoking the linear independence of
    the generators of $H^2(B_g,\BZ)$.  This implies that $d_n,f_n,b_n$ and
    $c_n$ satisfy the system \eqref{eq:dnen} with initial conditions
    \eqref{eq:ini}.
    
    In order to compute the algebraic entropy from \Cref{def:algent}, we
    need to evaluate the asymptotic behaviour of $d_{n}^{\Phi}$ in
    \Cref{eq:dnen}. Since the system \eqref{eq:dnen} is linear
    we use the technique explained in \cite[Chap. 3]{Elaydi2005}.
    Writing the system as
    \begin{equation}
        \begin{pmatrix}
            d_n
            \\
            f_n
            \\
            b_n
            \\
            c_n
        \end{pmatrix}
        =
        M_g
        \begin{pmatrix}
            d_{n-1}
            \\
            f_{n-1}
            \\
            b_{n-1}
            \\
            c_{n-1}
        \end{pmatrix}\!,
        \quad \mbox{ where }
        M_g =
        \begin{pmatrix}
            3&-4&0&0
            \\ 
            0&0&0&1
            \\
            2&-3&0&0
            \\
            0&0&1&0
        \end{pmatrix}\!,
    \end{equation}
    then the solution is
    \begin{equation}
        \begin{pmatrix}
            d_n
            \\
            f_n
            \\
            b_n
            \\
            c_n
        \end{pmatrix}
        =
        M_g^n
        \begin{pmatrix}
            d_{0}
            \\
            f_{0}
            \\
            b_{0}
            \\
            c_{0}
        \end{pmatrix}
        =
        M_g^n
        \begin{pmatrix}
            1
            \\
            0
            \\
            0
            \\
            0
        \end{pmatrix}\!.
    \end{equation}
    Computing $M_g^n$, e.g. using Putzer algorithm \cite[Sect. 3.1.1]{Elaydi2005},
    we obtain the following solution for all $n\in\BN$
    \begin{equation}
        \begin{pmatrix}
            d_n
            \\
            f_n
            \\
            b_n
            \\
            c_n
        \end{pmatrix}
        =
        \begin{pmatrix}
            \displaystyle
            \frac{8}{5}\left(\varphi^{2n}+\varphi^{-2n}\right)
            -\frac{1}{5}\left( -1 \right) ^{n}-2
            \\ 
            \noalign{\medskip}
            \displaystyle
            \frac{2}{5}\left(\varphi^{2n-2}+\varphi^{-2n+2}\right)
            -\frac{1}{5}\left( -1\right) ^{n}-1
            \\ 
            \noalign{\medskip}
            \displaystyle
            \frac{2}{5}\left(\varphi^{2n+2}+\varphi^{-2n-2}\right)
            -\frac{1}{5}\left( -1\right) ^{n}-1
            \\
            \noalign{\medskip}
            \displaystyle
            \frac{2}{5}\left(\varphi^{2n}+\varphi^{-2n}\right)
            -\frac{1}{5}\left( -1 \right) ^{n}-1
        \end{pmatrix}\!,
    \end{equation}
    where $\varphi$ is the golden ratio. 
    Since $d_n=d_n^\Phi$, we have
    \begin{equation}
        d_n^\Phi \sim \varphi^{2n}, \quad
        n\to\infty,
    \end{equation}
    and, from \eqref{eq:algentdef},
    formula \eqref{eq:entropyB} follows.
    
\end{proof}

\section{Rational surfaces associated with differential and discrete Painlev\'e equations }\label{sec:sec3}

In the remainder of these notes we survey the theory of rational surfaces associated with continuous and discrete Painlev\'e equations and the insights it brings into their integrability.
For the sake of completeness we begin with a brief overview of Painlev\'e equations and the mechanism by which they are associated to rational surfaces, but for a more complete historical account we refer to one of the existing standard references, e.g. \cite{painlevehandbook,riemannhilbertapproach,gromaklaineshimomura, gausstopainleve}.

The classical Painlev\'e equations are six non-linear non-autonomous second-order ordinary differential equations (ODEs) which were singled out as part of the program  of \textit{Painlev\'e} and then his student \textit{Gambier}, which aimed to obtain non-linear ODEs defining new special functions.
These are often denoted by $\pain{I}$-$\pain{VI}$, and can be written in the following forms:
\begin{align*}
\pain{I}:~ w'' &= 6 w^2 + t, \\
\pain{II}:~ w'' &= 2 w^3 + t w + \alpha, \\
\pain{III}:~ w'' &= \frac{ (w')^2}{w} - \frac{w'}{t} + \alpha \frac{w^2}{t} + \frac{\beta}{t}  + \gamma w^3 + \frac{\delta}{w} ,   \\
\pain{IV}:~ w'' &= \frac{1}{2w} \left( w' \right)^2 + \frac{3}{2} w^3 + 4 t w^2 + 2(t^2 - \alpha) w + \frac{\beta}{w}, \\
\pain{V}:~ w'' &= \left( \frac{1}{2w} + \frac{1}{w-1} \right) (w')^2 - \frac{w'}{t}  + \frac{(w-1)^2}{t^2} \left(\alpha w + \frac{\beta}{w} \right) + \gamma \frac{w}{t} + \delta \frac{w (w+1)}{w-1}, \\
\pain{VI}:~ w'' &= \frac{1}{2} \left( \frac{1}{w} + \frac{1}{w-1} + \frac{1}{w-t} \right) (w')^2 - \left( \frac{1}{t} + \frac{1}{t-1} + \frac{1}{w-t} \right)w' \\
&\quad \quad \quad \quad \quad \quad + \frac{w (w-1)(w-t)}{t^2(t-1)^2} \left( \alpha + \beta \frac{t}{w^2} + \gamma \frac{t-1}{(w-1)^2} + \delta \frac{t(t-1)}{(w-t)^2} \right)\!.
\end{align*}
In each case, $\alpha, \beta, \gamma, \delta$ are complex parameters and, for brevity, we put $w= w(t)$ and $'=\frac{d}{dt}$.

The idea of defining special functions in terms of solutions of ODEs which are non-linear is complicated by the fact that, in general, different solutions can exhibit branching at different points in the complex plane, which prevents one from considering all solutions  as single-valued functions on the same Riemann surface. While the locations of singularities of solutions of linear ODEs must belong to a set of points dictated by the coefficients of the equation, in the non-linear case there may be movable singularities, whose locations depend on the particular solution in question.
For example, consider the non-linear first-order ODE 
\begin{equation}
    w' + w^3=0.
\end{equation}
Any solution is given by $w(t) = \frac{1}{\sqrt{2(t-C)}}$ for some constant of integration $C \in\BC$, so solutions have branch points whose locations depend on $C$.
In order for there to be a notion of general solution of a non-linear ODE as a function on a single Riemann surface and a way to define special functions in terms of it, one imposes the condition that all solutions are single-valued around all movable singularities. 
Painlev\'e \cite{painleve} set out to classify ODEs of the form $w''=R(w,w',t)$, 
with $R$ rational in $w,w'$ and meromorphic in $t$, subject to this condition, now known as the Painlev\'e property, up to changes of dependent and independent variables by M\"obius transformations.
The equations $\pain{I}$-$\pain{VI}$ arose as representatives of classes of equations whose solutions in general cannot be reduced to those of linear ODEs, and in this sense define new special functions called the \textit{Painlev\'e transcendents} \cite{NIST}.

\begin{remark}
    Note that $\pain{VI}$ was missed in Painlev\'e's initial attempt at classification \cite{painleve}.
    The sixth Painlev\'e equation was derived by \textit{Richard Fuchs} as the equation governing a monodromy-preserving deformation of a linear system of two first-order ODEs with four regular singular points on $\BP^1$ \cite{fuchs} and was added to the classification by Gambier when he completed the list\cite{gambier}. 
    Note also that a special case of $\pain{VI}$ appeared earlier in the work of \textit{Picard} \cite{Picard1889}. 
\end{remark}

The term `discrete Painlev\'e equation' appeared in the literature for the first time in a paper of \textit{Its, Kitaev} and \textit{Fokas} \cite{ItsKitaevFokas}, see also \cite{FokasItsKitaev}.
After this, there were many efforts by researchers in integrable systems to derive discrete analogues of the Painlev\'e differential equations, and many early examples were found via discrete counterparts to ways that the Painlev\'e differential equations can be derived.
These include through reductions of integrable partial-difference equations, by analogy with the relation of Painlev\'e equations to integrable partial differential equations (e.g. \cite{NijhoffPapageourgiou}), and through discrete versions of isomonodromic deformations (e.g. \cite{isomonodromicdPs, JimboSakai}).
There was also the proposal, by \textit{Grammaticos, Ramani} and \textit{Papageorgiou}, of \emph{singularity confinement} as a discrete counterpart to the Painlev\'e property \cite{singularityconfinement}.
Procedures based on this were used to great effect (see \cite{dPsreview} and references therein) to obtain discrete Painlev\'e equations by de-autonomisation of discrete systems solved by elliptic functions, namely \textit{Quispel-Roberts-Thompson (QRT)} maps \cite{QRT1, QRT2}.

In these notes, we  present the definition of a discrete Painlev\'e equation in the sense of the \textit{Sakai} framework. 
This was proposed in the seminal paper \cite{Sakai2001}, and has since formed the basis for many insights into discrete and differential Painlev\'e equations and their properties, as surveyed in \cite{KNY}.
Sakai's approach takes cues from the construction, by \textit{Okamoto}, of \emph{spaces of initial conditions} for the Painlev\'e differential equations. 
This involves using techniques from classical algebraic geometry, see \Cref{sec:AG} for some classical constructions, to obtain an augmented phase space for a Hamiltonian form of a Painlev\'e equation on which it is regularised, in a sense which we will explain in \Cref{subsec:okamotospace}.
As part of the construction, there appears a smooth projective rational surface with an effective divisor given by a collection of curves in a configuration related to an affine root system.

For discrete Painlev\'e equations, Sakai recognised singularity confinement as reflecting an analogous regularisability of discrete systems defined by birational mappings, and introduced a class of surfaces associated with affine root systems, which generalise those appearing in Okamoto's work.
Discrete Painlev\'e equations are then defined in terms of these surfaces, in a way such that they provide their spaces of initial conditions.

\subsection{Okamoto's spaces of initial conditions for Painlev\'e equations}

\label{subsec:okamotospace}

The Painlev\'e differential equations can be associated with a special class of rational surfaces. 
This association is via a certain resolution of singularities of the foliations defined by their flows, and goes back to the work of Okamoto \cite{Okamoto1979}.
 Denote by $\mathcal{T}   \subset\BC$ the  complement of the finite set of isolated fixed singularities defined by the coefficients in the equation $\pain{\bullet}$, for some $\bullet\in\Set{\mathrm{I},\ldots,\mathrm{VI}}$. 
 The fact that all solutions are single-valued about all movable singularities means that any local analytic solution $w(t)$ at $t=t_* \in \mathcal{T}$ can be meromorphically continued along any path $\ell$ in $\mathcal{T}$ with starting point $t_*$. Explicitly, we have $\mathcal{T}=\BC\setminus \{0,1\}$ for $\pain{VI}$,  $\mathcal{T}=\BC\setminus \{0\}$ for $\pain{III}$, $\pain{V}$, and $\mathcal{T}=\BC$ for $\pain{I}$, $\pain{II}$, $\pain{IV}$. 

Given $\pain{\bullet}$, for $\bullet\in \Set{\mathrm{I},\ldots,\mathrm{VI}}$, Okamoto considered an equivalent form of the equation as a non-autonomous Hamiltonian system
\begin{equation} \label{eq:hamiltonianHJ}
    q' = \frac{ \partial \operatorname{H}_{\bullet}}{\partial p}, \quad 
    p' = - \frac{ \partial \operatorname{H}_{\bullet}}{\partial q},
\end{equation}
where $\mathrm{H}_{\bullet}=\mathrm{H}_{\bullet}(q,p,t)$ is polynomial in $q,p$ with coefficients being rational functions of $t$ analytic in $\mathcal{T}$.
Equation \eqref{eq:hamiltonianHJ} can be interpreted as a rational vector field which is everywhere regular on $\BC^2\times\mathcal{T}$, and we have existence and uniqueness of local analytic solutions.
The Painlev\'e property of $\pain{\bullet}$ translates, in this setting, to the fact that any local solution $(q(t),p(t))$ near $t=t_*$ can be meromorphically continued in $\BC^2\times\mathcal{T}$ along any path in $\mathcal{T}$ with starting point $t_*$. 
 However, since solutions can develop (movable) poles, we cannot globally analytically continue solutions along paths in $\mathcal{T}$. This gives rise to the need to compactify appropriately the $\BC^2$ fibres in order to give a parametrisation of the space of solutions. For instance, Okamoto used a compactification isomorphic to a Hirzebruch surface; more examples can be found in \cite{3dpd}. Doing so, one obtains a trivial bundle over $\mathcal{T}$ with compact fibres, in which solutions can be globally analytically continued. 
However, the vector field ceases to be regular on the part of the fibres added in the compactification process, and there may be infinitely many solutions passing through the same point in the fibre at the same $t\in \mathcal{T}$.
To resolve this, Okamoto performed an appropriate sequence of blow-ups of the compactified fibre (of possibly $t$-dependent points) to separate such solutions.

Then, the flow of the system extended to the resulting space defines a foliation of it into disjoint complex one-dimensional leaves. 
Some of these will be vertical with respect to the bundle structure over $\mathcal{T}$, i.e. contained in a single fibre. 
Removing these \emph{vertical leaves} yields a
triple $(E,\pi,\mathcal{T})$ such that the flow of the Hamiltonian system \eqref{eq:hamiltonianHJ} defines a foliation with properties that we collect in the following definition.
\begin{definition} \label{def:uniformfoliation}
    Consider a triple $(E,\pi,\mathcal{T})$, with $E$ a smooth quasi-projective variety and $\pi : E\rightarrow \mathcal{T}$ a surjective morphism, such that $E$ contains an open subset $\mathcal{T}$-isomorphic to $\BC^2\times\mathcal{T}$.
    The flow of the Hamiltonian system \eqref{eq:hamiltonianHJ} extended to $E$ defines a (nonsingular) \emph{uniform foliation} $\mathcal{F}$ of $E$ into complex-analytic 1-dimensional leaves transverse to the fibres if
        \begin{itemize}
        \item each leaf of $\mathcal{F}$ intersects every fibre $E_t=\pi^{-1}(t)$, for $t\in \mathcal{T}$, transversally; 
        \item for any path $\ell$ in $\mathcal{T}$ with starting point $t_*  \in \mathcal{T}$ and any point $p\in E_{t_* }$, the path $\ell$ can be lifted to the leaf passing through $p$.
    \end{itemize} 
\end{definition}

\begin{remark}
    Note that for Painlev\'e equations with fixed singularities, namely $\pain{III}$, $\pain{V}$ and $\pain{VI}$ a leaf of $\mathcal{F}$ may intersect a fibre $E_t$ at more than one point due to the branching that can occur when solutions are continued around fixed singularities. This phenomenon  can be regarded as a kind of non-linear monodromy.
\end{remark}
\begin{remark}
    While the space $E$ constructed from each Painlev\'e equation is a quasi-projective variety and $\pi : E \rightarrow \mathcal{T}$ is a surjective morphism, $E$ carries the structure of a complex analytic fibre bundle over $\mathcal{T}$, but not an algebraic one.
    This is because the isomorphisms between different fibres are given by the flow of the Painlev\'e equation, which is generally transcendental.
\end{remark}
Every point in the fibre over $t_*\in \mathcal{T}$ determines a solution that can be continued along any path starting from $t_*$. 
Each fibre is called a \emph{space of initial conditions} for the corresponding Painlev\'e equation, and this is the origin of the terminology in the literature for spaces of initial conditions in the discrete case as adopted in \Cref{def:spaceofinitialcond}.
For each of the Painlev\'e equations, before the removal of vertical leaves, each fibre of the projection onto $\mathcal{T}$ is a smooth rational projective surface $S$.
The vertical leaves give a collection of curves in each fibre, which are the irreducible components of an effective anti-canonical divisor $D\in\Div(S)$ on the surface. 
These curves have intersection configuration encoded by an (extended) Dynkin diagram associated with an affine root system, which will be introduced in \Cref{subsubsec:affinerootsystems}, see \Cref{E8dynkin} for an example.
In this setting, vertices correspond to irreducible components of the divisor $D$, with two vertices joined by an edge if the corresponding curves intersect.
These are labelled by their types in \Cref{table:surfacetypes}, and the $E_8^{(1)}$ diagram in \Cref{E8dynkin} corresponds to $\pain{I}$.
The three different types for $\pain{III}$ correspond to generic and degenerate cases, each associated with a different type of surface in Sakai's classification,  see \cite{StudiesV}.

\renewcommand{\arraystretch}{1.5}
\begin{table}[htb]
\centering
    \begin{tabular}{  | c | c | c c c  | c | c | c |}
    $\pain{I}$ & $\pain{II}$ & ~& $\pain{III}$ &~& $\pain{IV}$ & $\pain{V}$ & $\pain{VI}$ \\ \hline   
   $E_8^{(1)}$ & $E_7^{(1)}$ & $D_8^{(1)}$ & $D_7^{(1)}$ & $D_6^{(1)}$ & $E_6^{(1)}$ & $D_5^{(1)}$ & $D_4^{(1)}$   
    \end{tabular}
    \caption{Dynkin diagrams from intersection configuration of vertical leaves for differential Painlev\'e equations.}
   \label{table:surfacetypes}
   \end{table}
   \renewcommand{\arraystretch}{1}
   
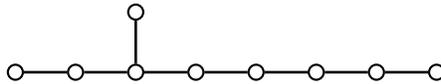
\begin{figure}[htb]
\centering
    \begin{tikzpicture}[
 			elt/.style={circle,draw=black!100,thick, inner sep=0pt,minimum size=2mm},scale=.8]
 		\path 	(-2,0) 	node 	(d1) [elt] {}
                    (-1,0) 	node 	(d2) [elt] {}
                    (0,0) 	node 	(d3) [elt] {}
                    (1,0) 	node 	(d4) [elt] {}
                    (2,0) 	node 	(d5) [elt] {}
                    (3,0) 	node 	(d6) [elt] {}
                    (4,0) 	node 	(d7) [elt] {}
                    (5,0) 	node 	(d0) [elt] {}
                    (0,1) 	node 	(d8) [elt] {}
 		       ;
 		\draw [black,line width=1pt ] (d1) -- (d2) -- (d3) -- (d4) -- (d5) -- (d6) -- (d7) -- (d0)  (d3) -- (d8);
 	\end{tikzpicture}
\caption{(Extended) Dynkin diagram $E_8^{(1)}$ giving intersection configuration of vertical leaves for $\pain{I}$.}
\label{E8dynkin}
\end{figure}

\begin{remark}
    Okamoto's space is not the only way in which the Painlev\'e equations are associated to algebraic surfaces. 
    Another way is via \emph{monodromy manifolds} for associated linear systems, each of which can be realised as an affine cubic surface, i.e. the vanishing locus in $\BA^3$ of a single cubic polynomial with coefficients given in terms of parameters from the corresponding Painlev\'e equation.
    These realisations can be found in \cite{vanderputsaito, chekhovmazzoccorubtsov}. 
    The affine cubic surfaces are related to the spaces of initial conditions by instances of the Riemann-Hilbert correspondence, but this is in general accepted to be transcendental and to lie outside the class of isomorphisms of varieties.
\end{remark}

\subsection{Space of initial conditions for an ODE solved by elliptic functions}
\label{example:wp} 

    We show the calculations involved in the construction of a space of initial conditions in the case of a second-order ODE with the Painlev\'e property.
    For simplicity we consider an 
    autonomous version of $\pain{I}$ whose solutions are given in terms of the Weierstrass elliptic function $\wp$, which we define now, following \cite[pg. 327]{Hartshorne77}, see also \cite[Ch. 23]{DLMF}.
   \begin{definition} \label{def:weierstrassp}
        Let $\Lambda=\BZ+\BZ\tau\subset \BC$, for some $\tau\in\BC\setminus \BR$, and also let $\Lambda'=\Lambda\setminus\Set{0}$.
        The \textit{Weierstrass }$\wp$-\textit{function} associated to these data  is
        \begin{equation}
            \wp(t;g_2,g_3) = \frac{1}{t^2} + \sum_{\omega\in\Lambda'} \left( \frac{1}{(t-\omega)^2} - \frac{1}{\omega^2}\right),
        \end{equation}
        where $g_2,g_3$ are parameters defined by
        \begin{equation}
            g_2 = 60 \sum_{\omega\in\Lambda'}\frac{1}{\omega^4} ,\qquad g_3 = 140 \sum_{\omega\in\Lambda'} \frac{1}{\omega^6}.
        \end{equation}
        As a function of $t$, $\wp$ is a $\Lambda$-periodic meromorphic function, as is its derivative $\wp'$. 
        Then $q(t)=\wp(t;g_2,g_3)$ solves the first-order ODE
    \begin{equation} \label{eq:weierstrassfirstorder}
        (q')^2 = 4 q^3 - g_2 q - g_3.
    \end{equation} 
    \end{definition}
    Differentiating \Cref{eq:weierstrassfirstorder} with respect to $t$ leads to the autonomous second-order ODE 
    \begin{equation} \label{eq:autonomousP1scalar}
        q'' = 6 q^2 - \frac{g_2}{2}.
    \end{equation}
    We regard $g_2$ as a complex parameter in this ODE, which we consider as fixed for the remainder of this section.
     \begin{lemma} \label{lem:wpexpansion}
         \Cref{eq:autonomousP1scalar} has the Painlev\'e property and all local solutions can be extended to meromorphic functions on $\BC$, given, for fixed $g_2\in\BC$, by $q(t)=\wp(t-c;g_2,g_3)$ for some $c,g_3\in\BC$. 
       Further, if a solution $q(t)$ of \Cref{eq:autonomousP1scalar} fails to be analytic at some $t=t_*\in\BC$, then $q(t)$ has a pole of order 2 at $t_*$ and is given in a neighbourhood of $t_*$ by a Laurent expansion
     \begin{equation} \label{eq:laurentexpansion}
         q(t) = \frac{1}{(t-t_*)^2} + \frac{g_2}{20} (t-t_*)^2 + \mu(t-t_*)^4 + \frac{g_2^2}{1200}(t-t_*)^6 + \cdots,
     \end{equation}
     for some $\mu\in\BC$.
     \end{lemma} 
    \Cref{eq:autonomousP1scalar} can be written as the (autonomous) system of first-order ODEs
    \begin{equation} \label{eq:autonomousP1system}
        q' = p, \qquad p' = 6 q^2 - \frac{g_2}{2}.
    \end{equation}
    Consider this first as a rational vector field on $\BC^2 \times \mathcal{T}$, with $\mathcal{T}=\BC$, by taking $q,p$ as coordinates for $\BC^2$. 
    This vector field is holomorphic on $\BC^2 \times \mathcal{T}$, and local existence and uniqueness theorems for ODEs ensure that the flow of \Cref{eq:autonomousP1system} defines a foliation of  $\BC^2 \times \mathcal{T}$, with leaves being disjoint solution curves, transverse to the fibre over each $t\in\mathcal{T}$. 
    However, solutions having poles at $t=t_*$ as in \Cref{lem:wpexpansion} means that paths in $\mathcal{T}$ cannot be globally lifted to leaves. 
    
So we compactify the fibres to $\BP^2$ in the following way. We fix homogeneous coordinates $[x_0:x_1:x_2]$ on $\BP^2$ and we identify the coordinate chart $\mathcal{U}_0$ with $\BC^2$ according to $[1:q:p]=[x_0:x_1:x_2]$.  Finally, we consider the system \eqref{eq:autonomousP1system} extended to a rational vector field on $\BP^2 \times \mathcal{T}$, see \Cref{rmk:coorchart}. 
    Then, any local solution at $t_* \in \mathcal{T}$ can be continued to a holomorphic embedding of $\mathcal{T}$ into $\BP^2 \times \mathcal{T}$. 
    We call the images of these embeddings \textit{solution curves}.
    The fact that there are infinitely many solutions with a pole at the same $t=t_*$ as in \Cref{lem:wpexpansion} means that the solution curves are not disjoint and they do not define a foliation of $\BP^2 \times \mathcal{T}$.
    We will perform explicit calculations to verify the following.
    \begin{prop} \label{prop:wpspaceofICs} There are a  birational morphism $\varepsilon : S \rightarrow \BP^2$, where $S$ is a smooth projective rational surface, and a hypersurface of $S$ given by a union of irreducible curves $D_i$, $i=0,\dots,8$, such that the flow of the system \eqref{eq:autonomousP1system} defines a foliation $\mathcal{F}$ of the total space of the trivial bundle $E = (S\setminus \cup_{i=0}^{8} D_i) \times \mathcal{T} \rightarrow \mathcal{T}$  as in \Cref{def:uniformfoliation}.
        Further, the morphism $\varepsilon $ is a composition of nine blow-ups each centred at a point. 
    \end{prop}
\begin{proof}
    We will construct the surface $S$ by performing blow-ups of $\BP^2$ to separate the solution curves passing through the same point in the fibre over $t_*\in\mathcal{T}$, corresponding to Laurent expansions \eqref{eq:laurentexpansion} with different values of $\mu$.
    First introduce notation for the coordinate atlas for $\BP^2$, see \Cref{rmk:coorchart}, as follows
    \begin{equation}
        \begin{aligned}
         \mathcal{U}_0 \cong \BA^2_{(q,p)}, \quad &[x_0:x_1:x_2]=[1:q:p], \\
         \mathcal{U}_1\cong \BA^2_{(x,y)}, \quad &[x_0:x_1:x_2]=[x:1:y], \\
         \mathcal{U}_2\cong \BA^2_{(z,w)},  \quad &[x_0:x_1:x_2]=[z:w:1],
        \end{aligned}
    \end{equation}
    so we think of $\BP^2$ as three copies of $\BA^2$ glued together according to transition maps on the overlaps $\mathcal{U}_i\cap \mathcal{U}_j$, for $i,j\in\Set{0,1,2}$, which we write (formally) as the equalities
    \begin{equation}
        q = \frac{1}{x} = \frac{w}{z}, \quad p = \frac{y}{x} = \frac{1}{z}.
    \end{equation}
    Note that the expansions \eqref{eq:laurentexpansion} correspond to solution curves passing through the point 
    \begin{equation}
        b_1 = [0:0:1] \in \BP^2.
    \end{equation}
    In the affine chart $\mathcal{U}_2$, having $b_1$ as the origin, the expansion \eqref{eq:laurentexpansion} reads as follows:
    \begin{equation} \label{eq:laurentzw}
        z(t) = -\frac{1}{2}(t-t_*)^3 - \frac{g_2}{40}(t-t_*)^7+\cdots, \quad w(t) = -\frac{1}{2}(t-t_*)- \frac{g_2}{20}(t-t_*)^5+\cdots.
    \end{equation}
    Note we have only shown the first few terms in the expansions \eqref{eq:laurentzw}, but all of them can be recursively computed.
    The constant $\mu$ appears only in later terms, and we keep track of further coefficients in the calculations that follow.

    Our strategy now is to perform blow-ups over $b_1$ to separate the solution curves corresponding to solutions with different values of $\mu$ in their Laurent expansions at a pole at the same $t_*$.

    Recall from \Cref{exa:plow1} that the blow-up of $\BA^2$ at a point can be realised as a surface embedded in $\BA^2\times\BP^1$. Moreover, the exceptional line is contained in two affine charts which we have described in the example. We denote here by $\CW^{(i)}_j$, for $i=0,1$ and $j=1,\ldots ,9$, the two charts covering the $j$-th exceptional line which we denote by $E_j$. Explicitly, we have

    \begin{equation}
        \begin{aligned}
            \CW_1^{(0)} &\cong \BA^2_{(U_1,V_1)},\\
            \CW_1^{(1)}&\cong \BA^2_{(u_1,v_1)},  
        \end{aligned} \quad 
        \pi_{b_1} : \left\{
        \begin{aligned} 
            (U_1,V_1) &\mapsto (z,w) = (V_1, U_1 V_1),\\
            (u_1,v_1) &\mapsto (z,w) = (u_1 v_1, v_1).
        \end{aligned}   
        \right.
    \end{equation}
     
    Note that the local equations of the exceptional curve $E_1$ in the coordinates $(u_1,v_1)$ and $(U_1,V_1)$ are $v_1=0$ and $V_1=0$, respectively.
We give a pictorial  representation of these charts on \Cref{fig:blow-upb1}.

   \begin{figure}[!ht]
\centering
    \begin{tikzpicture}[scale=.5,basept/.style={circle, draw=red!100, fill=red!100, thick, inner sep=0pt,minimum size=1.2mm}]
    \node at (0,-3) {$\mathbb{P}^2$};
    \draw[thick,black] (-2.5,-2) -- (+2.5,-2) node [pos=0,left] {\tiny $x_2=0$} ;
    \draw[thick,black] (-2.25,-2.5) -- (+.3,+2) node [pos=0,below,black] {\tiny $x_1=0$} ;
    \draw[thick,black] (+2.25,-2.5) -- (-.3,+2) node [pos=1,left,black] {\tiny $x_0=0$} ;
    \node (b1) at (0,+1.475) [basept,label={[xshift=12pt, yshift = -10 pt] \small $b_{1}$}] {};
    \draw[->,thick, opacity=.8,magenta] (b1) -- (.55,+.5) node[pos=1,right] {\tiny $w$};
    \draw[->,thick, opacity=.8,magenta] (b1) -- (-.55,+.5) node[pos=1,left] {\tiny $z$};
    \draw[->,thick, opacity=.8,magenta] (-1.96,-2) -- (-1.54,-1.25) node[pos=1,above left] {\tiny $p$};
    \draw[->,thick, opacity=.8,magenta] (-1.96,-2) -- (-1.16,-2) node[pos=1,below] {\tiny $q$};
    \draw[->,thick, opacity=.8,magenta] (1.96,-2) -- (1.54,-1.25) node[pos=1,above right] {\tiny $y$};
    \draw[->,thick, opacity=.8,magenta] (1.96,-2) -- (1.16,-2) node[pos=1,below] {\tiny $x$};
    \draw[thick,black,->] (+4.5,0)--(+2.5,0) node[ midway,above] {$\pi_{b_1}$};

        \begin{scope}[xshift=+8cm]
    \node at (0,-3) {$\Bl_{b_1} \BP^2$};
    \draw[thick,black] (-2.5,-2)  -- (+2.5,-2)  node[pos=0,left] {} ;
    \draw[thick,red] (-2.5,+2) -- (+2.5,+2) node[pos=0,left,black] {} node[pos=0.5,above,red] {\small $E_1$};
    \draw[thick,black] (-2,-2.5) -- (-2,+2.5) node[pos=0,below] {} ;
    \draw[thick,black] (+2,-2.5) -- (+2,+2.5)node[pos=0,below,black] {} ;
    \draw[->,thick, opacity=.8,magenta] (2,2) -- (2,1.4) node[pos=1,right] {\tiny $v_1$};
    \draw[->,thick, opacity=.8,magenta] (2,2) -- (1.4,2) node[pos=1,above] {\tiny $u_1$};
    \draw[->,thick, opacity=.8,magenta] (-2,2) -- (-2,1.4) node[pos=1,right] {\tiny $V_1$};
    \draw[->,thick, opacity=.8,magenta] (-2,2) -- (-1.4,2) node[pos=1,above] {\tiny $U_1$};   
    \draw[->,thick, opacity=.8,magenta] (-2,-2) -- (-2,-1.4) node[pos=1,right] {\tiny $p$};
    \draw[->,thick, opacity=.8,magenta] (-2,-2) -- (-1.4,-2) node[pos=1,above] {\tiny $q$};   
    \draw[->,thick, opacity=.8,magenta] (2,-2) -- (2,-1.4) node[pos=1,right] {\tiny $y$};
    \draw[->,thick, opacity=.8,magenta] (2,-2) -- (1.4,-2) node[pos=1,above] {\tiny $x$};
        \end{scope}

\end{tikzpicture}
\caption{Coordinate charts for the blow-up of $\BP^2$ at $b_1$.}
\label{fig:blow-upb1}
\end{figure}
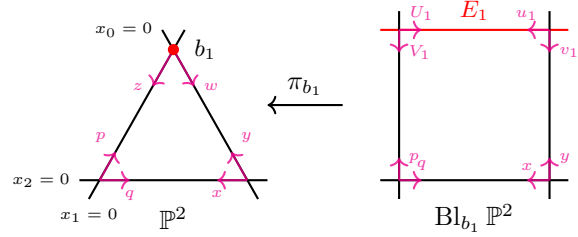 

    Lifting the solution curves given by expansions \eqref{eq:laurentzw} to $(\Bl_{b_1}\BP^2) \times \mathcal{T}$, these are written in the chart $\CW_1^{(1)}$ as
    \begin{equation}
        u_1(t) = (t-t_*)^2 - \frac{g_2}{20} (t-t_*)^2 - \mu (t-t_*)^8 + \cdots, \quad v_1(t) = -\frac{1}{2}(t-t_*) - \frac{g_2}{20}(t-t_*)^5 + \cdots .
    \end{equation}
    Note that all of them still pass through a single point in the fibre over $t_*$, given explicitly in coordinates by
    \begin{equation}
        b_2 : (u_1,v_1)=(0,0),  \quad b_2 \in E_1 \subset \Bl_{b_1}\BP^2.
    \end{equation}
    Note also that $b_2$ lies at the intersection of $E_1$ with the strict transform of the coordinate axis $V(x_0)\subset \BP^2$ under $\pi_{b_1}$, see \Cref{subsec:ratmap}. By blowing-up $b_2$, we introduce the charts 
    \begin{equation}
        \begin{aligned}
            \CW_2^{(0)} &\cong \BA^2_{(U_2,V_2)},\\
            \CW_2^{(1)}&\cong \BA^2_{(u_2,v_2)},  
        \end{aligned} \quad 
        \pi_{b_2} : \left\{
        \begin{aligned}
            (U_2,V_2) &\mapsto (u_1,v_1) = (V_2, U_2 V_2),\\
            (u_2,v_2) &\mapsto (u_1,v_1) = (u_2 v_2, v_2).
        \end{aligned}   
        \right.
    \end{equation}
    The lifts of the solution curves under $\pi_{b_2}$ still all pass through a single point
    \begin{equation}
        b_3 : (u_2,v_2)=(0,0),  \quad b_3 \in E_2 \subset \Bl_{b_2}\Bl_{b_1}\BP^2,
    \end{equation} 
    which lies at the intersection of $E_2$ and the strict transform of the curve locally given by $V(u_1)\subset \CW_1^{(1)}\subset \Bl_{b_1}\BP^2$ under $\pi_{b_2}$, which is the same as the strict transform of $V(x_0)\subset\BP^2$ under $\pi_{b_1}\circ \pi_{b_2}$.
    Continuing this way, we perform in total 9 blow-ups of $\BP^2$, with points and coordinate charts listed in \Cref{fig:blow-updatawp}, and obtain the surface and the morphism 
    \begin{equation*}
        S=\Bl_{b_9}\cdots \Bl_{b_1}\BP^2, \qquad \varepsilon = \pi_{b_1}\circ\cdots\circ \pi_{b_9}: S \rightarrow \BP^2.
    \end{equation*} 

\begin{table}[!ht]
    \centering
    {\small \begin{tabular}{c|c}
        $b_i$ & $\pi_{b_i}$ \\\hline 
        $\CU_2 \ni b_1:(z,w)=(0,0)$\hspace{0.5cm} & $\begin{tikzcd}[row sep =tiny]
            \CW_1^{(0)} \ni (U_1,V_1)\arrow[r,mapsto]& (V_1, U_1 V_1)\in\CU_2\\[-1em]
            \CW_1^{(1)} \ni(u_1 , v_1)\arrow[r,mapsto]& (u_1 v_1, v_1)\in\CU_2
        \end{tikzcd}$ \\\hline 
        $E_1 \ni b_2:(u_1,v_1)=(0,0)$\hspace{0.5cm}  & $\begin{tikzcd}[row sep =tiny]
            \CW_2^{(0)} \ni (U_2,V_2)\arrow[r,mapsto]& (V_2, U_2 V_2)\in\CW_1^{(1)}\\[-1em]
            \CW_2^{(1)} \ni(u_2 , v_2)\arrow[r,mapsto]& (u_2 v_2, v_2)\in\CW_1^{(1)}
        \end{tikzcd}$ \\\hline 
        $E_2 \ni b_3:(u_2,v_2)=(0,0)$\hspace{0.5cm}  & $\begin{tikzcd}[row sep =tiny]
            \CW_3^{(0)} \ni (U_3,V_3)\arrow[r,mapsto]& (V_3, U_3 V_3)\in\CW_2^{(1)}\\[-1em]
            \CW_3^{(1)} \ni(u_3 , v_3)\arrow[r,mapsto]& (u_3 v_3, v_3)\in\CW_2^{(1)}
        \end{tikzcd}$ \\\hline 
        $E_3 \ni b_4:(u_3,v_3)=(4,0)$\hspace{0.5cm}  & $\begin{tikzcd}[row sep =tiny]
            \CW_4^{(0)} \ni (U_4,V_4)\arrow[r,mapsto]& (4+V_4, U_4 V_4)\in\CW_3^{(1)}\\[-1em]
            \CW_4^{(1)} \ni(u_4 , v_4)\arrow[r,mapsto]& (4+u_4 v_4, v_4)\in\CW_3^{(1)}
        \end{tikzcd}$ \\\hline 
        $E_4 \ni b_5:(u_4,v_4)=(0,0)$\hspace{0.5cm}  & $\begin{tikzcd}[row sep =tiny]
            \CW_5^{(0)} \ni (U_5,V_5)\arrow[r,mapsto]& (V_5, U_5 V_5)\in\CW_4^{(1)}\\[-1em]
            \CW_5^{(1)} \ni(u_5 , v_5)\arrow[r,mapsto]& (u_5 v_5, v_5)\in\CW_4^{(1)}
        \end{tikzcd}$ \\\hline 
        $E_5 \ni b_6:(u_5,v_5)=(0,0)$\hspace{0.5cm}  & $\begin{tikzcd}[row sep =tiny]
            \CW_6^{(0)} \ni (U_6,V_6)\arrow[r,mapsto]& (V_6, U_6 V_6)\in\CW_5^{(1)}\\[-1em]
            \CW_6^{(1)} \ni(u_6 , v_6)\arrow[r,mapsto]& (u_6 v_6, v_6)\in\CW_5^{(1)}
        \end{tikzcd}$ \\\hline 
        $E_6 \ni b_7:(u_6,v_6)=(0,0)$\hspace{0.5cm}  & $\begin{tikzcd}[row sep =tiny]
            \CW_7^{(0)} \ni (U_7,V_7)\arrow[r,mapsto]& (V_7, U_7 V_7)\in\CW_6^{(1)}\\[-1em]
            \CW_7^{(1)} \ni(u_7 , v_7)\arrow[r,mapsto]& (u_7 v_7, v_7)\in\CW_6^{(1)}
        \end{tikzcd}$ \\\hline 
        $E_7 \ni b_8:(u_7,v_7)=(-16g_2,0)$\hspace{0.5cm}  & $\begin{tikzcd}[row sep =tiny]
            \CW_8^{(0)} \ni (U_8,V_8)\arrow[r,mapsto]& (-16g_2+V_8, U_8 V_8)\in\CW_7^{(1)}\\[-1em]
            \CW_8^{(1)} \ni(u_8 , v_8)\arrow[r,mapsto]& (-16g_2+u_8 v_8, v_8)\in\CW_7^{(1)}
        \end{tikzcd}$ \\\hline 
        $E_8 \ni b_9:(u_8,v_8)=(0,0)$\hspace{0.5cm}  & $\begin{tikzcd}[row sep =tiny]
            \CW_9^{(0)} \ni (U_9,V_9)\arrow[r,mapsto]& (V_9, U_9 V_9)\in\CW_8^{(1)}\\[-1em]
            \CW_9^{(1)} \ni(u_9 , v_9)\arrow[r,mapsto]& (u_9 v_9, v_9)\in\CW_8^{(1)}
        \end{tikzcd}$  
    \end{tabular}}
    \caption{Blow-up data for the space of initial conditions for system \ref{eq:autonomousP1system}.}
    \label{fig:blow-updatawp}
\end{table}
 
        Lifting the solution curves to $S\times\mathcal{T}$, we see they are represented in the chart $\CW_9^{(0)}\subset S$ by the series
    \begin{equation}
        u_9(t) = -1792 \mu +32 g_2^2(t-t_*)^2 + \cdots, \quad v_9(t) = -\frac{1}{2}(t-t_*) - \frac{g_2}{20}(t-t_*)^5+\cdots.
    \end{equation}
    So, in particular the solution curve with a pole at $t_*$ and parameter $\mu$ intersects the fibre over $t_*$ at the point in $E_9$ whose coordinates are $(u_9,v_9)=(-1792 \mu,0)$. Thus, we have separated the family of solution curves passing through $b_1\in\BP^2$ in the fibre over $t_*$.
    
    The final step in verifying that we have a space of initial conditions is  to determine and remove the vertical leaves, in order to achieve transversality  of the intersections of leaves and fibres.
    Consider the exceptional curves $E_1,\ldots,E_9$ on $S$, which we give a pictorial representation of in \Cref{fig:surfaceautonomousP1}. 
    Note that we have slightly abused notation in denoting by $E_i$ both the exceptional curve on $\Bl_{b_i}\cdots\Bl_{b_1}\BP^2$ of the blow-up of $b_i$ and its pull-back by any further blow-ups of $b_{i+1},\cdots ,b_9$, which is a divisor on $S$. We also denote by $H=\varepsilon^*V(x_0) \in \Div(S)$ the pull-back of the hyperplane $V(x_0)\subset \BP^2$.

    The following curves on $S$ correspond to the vertical leaves, i.e. none of the solution curves corresponding to the family \eqref{eq:laurentexpansion} pass through  them:
    \begin{equation} \label{eq:Denum}
        D_0 =  E_8-E_9, \quad D_i = E_i - E_{i+1}, ~~(i=1,\ldots,7), \quad D_8 = H- E_1-E_2-E_3.
    \end{equation}
    These form the irreducible components of an effective anti-canonical divisor 
    \begin{equation} \label{eq:Dwp}
        D :=D_0+2D_1 +4D_2+6D_3+5D_4+ 4D_5+ 3D_6+2 D_7+3 D_8 \in |-K_S|,
    \end{equation}
    since from the blow-up formula in \Cref{th:blow-upformula}, the anti-canonical class of $S$ is 
    \begin{equation} \label{eq:ACdivP2blownup}
        -K_S = 3H - E_1-E_2-\cdots-E_8-E_9,
    \end{equation}
    in which we have used the same symbol to denote divisors and their classes, c.f. \Cref{not:linearequiv}.
    By removing $\bigcup_{i=0}^8 D_i$ from $S$, we obtain the bundle $E$ with foliation $\mathcal{F}$ as required.
    \end{proof}

       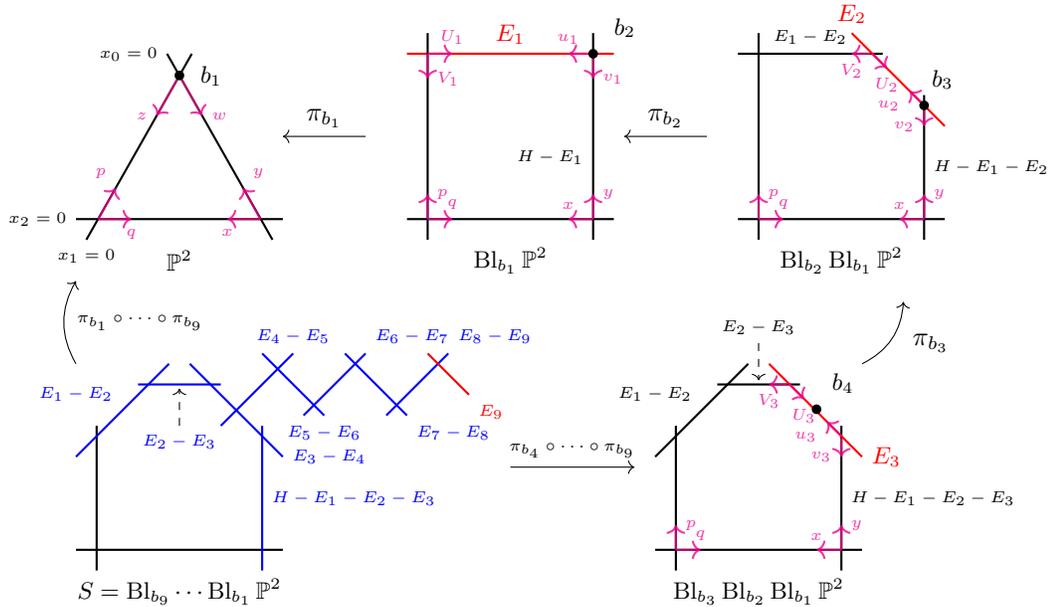
\begin{figure}[!ht]
\centering
    \begin{tikzpicture}[scale=.55,basept/.style={circle, draw=black!100, fill=black!100, thick, inner sep=0pt,minimum size=1.0mm}] 
    \node at (0,-3) {\small$\mathbb{P}^2$};
    \draw[thick,black] (-2.5,-2) -- (+2.5,-2) node [pos=0,left] {\tiny $x_2=0$} ;
    \draw[thick,black] (-2.25,-2.5) -- (+.3,+2) node [pos=0,below,black] {\tiny $x_1=0$} ;
    \draw[thick,black] (+2.25,-2.5) -- (-.3,+2) node [pos=1,left,black] {\tiny $x_0=0$} ;
    \node (b1) at (0,+1.475) [basept,label={[xshift=12pt, yshift = -10 pt] \small $b_{1}$}] {}; 
    \draw[->,thick, opacity=.8,magenta] (b1) -- (.55,+.5) node[pos=1,right] {\tiny $w$};
    \draw[->,thick, opacity=.8,magenta] (b1) -- (-.55,+.5) node[pos=1,left] {\tiny $z$}; 
    \draw[->,thick, opacity=.8,magenta] (-1.96,-2) -- (-1.54,-1.25) node[pos=1,above left] {\tiny $p$};
    \draw[->,thick, opacity=.8,magenta] (-1.96,-2) -- (-1.16,-2) node[pos=1,below] {\tiny $q$}; 
    \draw[->,thick, opacity=.8,magenta] (1.96,-2) -- (1.54,-1.25) node[pos=1,above right] {\tiny $y$};
    \draw[->,thick, opacity=.8,magenta] (1.96,-2) -- (1.16,-2) node[pos=1,below] {\tiny $x$};
    \draw[black,->] (+4.5,0)--(+2.5,0) node[ midway,above] {$\pi_{b_1}$}; 
        \begin{scope}[xshift=+8cm]
    \node at (0,-3) {\small$\Bl_{b_1} \BP^2$};
    \draw[thick,black] (-2.5,-2)  -- (+2.5,-2)  node[pos=0,left] {} ;
    \draw[thick,red] (-2.5,+2) -- (+2.5,+2) node[pos=0,left,black] {} node[pos=0.5,above,red] {\small $E_1$};
    \draw[thick,black] (-2,-2.5) -- (-2,+2.5) node[pos=0,below] {} ;
    \draw[thick,black] (+2,-2.5) -- (+2,+2.5) node[pos=0.4,left,black] { \tiny $H-E_1$}  ; 
    \draw[->,thick, opacity=.8,magenta] (2,2) -- (2,1.4) node[pos=1,right] {\tiny $v_1$};
    \draw[->,thick, opacity=.8,magenta] (2,2) -- (1.4,2) node[pos=1,above] {\tiny $u_1$}; 
    \draw[->,thick, opacity=.8,magenta] (-2,2) -- (-2,1.4) node[pos=1,right] {\tiny $V_1$};
    \draw[->,thick, opacity=.8,magenta] (-2,2) -- (-1.4,2) node[pos=1,above] {\tiny $U_1$}; 
    \draw[->,thick, opacity=.8,magenta] (-2,-2) -- (-2,-1.4) node[pos=1,right] {\tiny $p$};
    \draw[->,thick, opacity=.8,magenta] (-2,-2) -- (-1.4,-2) node[pos=1,above] {\tiny $q$};  
    \draw[->,thick, opacity=.8,magenta] (2,-2) -- (2,-1.4) node[pos=1,right] {\tiny $y$};
    \draw[->,thick, opacity=.8,magenta] (2,-2) -- (1.4,-2) node[pos=1,above] {\tiny $x$}; 
    \node (b2) at (+2,+2) [basept,label={[xshift=12pt, yshift = 0 pt] \small $b_{2}$}] {};

        \draw[black,->] (+4.75,0)--(+2.75,0) node[ midway,above] {$\pi_{b_2}$};

        \end{scope} 
    \begin{scope}[xshift=+16cm]
    \node at (0,-3) {\small$\Bl_{b_2}\Bl_{b_1} \BP^2$};
    \draw[thick,black] (-2.5,-2)  -- (+2.5,-2)  node[pos=0,left] {} ;
    \draw[thick,black] (-2.5,+2) -- (1,+2) node[pos=0,left,black] {} node[pos=0.5,above,black] { \tiny $E_1-E_2$}  ;
    \draw[thick,black] (-2,-2.5) -- (-2,+2.5) node[pos=0,below] {} ;
    \draw[thick,black] (+2,-2.5) -- (+2,1)node[pos=0.5,right,black] { \tiny $H-E_1-E_2$} ; 
    \draw[thick,red] (0.25,2.5) -- (+2.5,0.25) node[pos=0,above,red] {\small $E_2$} ;  
    \draw[->,thick, opacity=.8,magenta] (0.7,2) -- (0.25,2) node[pos=1,below] {\tiny $V_2$};
    \draw[->,thick, opacity=.8,magenta] (0.75,2) -- (1.1,1.65) node[pos=1,below] {\tiny $U_2$}; 
    \draw[->,thick, opacity=.8,magenta] (2,0.7) -- (2,0.25) node[pos=1,left] {\tiny $v_2$};
    \draw[->,thick, opacity=.8,magenta] (2,0.75) -- (1.65,1.1) node[pos=1, below left] {\tiny $u_2$}; 
    \draw[->,thick, opacity=.8,magenta] (-2,-2) -- (-2,-1.4) node[pos=1,right] {\tiny $p$};
    \draw[->,thick, opacity=.8,magenta] (-2,-2) -- (-1.4,-2) node[pos=1,above] {\tiny $q$};    
    \draw[->,thick, opacity=.8,magenta] (2,-2) -- (2,-1.4) node[pos=1,right] {\tiny $y$};
    \draw[->,thick, opacity=.8,magenta] (2,-2) -- (1.4,-2) node[pos=1,above] {\tiny $x$}; 
    \node (b3) at (2,0.75) [basept,label={[xshift=7pt, yshift = 0 pt] \small $b_{3}$}] {};
    \end{scope}

    \begin{scope}[xshift=+14cm, yshift=-8cm]
    \node at (0,-3) {\small$\Bl_{b_3}\Bl_{b_2}\Bl_{b_1} \BP^2$};
    \draw[thick,black] (-2.5,-2)  -- (+2.5,-2)  node[pos=0,left] {} ;
    \draw[thick,black] (-1,+2) -- (+1,+2) node[pos=0.5,below,black] {} node[pos=0.5,below] {};
    \draw[thick,black] (-2,-2.5) -- (-2,1) node[pos=0.5,right,black] {} ;
    \draw[thick,black] (+2,-2.5) -- (+2,1) node[pos=0.5,right,black] { \tiny $H-E_1-E_2-E_3$} ; 
    \draw[dashed, ->] (0,3)  -- (0,2.1) node[pos=0,above,black] {\tiny $E_2-E_3$}; 
    \draw[thick,red] (0.25,2.5) -- (+2.5,0.25) node[pos=1,right,red] {\small $E_3$} ; 
    \draw[thick,black] (-0.25,2.5) -- (-2.5,0.25) node[pos=0.5,above left,black] {\tiny $E_1-E_2$} ;   
    \draw[->,thick, opacity=.8,magenta] (0.7,2) -- (0.25,2) node[pos=1,below] {\tiny $V_3$};
    \draw[->,thick, opacity=.8,magenta] (0.75,2) -- (1.1,1.65) node[pos=1,below] {\tiny $U_3$};  
    \draw[->,thick, opacity=.8,magenta] (2,0.7) -- (2,0.25) node[pos=1,left] {\tiny $v_3$};
    \draw[->,thick, opacity=.8,magenta] (2,0.75) -- (1.65,1.1) node[pos=1,below left] {\tiny $u_3$}; 
    \draw[->,thick, opacity=.8,magenta] (-2,-2) -- (-2,-1.4) node[pos=1,right] {\tiny $p$};
    \draw[->,thick, opacity=.8,magenta] (-2,-2) -- (-1.4,-2) node[pos=1,above] {\tiny $q$};    
    \draw[->,thick, opacity=.8,magenta] (2,-2) -- (2,-1.4) node[pos=1,right] {\tiny $y$};
    \draw[->,thick, opacity=.8,magenta] (2,-2) -- (1.4,-2) node[pos=1,above] {\tiny $x$}; 
    \node (b4) at (1.4,1.4) [basept,label={[xshift=10pt, yshift = 0 pt] \small $b_4$}] {}; 
    \draw[black, -> ] (+2.5,+2.5) to[bend right=30] (+3.5,4) node[ midway,right] {};
    \draw (+3.5,3) node[right] {$\pi_{b_3}$};
    \end{scope} 
    
    \begin{scope}[xshift=0cm, yshift=-8cm]
    \node at (0,-3) {\small$S=\Bl_{b_9}\cdots\Bl_{ b_1}\BP^2$};
    \draw[thick,black] (-2.5,-2)  -- (+2.5,-2)  node[pos=0,left] {} ;
    \draw[thick,blue] (-1,+2) -- (+1,+2) node[pos=0.5,below,black] {} node[pos=0.5,below] {};
    \draw[thick,black] (-2,-2.5) -- (-2,1) node[pos=0.5,right,black] {} ;
    \draw[thick,blue] (+2,-2.5) -- (+2,1) node[pos=0.5,right] { \tiny $H-E_1-E_2-E_3$} ; 
    \draw[dashed, ->] (0,1)  -- (0,1.9) node[pos=0,below,blue] {\tiny $E_2-E_3$}; 
    \draw[thick,blue] (0.25,2.5) -- (+2.5,0.25) node[pos=1,right] {\tiny $E_3-E_4$} ; 
    \draw[thick,blue] (-0.25,2.5) -- (-2.5,0.25) node[pos=0.5,above left] {\tiny $E_1-E_2$} ; 
    \draw[thick,blue] (1,1) -- (+2.75,2.75) node[pos=1,above] {\tiny $E_4-E_5$} ; 
    \draw[thick,blue] (2,2.75) -- (+3.5,1.25) node[pos=1,below] {\tiny $E_5-E_6$} ; 
    \draw[thick,blue] (3,1.25) -- (+4.5,2.75) node[pos=1,above right] {\tiny $E_6-E_7$} ; 
    \draw[thick,blue] (4,2.75) -- (+5.5,1.25) node[pos=1,below right] {\tiny $E_7-E_8$} ; 
    \draw[thick,blue] (5,1.25) -- (+6.5,2.75)      node[pos=1,above right] {\tiny $E_8-E_9$}  ; 
    \draw[thick,red] (6,2.75) -- (+7,1.75) node[pos=1,below right] {\tiny $E_9$} ; 
        \draw[black, -> ] (-2.5,2.5) to[bend left=30] (-2.5,4.5) node[ midway,right] {};
    \draw (-2.7,3.5) node[right] {\tiny$\pi_{b_1}\circ\cdots\circ\pi_{b_9}$};
    \draw[black, -> ] (8,0) --  (11,0) node[ midway,above] {\tiny $\pi_{b_4}\circ\cdots \circ\pi_{b_9}$ };
    \end{scope}
    
\end{tikzpicture}
\caption{Surface $S=\Bl_{b_9}\cdots\Bl_{ b_1}\BP^2$, with curves $D_i$ on $S$ giving vertical leaves in blue. 
}
\label{fig:surfaceautonomousP1}
\end{figure}

    \begin{prop} \label{prop:autP1anticanonicallinearsystem}
        The anti-canonical linear system of the surface $S$ in \Cref{prop:wpspaceofICs} has dimension $\dim |-K_S|=1$. 
        It is written as
        \begin{center}
        \begin{tikzcd}[row sep=tiny]
            \left|-K_S\right| \arrow[r] & \BP^1   \\
            D \arrow[r,mapsto] & \left[0:1\right] \\
            \widehat{C}_{g_2,g_3} \arrow[mapsto,r] & \left[1:g_3\right],
        \end{tikzcd}
        \end{center}
        where $\widehat{C}_{g_2,g_3}$ is the strict transform, under $\varepsilon$, 
        of the Weierstrass cubic\footnote{We remind the reader that in this section $g_2\in\BC$ is fixed.}
    \begin{equation} \label{eq:weierstrasscubic}
        C_{g_2,g_3} := V\left(x_0 x_2^2 - 4 x_1^3 + g_2 x_0^2{x_1}+g_3 x_0^3\right) \subset \BP^2.
    \end{equation}
        
    \end{prop}
 
    \begin{proof}
    We have already seen that $D$ is an effective divisor representing $-K_S$, and indeed it is given by $D = - \operatorname{div}\omega$, where $\omega$ is a rational 2-form on $S$ defined in the coordinates $q,p$ by $\omega = dq \wedge dp$, or in the coordinates $z,w$ by $\frac{dz\wedge dw}{z^3}$.  
    Any effective anti-canonical divisor on $S$ which is different from $D$ will be given by the strict  transform of a curve   $V(f)\subset \BP^2$, where
    \begin{equation}
          f = \sum_{\underset{i+j+k=3}{  i,j,k \geq 0}} c_{i,j,k} x_0^{i}x_1^j x_2^k\in\BC[x_0,x_1,x_2],
    \end{equation} 
    such that the strict transform of $V(f)$ is an effective divisor giving $-K_S$. Note that the expression \eqref{eq:ACdivP2blownup} obtained from the blow-up formula tells us that the degree of the polynomial $f$ must be 3, and the curve $V(f)$ must pass through the nine points.
    This leads to linear conditions on the coefficients $c_{i,j,k}$.
    For example, $b_1$ must lie on $V(f)\subset \BP^2$, which leads to $\sum_{i=0}^{3}c_{0,0,i}=0$.
    Since the remaining points are all infinitely near, requiring that they lie on (the strict transform of) $V(f)\subset \BP^2$ gives conditions on the partial derivatives of $f$.
    In practice, one may calculate the local equations in $\CW_i^{(1)}\cong \BA^2_{(u_i,v_i)}$ and $\CW_i^{(0)}\cong \BA^2_{(U_i,V_i)}$ for the pull-backs of $V(f)$ under the blow-up projections and enforce that they have the correct factors of $v_i, V_i$, i.e. the difference between the pull-back and strict transform as divisors is $E_i$.
    These computations lead to $f$ of the form
    \begin{equation} \label{eq:flambda}
       f= f_{\lambda} 
       = \lambda_0\left(x_0 x_2^2 - 4 x_1^3 + g_2 x_0^2x_1\right)+ \lambda_1 x_0^3, \quad \lambda=[\lambda_0:\lambda_1]\in \BP^1\!.
    \end{equation}
    The divisors corresponding to the open chart $\mathcal{U}_{0}\subset\BP^1$ can be parametrised in terms of $g_3= \frac{\lambda_1}{\lambda_0}$, and we have the representation in \Cref{eq:weierstrasscubic}.

\end{proof}

\begin{remark}
    The curves in the linear system $|-K_S|$ are preserved by the flow of the system \eqref{eq:autonomousP1system} on $(S\setminus \cup_{i=0}^{8}D_i )\times \mathcal{T}$.
    This is by construction since the system is solved by the Weierstrass elliptic function $\wp(t;g_2,g_3)$, which parametrises the curve $f_{\lambda}$ in \Cref{eq:flambda} with $\lambda=[g_3,1]$.
    We can construct a conserved quantity for the system \eqref{eq:autonomousP1system} as rational functions on $S$ given by ratios of distinct polynomials $f_{\lambda},f_{\lambda'}$. 
    For example, in the original $q,p$ variables, we can take
    \begin{equation}
        I = \frac{f_{[0:1]}}{f_{[1:0]}} = \frac{4 x_1^3 - x_0x_2^2 - g_2 x_0^2 x_1}{x_0^3} = 4 q^3 - p^2 - g_2 q,
    \end{equation} 
    which is conserved under the flow of the system \eqref{eq:autonomousP1system} and corresponds exactly to the Weierstrass cubic curve in \Cref{eq:weierstrassfirstorder}, with the conserved quantity denoted by $I=g_3$. 
\end{remark}

    Returning to the hypersurface $\cup_{i=0}^{8}D_i$, in \Cref{ex:E8config} the reader can verify that the irreducible components $D_i$, for $i=0,\ldots,8$, as in \Cref{eq:Denum} intersect according to the $E_8^{(1)}$ Dynkin diagram.  
    This will be formalised in terms of the generalised Cartan matrix and root lattice of type $E_8^{(1)}$ in \Cref{subsec:generalisedHalphensurfaces}.
    
    \begin{exercise} \label{ex:E8config}
        Use \Cref{lemma:eltransf} to compute the intersection pairings among $D_i\in \Div(S)$, $i=0,\ldots,8$ given in \Cref{eq:Denum}.
        Show that $D_i^2=-2$ for each $i$ and determine the enumeration of the vertices of the $E_8^{(1)}$ Dynkin diagram in \Cref{E8dynkin} by $i=0,\ldots,8$   such that it encodes the intersection configuration of the divisors $D_i$, for $i=0,\ldots,8$ as enumerated in \Cref{eq:Denum}.
        That is, show that $D_i.D_j\in\Set{0,1}$, for $i\not=j$ and then find an explicit bijiection between indices $i=0,\ldots,8$ and vertices of the Dynkin diagram such that two vertices $i$ and $j$ are joined by an edge if $D_i.D_j=1$, and not joined otherwise.  
    \end{exercise}

\begin{remark} \label{rem:ellipticsurfaces}
\Cref{prop:autP1anticanonicallinearsystem} is related to the fact that the surface $S$ constructed in \Cref{prop:wpspaceofICs} is an \textit{elliptic surface}. 
While a general definition of this is outside the scope of these notes, we describe it heuristically in this setting as a surface $S$ admitting an elliptic fibration, i.e. a morphism $S\to C$ to a curve $C$, with almost all fibres being smooth elliptic curves.
Examples of rational elliptic surfaces provide spaces of initial conditions for autonomous limits of Painlev\'e differential equations \cite{SakaiRationalEllipticSurfaces}.
This fact has counterpart in the discrete setting through QRT maps, which can be regarded as autonomous limits of discrete Painlev\'e equations and also have spaces of initial conditions provided by rational elliptic surfaces \cite{Tsuda,Duistermaat}, see also \cite{CarsteaTakenawaRationalEllipticSurfaces}. 
\end{remark}

\begin{remark} \label{rem:p1surface} 
    If, instead of the autonomous system \eqref{eq:autonomousP1system}, we perform the construction for $\pain{I}$ in the form of the non-autonomous Hamiltonian system
\begin{equation} \label{eq:hamP1}
    q' = p, \qquad p' = 6q^2 + t,
\end{equation}
we also obtain a space of initial conditions. 
However, some of the points to be blown-up will be $t$-dependent, and the projective surface $S$ forming the fibre will no longer admit an elliptic fibration. Rather, there will be a unique $D \in |-K_S|$, the irreducible components of which are the vertical leaves removed in the construction.
These will be $-2$ curves in the same configuration as pointed out in the proof of \Cref{prop:wpspaceofICs}, which intersect according to the $E_8^{(1)}$ Dynkin diagram in \Cref{E8dynkin}.
This surface is an example of a generalised Halphen surface $S$ with $\dim  |-K_S|=0$, which we call a \emph{Sakai surface} (see \Cref{def:sakaisurface}), the classification of which by Sakai \cite{Sakai2001} will be the main subject of \Cref{subsec:classificationofSakaisurfaces}.
\end{remark}
\begin{exercise}
    Construct a space of initial conditions along the same lines as in \Cref{prop:wpspaceofICs} for the second-order ODE
    \begin{equation} \label{eq:autonomousP2scalar}
        q'' = 2 q^3 + a q,
    \end{equation}
    where $a\in \BC$.
    This is an autonomous version of $\pain{II}$ and all solutions of \Cref{eq:autonomousP2scalar} are meromorphic on $\BC$, given in terms of the Jacobi elliptic function $\operatorname{sn}(z;k)$,  the definition and properties of which can be found in, e.g., \cite[Ch. 22]{DLMF}.
    Note that if a solution $q(t)$ of \Cref{eq:autonomousP2scalar}  fails to be analytic at $t=t_*\in \BC$, then it has a simple pole at $t=t_*$ of residue $\pm 1$, given by a Laurent series expansion, involving a free parameter which plays an analogous role to $\mu$ in \Cref{lem:wpexpansion}, that can be computed explicitly. 
\end{exercise}

\subsection{Space of initial conditions for a QRT map}
\label{subsec:QRT}

Consider an autonomous system of two first-order difference equations
\begin{equation} \label{eq:differenceeqgeneral}
(x_{n+1}, y_{n+1} ) = ( f(x_n,y_n), g(x_n,y_n) ),
\end{equation}
where $f$, $g$ are rational functions of their arguments with coefficients independent of $n$, such that the mapping 
\begin{equation} \label{eq:birationaldifferenceeqgeneral} 
\begin{tikzcd}[row sep =tiny]
    \BA^2_{(x_n,y_n)} \arrow[r,dashed]& \BA^2_{(x_{n+1},y_{n+1})}    \\
    (x_n,y_n) \arrow[r,mapsto]& ( f(x_n,y_n), g(x_n,y_n) ),
    \end{tikzcd}
\end{equation} 
is birational. 
By taking $(x_n,y_n)$ as coordinates on an affine chart for $\BP^2$ we can extend \eqref{eq:birationaldifferenceeqgeneral} to a birational transformation $\varphi\in\Bir(\BP^2)$. 
A space of initial conditions in the automorphism sense  as defined in  in \Cref{def:spaceofinitialcond}, for $\varphi$ as above, consists of a birational morphism $\varepsilon : S\rightarrow \BP^2$ from a smooth projective rational surface $S$, such that the lifted map $\tilde{\varphi}\in\Bir(S)$ is an automorphism.

For systems of difference equations of the form \eqref{eq:differenceeqgeneral}, it is sometimes more convenient to extend the map \eqref{eq:birationaldifferenceeqgeneral} to $\BP^1 \times \BP^1$ rather than $\BP^2$, and then find a birational morphism $\varepsilon : S\rightarrow \BP^1 \times \BP^1$ under which $\varphi$ lifts to an automorphism of $S$.
Note that as long as $\varepsilon$ is not an isomorphism, then we automatically also get a morphism $S\rightarrow \BP^2$, since $\BP^1\times\BP^1$ blown-up at one point is isomorphic $\BP^2$ blown-up at two points, as illustrated in \Cref{exa:2points}. Note that the converse is not true if the two points in $\BP^2$ are infinitely near - the surface $S$ from \Cref{prop:wpspaceofICs} provides an example which admits a morphism to $\BP^2$ but not to $\BP^1 \times \BP^1$.
This alternative choice of $\BP^1 \times \BP^1$ as \emph{compactification} of $\BA^2$ is performed  by letting $(x,y)=(x_n,y_n)$ and and $(\bar{x},\bar{y})=(x_{n+1},y_{n+1})$ be the affine coordinates on  $\mathcal{U}_{0,0}\subset \BP^1\times\BP^1$ as in \Cref{rmk:coorchart}, so the birational map \eqref{eq:birationaldifferenceeqgeneral} gives $\varphi\in \Bir(\BP^1\times\BP^1)$.
We will often specify a birational map $\varphi \in \Bir(\BP^1 \times \BP^1)$, just in the affine coordinates $(x,y)$, $(\bar{x},\bar{y})$, with extension to the rest of $\BP^1\times\BP^1$ being via the transition functions in \Cref{rmk:coorchart}.

We will illustrate this in the example of the second-order difference equation
\begin{equation} \label{scalarQRT}
x_{n+1} = \frac{(x_n-k)(x_n+k) x_{n-1}}{k^2 - x_n^2 +2 t x_n x_{n-1} },
\end{equation}
with parameters $k \neq 0, \pm 1$ and $t \neq 0$. 
Note that this can be written as a system of two first-order difference equations as in \Cref{eq:differenceeqgeneral} by letting $y_n=x_{n-1}$, so 
\begin{equation}
    x_{n+1} = \frac{(x_n-k)(x_n+k) y_{n}}{k^2 - x_n^2 +2 t x_n y_{n} }, \quad y_{n+1}=x_n.
\end{equation}
Equation \eqref{scalarQRT} in fact belongs to the family of QRT maps \cite{QRT1, QRT2}, the definition of which ensures that they admit a rational elliptic surface as a space of initial conditions as anticipated in \Cref{rem:ellipticsurfaces} (see \cite{Tsuda, Duistermaat}).
Consider \Cref{scalarQRT} as a birational transformation
\begin{equation} \label{eq:QRTphi} 
\varphi : \BP^1 \times \BP^1 \dashrightarrow \BP^1 \times \BP^1, \quad
(x,y) \mapsto (\bar{x}, \bar{y}) = \left( \frac{(x-k)(x+k) y}{k^2 - x^2 +2 t x y },  x\! \right)\!. 
\end{equation}
Then, we have the following.

\begin{prop} \label{prop:QRTspaceofICs} Let $\varphi\in\Bir(\BP^1\times\BP^1)$ be as in \eqref{eq:QRTphi}. Then, there is a birational morphism $\varepsilon : S \rightarrow \BP^1 \times \BP^1$, with $S$ a smooth projective rational surface, such that $\varphi $ lifts to an automorphism $\tilde{\varphi} = \varepsilon^{-1}\circ \varphi\circ \varepsilon$ of $S$.
    Further, the surface $S=\Bl_{b_8} \cdots\Bl_{b_1} (\BP^1 \times \BP^1)$ is obtained via a sequence of eight blow-ups each centred at a point. 
\end{prop}

\begin{proof}
    We begin by finding the indeterminacy loci of both $\varphi,\varphi^{-1}\in\Bir(\BP^1 \times \BP^1)$.
    Writing both in the affine coordinate charts from \Cref{rmk:coorchart}, we see that 
    \begin{equation} \label{eq:indQRT}
    \begin{aligned}
        \ind\varphi &= \{ b_1, b_3\}, \qquad b_1 : (x,y)=(k,0), \quad  b_{3}:(x,y)=(-k,0),\\
        \ind\varphi^{-1} &= \{ b_5, b_7\}, \qquad b_5 : (x,y)=(0,k), \quad  b_{7}:(x,y)=(0,-k).
    \end{aligned}
    \end{equation}
    Our aim is first to resolve these indeterminacies through an appropriate sequence of blow-ups, then verify that the lifted map is an automorphism by calculations in coordinates.
    First consider $b_1\in \ind\varphi$.
    By blowing-up $b_1$, we introduce two new affine charts $\CW_1^{(1)}\cong \BA^2_{(u_1,v_1)}$, $\CW_1^{(0)}\cong \BA^2_{(U_1,V_1)}$ using the same convention as in the proof of \Cref{prop:wpspaceofICs}.
    The lifted map $\Bl_{b_1}(\BP^1\times\BP^1)\dashrightarrow \BP^1 \times \BP^1$ is written in the chart $\CW_1^{(1)}$ as 
    \begin{equation}
        (u_1,v_1) \mapsto (\bar{x},\bar{y}) = \left( \frac{u_1v_1 (2k+u_1 v_1)}{2 k(u_1-t) +u_1v_1(u_1-2t)} , k+ u_1v_1\right)\!,
    \end{equation}
    so there is still an indeterminacy at $b_2 \in E_1$, given by
    $$b_2 : (u_1,v_1)=(t,0).$$
    By blowing-up $b_2$ and introducing coordinates in the same way (given explicitly in \Cref{fig:blow-updataQRT}), we see that the lifted map $\Bl_{b_2}\Bl_{b_1}(\BP^1 \times \BP^1) \dashrightarrow \BP^1 \times \BP^1$ is written in the chart $\CW_2^{(1)}\cong \BA^2_{(u_2,v_2)}$ as  
    \begin{equation}
        (u_2,v_2) \mapsto (\bar{x},\bar{y}) = \left( \frac{ (t+ u_2v_2) (2k+ t v_2 + u_2 v_2^2)}{ t^2 -2 k u_2-u_2^2 v_2^2} , k+ t v_2 + u_2v_2^2\right)\!.
    \end{equation}
    Then, the restriction of $\varphi$ to  $\CW_2^{(1)}\cap E_2$ has no indeterminacies and hence it restricts on $  V(v_2) \subset \BA^2_{(u_2,v_2)}$ to the morphism  given by 
    \begin{equation}
        (u_2,0) \mapsto (\bar{x},\bar{y}) = \left( \frac{ 2 t k}{ t^2 -2 k u_2} , k\right)\!.
    \end{equation}  
    This means that the lifted map restricts to an isomorphism from $E_2$ to the strict transform of the curve $\bar{y}=k$, which corresponds to $H_y-E_5\in \Pic(S)$.
    Similar calculations show that it requires two blow-ups to resolve each of the three remaining indeterminacies in \Cref{eq:indQRT}, and we obtain\footnote{Note the abuse of notation: the points $b_3,b_7$ formally do not lie on $\BP^1\times \BP^1$, but rather on a blow-up of it. However, no confusion should arise, as they do not lie on the exceptional divisors, outside of which the blow-up morphism is an isomorphism.} the surface $S = \Bl_{b_8}\cdots \Bl_{b_1}(\BP^1 \times \BP^1)$, with data of points and coordinates in \Cref{fig:blow-updataQRT}.
    A pictorial representation of $S$ is given in \Cref{fig:surfaceQRT}.

\begin{table}[htb]
    \centering
   {\small \begin{tabular}{c|c}
        $b_i$ & $\pi_{b_i}$ \\\hline 
        $\mathcal{U}_{0,0} \ni b_1:(x,y)=(k,0)$\hspace{0.5cm} & $\begin{tikzcd}[row sep =tiny]
            \CW_1^{(0)} \ni (U_1,V_1)\arrow[r,mapsto]& (k+V_1, U_1 V_1)\in\mathcal{U}_{0,0}\\[-1em]
            \CW_1^{(1)} \ni(u_1 , v_1)\arrow[r,mapsto]& (k+u_1 v_1, v_1)\in\mathcal{U}_{0,0}
        \end{tikzcd}$ \\\hline 
        $E_1 \ni b_2:(u_1,v_1)=(t,0)$\hspace{0.5cm}  & $\begin{tikzcd}[row sep =tiny]
            \CW_2^{(0)} \ni (U_2,V_2)\arrow[r,mapsto]& (t+V_2, U_2 V_2)\in\CW_1^{(1)}\\[-1em]
            \CW_2^{(1)} \ni(u_2 , v_2)\arrow[r,mapsto]& (t+u_2 v_2, v_2)\in\CW_1^{(1)}
        \end{tikzcd}$ \\\hline 
        $\mathcal{U}_{0,0} \ni b_3:(x,y)=(-k,0)$\hspace{0.5cm}  & $\begin{tikzcd}[row sep =tiny]
            \CW_3^{(0)} \ni (U_3,V_3)\arrow[r,mapsto]& (-k+V_3, U_3 V_3)\in\mathcal{U}_{0,0}\\[-1em]
            \CW_3^{(1)} \ni(u_3 , v_3)\arrow[r,mapsto]& (-k+u_3 v_3, v_3)\in\mathcal{U}_{0,0}
        \end{tikzcd}$ \\\hline 
        $E_3 \ni b_4:(u_3,v_3)=(t,0)$\hspace{0.5cm}  & $\begin{tikzcd}[row sep =tiny]
            \CW_4^{(0)} \ni (U_4,V_4)\arrow[r,mapsto]& (t+V_4, U_4 V_4)\in\CW_3^{(1)}\\[-1em]
            \CW_4^{(1)} \ni(u_4 , v_4)\arrow[r,mapsto]& (t+u_4 v_4, v_4)\in\CW_3^{(1)}
        \end{tikzcd}$ \\\hline 
        $\mathcal{U}_{0,0} \ni b_5:(x,y)=(0,k) $\hspace{0.5cm}  & $\begin{tikzcd}[row sep =tiny]
            \CW_5^{(0)} \ni (U_5,V_5)\arrow[r,mapsto]& (V_5,k+ U_5 V_5)\in\mathcal{U}_{0,0}\\[-1em]
            \CW_5^{(1)} \ni(u_5 , v_5)\arrow[r,mapsto]& (u_5 v_5,k+ v_5)\in\mathcal{U}_{0,0}
        \end{tikzcd}$ \\\hline 
        $E_5 \ni b_6:(U_5,V_5)=(t,0)$\hspace{0.5cm}  & $\begin{tikzcd}[row sep =tiny]
            \CW_6^{(0)} \ni (U_6,V_6)\arrow[r,mapsto]& (t+V_6, U_6 V_6)\in\CW_5^{(0)}\\[-1em]
            \CW_6^{(1)} \ni(u_6 , v_6)\arrow[r,mapsto]& (t+u_6 v_6, v_6)\in\CW_5^{(0)}
        \end{tikzcd}$ \\\hline 
        $\mathcal{U}_{0,0} \ni b_7:(x,y)=(0,-k)$\hspace{0.5cm}  & $\begin{tikzcd}[row sep =tiny]
            \CW_7^{(0)} \ni (U_7,V_7)\arrow[r,mapsto]& (V_7,-k+ U_7 V_7)\in\mathcal{U}_{0,0}\\[-1em]
            \CW_7^{(1)} \ni(u_7 , v_7)\arrow[r,mapsto]& (u_7 v_7, -k+v_7)\in\mathcal{U}_{0,0}
        \end{tikzcd}$ \\\hline 
        $E_7 \ni b_8:(U_7,V_7)=(t,0)$\hspace{0.5cm}  & $\begin{tikzcd}[row sep =tiny]
            \CW_8^{(0)} \ni (U_8,V_8)\arrow[r,mapsto]& (t+V_8, U_8 V_8)\in\CW_7^{(0)}\\[-1em]
            \CW_8^{(1)} \ni(u_8 , v_8)\arrow[r,mapsto]& (t+u_8 v_8, v_8)\in\CW_7^{(0)}
        \end{tikzcd}$  
    \end{tabular}}
    \caption{Blow-up data for the space of initial conditions for the mapping \ref{eq:QRTphi}.}
    \label{fig:blow-updataQRT}
\end{table}

           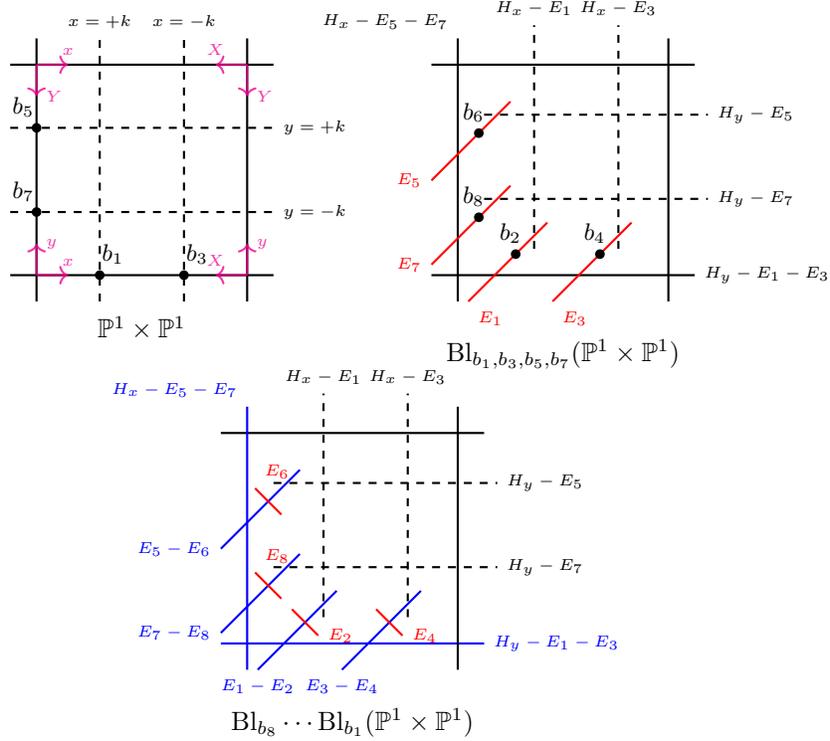
\begin{figure}[htb]
\centering
    \begin{tikzpicture}[scale=.7,basept/.style={circle, draw=black!100, fill=black!100, thick, inner sep=0pt,minimum size=1.0mm}] 
    \node at (0,-3) {$\BP^1\times\BP^1$};
    \draw[thick,black] (-2.5,-2)  -- (+2.5,-2)  node[pos=0,left] {} ;
    \draw[thick,black] (-2.5,+2) -- (+2.5,+2) node[pos=0,left,black] {} node[pos=0.5,above] {};
    \draw[thick,black] (-2,-2.5) -- (-2,+2.5) node[pos=0,below] {} ;
    \draw[thick,black] (+2,-2.5) -- (+2,+2.5) node[pos=0.4,left,black] {}  ; 
    \draw[->,thick, opacity=.8,magenta] (2,2) -- (2,1.4) node[pos=1,right] {\tiny $Y$};
    \draw[->,thick, opacity=.8,magenta] (2,2) -- (1.4,2) node[pos=1,above] {\tiny $X$}; 
    \draw[->,thick, opacity=.8,magenta] (-2,2) -- (-2,1.4) node[pos=1,right] {\tiny $Y$};
    \draw[->,thick, opacity=.8,magenta] (-2,2) -- (-1.4,2) node[pos=1,above] {\tiny $x$}; 
    \draw[->,thick, opacity=.8,magenta] (-2,-2) -- (-2,-1.4) node[pos=1,right] {\tiny $y$};
    \draw[->,thick, opacity=.8,magenta] (-2,-2) -- (-1.4,-2) node[pos=1,above] {\tiny $x$}; 
    \draw[->,thick, opacity=.8,magenta] (2,-2) -- (2,-1.4) node[pos=1,right] {\tiny $y$};
    \draw[->,thick, opacity=.8,magenta] (2,-2) -- (1.4,-2) node[pos=1,above] {\tiny $X$}; 
    \node (b3) at (+.8,-2) [basept,label={[xshift=5pt, yshift = -2 pt] \small $b_{3}$}] {};
    \node (b1) at (-.8,-2) [basept,label={[xshift=5pt, yshift = -2 pt] \small $b_{1}$}] {};
    \node (b5) at (-2,+.8) [basept,label={[xshift=-5pt, yshift = -2 pt] \small $b_{5}$}] {};
    \node (b7) at (-2,-.8) [basept,label={[xshift=-5pt, yshift = -2 pt] \small $b_{7}$}] {};

    \draw[thick, dashed, black] ($(b1) - (0,.5)$) -- ($(b1) + (0,4.5)$) node[pos=1,above] {\tiny $x=+k$};
    \draw[thick, dashed, black] ($(b3) - (0,.5)$) -- ($(b3) + (0,4.5)$) node[pos=1,above] {\tiny $x=-k$};
    \draw[thick, dashed, black] ($(b5) - (.5,0)$) -- ($(b5) + (4.5,0)$) node[pos=1,right] {\tiny $y=+k$};
    \draw[thick, dashed, black] ($(b7) - (.5,0)$) -- ($(b7) + (4.5,0)$) node[pos=1,right] {\tiny $y=-k$};

    \begin{scope}[xshift=+8cm]
    \node at (0,-3.5) {$\Bl_{b_1,b_3,b_5,b_7}(\BP^1\times\BP^1)$};
    \draw[thick,black] (-2.5,-2)  -- (+2.5,-2)  node[pos=1,right] { \tiny $H_y - E_1 - E_3$} ;
    \draw[thick,black] (-2.5,+2) -- (+2.5,+2) node[pos=0,left,black] {} node[pos=0.5,above] {};
    \draw[thick,black] (-2,-2.5) -- (-2,+2.5) node[pos=1,above left] { \tiny $H_x - E_5 - E_7$} ;
    \draw[thick,black] (+2,-2.5) -- (+2,+2.5) node[pos=0.4,left,black] {}  ; 
    \node (b3) at (+.8,-2)  {};
    \node (b1) at (-.8,-2)  {};
    \node (b5) at (-2,+.8)  {};
    \node (b7) at (-2,-.8)  {};
    \draw[thick, red] ($(b3)-(.5,0) - (.5,.5)$) -- ($(b3) -(.5,0)+ (1,1)$) node [pos=0, below right] {\tiny $E_3$};
    \draw[thick, red] ($(b1) -(.5,0)- (.5,.5)$) -- ($(b1)-(.5,0) + (1,1)$) node [pos=0, below right] {\tiny $E_1$};
    \draw[thick, dashed] ($(b1) +(.25,.75) - (0,0.25)$) -- ($(b1)+(.25,.75)+(0,4)$) node [pos=1, above] {\tiny $H_x-E_1$};
    \draw[thick, dashed] ($(b3) +(.25,.75) - (0,0.25)$) -- ($(b3)+(.25,.75)+(0,4)$) node [pos=1, above] {\tiny $H_x-E_3$}; 
    \draw[thick, red] ($(b5)-(0,.5) - (.5,.5)$) -- ($(b5) -(0,.5)+ (1,1)$) node [pos=0, left] {\tiny $E_5$};
    \draw[thick, red] ($(b7) -(0,.5)- (.5,.5)$) -- ($(b7)-(0,.5) + (1,1)$) node [pos=0, left] {\tiny $E_7$};
    \draw[thick, dashed] ($(b5) +(.75,.25) - (0.25,0)$) -- ($(b5)+(.75,.25)+(4,0)$) node [pos=1, right] {\tiny $H_y-E_5$};
    \draw[thick, dashed] ($(b7) +(.75,.25) - (0.25,0)$) -- ($(b7)+(.75,.25)+(4,0)$) node [pos=1, right] {\tiny $H_y-E_7$}; 
    \node (b2) at ($(b1)+  (-.1,.4)$) [basept,label={[xshift=-2pt, yshift = -2 pt] \small $b_{2}$}]  {};
    \node (b4) at ($(b3)+  (-.1,.4)$) [basept,label={[xshift=-2pt, yshift = -2 pt] \small $b_{4}$}]  {};
    \node (b6) at ($(b5)+  (.4,-.1)$) [basept,label={[xshift=-2pt, yshift = -2 pt] \small $b_{6}$}]  {};
    \node (b8) at ($(b7)+  (.4,-.1)$) [basept,label={[xshift=-2pt, yshift = -2 pt] \small $b_{8}$}]  {};
    \end{scope}

        \begin{scope}[xshift=+4cm, yshift=-7cm]
    \node at (0,-3.5) {$\Bl_{b_8}\cdots\Bl_{b_1}(\BP^1\times\BP^1)$};
    \draw[thick,blue] (-2.5,-2)  -- (+2.5,-2)  node[pos=1,right] { \tiny $H_y - E_1 - E_3$};
    \draw[thick,black] (-2.5,+2) -- (+2.5,+2) node[pos=0,left,black] {} node[pos=0.5,above] {};
    \draw[thick,blue] (-2,-2.5) -- (-2,+2.5) node[pos=1,above left] { \tiny $H_x - E_5 - E_7$}  ;
    \draw[thick,black] (+2,-2.5) -- (+2,+2.5) node[pos=0.4,left,black] {}  ; 
    \node (b3) at (+.8,-2)  {};
    \node (b1) at (-.8,-2)  {};
    \node (b5) at (-2,+.8)  {};
    \node (b7) at (-2,-.8)  {};
    \draw[thick, blue] ($(b3)-(.5,0) - (.5,.5)$) -- ($(b3) -(.5,0)+ (1,1)$) node [pos=0, below] {\tiny $E_3-E_4$};
    \draw[thick, blue] ($(b1) -(.5,0)- (.5,.5)$) -- ($(b1)-(.5,0) + (1,1)$) node [pos=0, below] {\tiny $E_1-E_2$};
    \draw[thick, dashed] ($(b1) +(.25,.75) - (0,0.25)$) -- ($(b1)+(.25,.75)+(0,4)$) node [pos=1, above] {\tiny $H_x-E_1$};
    \draw[thick, dashed] ($(b3) +(.25,.75) - (0,0.25)$) -- ($(b3)+(.25,.75)+(0,4)$) node [pos=1, above] {\tiny $H_x-E_3$}; 
    \draw[thick, blue] ($(b5)-(0,.5) - (.5,.5)$) -- ($(b5) -(0,.5)+ (1,1)$) node [pos=0, left] {\tiny $E_5-E_6$};
    \draw[thick, blue] ($(b7) -(0,.5)- (.5,.5)$) -- ($(b7)-(0,.5) + (1,1)$) node [pos=0, left] {\tiny $E_7-E_8$};
    \draw[thick, dashed] ($(b5) +(.75,.25) - (0.25,0)$) -- ($(b5)+(.75,.25)+(4,0)$) node [pos=1, right] {\tiny $H_y-E_5$};
    \draw[thick, dashed] ($(b7) +(.75,.25) - (0.25,0)$) -- ($(b7)+(.75,.25)+(4,0)$) node [pos=1, right] {\tiny $H_y-E_7$}; 
    \draw[thick, red] ($(b1)+  (-.1,.4)+(-.25,+.25)$) -- ($(b1)+  (-.1,.4)+(+.25,-.25)$) node[pos=1,right] {\tiny $E_2$};
    \draw[thick, red] ($(b3)+  (-.1,.4)+(-.25,+.25)$) -- ($(b3)+  (-.1,.4)+(+.25,-.25)$) node[pos=1,right] {\tiny $E_4$};
    \draw[thick, red] ($(b5)+  (.4,-.1)+(-.25,+.25)$) -- ($(b5)+  (.4,-.1)+(+.25,-.25)$) node[pos=0,above right] {\tiny $E_6$};
    \draw[thick, red] ($(b7)+  (.4,-.1)+(-.25,+.25)$) -- ($(b7)+  (.4,-.1)+(+.25,-.25)$) node[pos=0, above right] {\tiny $E_8$};
    
    \end{scope}

\end{tikzpicture}
\caption{Surface $S=\Bl_{b_8}\cdots \Bl_{b_1}(\BP^1\times\BP^1)$ forming the space of initial conditions for the mapping \eqref{eq:QRTphi}. Blue indicates $-2$ curves. 
}
\label{fig:surfaceQRT}
\end{figure}
        It is important to note that resolving the indeterminacies of $\varphi$ and $\varphi^{-1}$ is a priori not sufficient to ensure that the lift $\tilde{\varphi}$ of $\varphi$ to $S$ is an automorphism. 
    In this example, we have that $\varphi$ sends  the points in $\ind\varphi^{-1}$  to $\ind\varphi$, and the mapping exhibits singularity confinement along the same lines as explained in \Cref{rem:singularities}.
    Precisely, we have
    \[
    \begin{tikzcd}[row sep =tiny]
        \ind\varphi^{-1} \arrow[r,"\varphi"]&\ind \varphi\\
        b_7\arrow[r,mapsto]&b_1\\
        b_5\arrow[r,mapsto]&b_3,
    \end{tikzcd}
    \] 
    as shown in \Cref{fig:singpattern}.
    This ensures that $\varphi$ lifts to an automorphism under the resolution of $\ind \varphi$ and $\ind \varphi^{-1}$. 
    The fact that $\tilde{\varphi}$ is an automorphism is verified by calculations in the local coordinates given in \Cref{fig:blow-updataQRT}. 
\end{proof} 
    \begin{figure}[htb] 
\begin{center}
\begin{tikzpicture}[scale=.2]
\draw (1, -1) -- node[midway,left]{\small$~~~~S~~~~~~~~~~~$}(1,5) node[above]{}
(8,-1) -- (8,5) node[above]{\small$E_8$} 
(15,-1) -- (15,5) node[above]{\small$E_2$} 
(22,-1) -- (22,5) node[above]{} ;
\draw (1, -9) node[below]{\small$x=-k$}-- node[midway, left]{\small$\BP^1\times\BP^1~~~~~~~~$} (1,-3)
(22, -9) node[below]{\small$y=k$}-- (22,-3);
\filldraw 
(8,-8) circle (4pt) node[below]{\small$b_7$}
(15,-8) circle (4pt) node[below]{\small$b_1$}; 
\draw[->]
(3,2) -- node[above,midway]{\small$\tilde{\varphi}$} (6,2);
\draw[->]
(10,2) -- node[above,midway]{\small$\tilde{\varphi}$} (13,2);
\draw[->]
(17,2) -- node[above,midway]{\small$\tilde{\varphi}$} (20,2);
 
\draw[->]
(3,-8) -- node[above,midway]{} (6,-8);
\draw[->]
(10,-8) -- node[above,midway]{} (13,-8);
\draw[->]
(17,-8) -- node[above,midway]{} (20,-8);
 
\draw[->]
(8,-3) -- node[left,midway]{\small$\varepsilon$} (8,-5);
\draw[->]
(15,-3) -- node[left,midway]{\small$\varepsilon$} (15,-5);
\end{tikzpicture}
\caption{Movement of $\ind\varphi^{-1}$, under $\varphi$ in \Cref{eq:QRTphi} and isomorphism between exceptional curves.}
\label{fig:singpattern}
\end{center}
\end{figure}

 \Cref{th:blow-upformula} implies that the surface $S = \Bl_{b_8}\cdots \Bl_{b_1} (\BP^1 \times \BP^1)$ has 
    \begin{equation} \label{eq:picbasisQRT}
        \Pic(S) = \langle H_x, H_y, E_1, \ldots , E_8\rangle_{\BZ}, 
    \end{equation}
    where $H_x$ and $H_y$ are (the classes of) the fibres of the  two canonical  projections from $\BP^1 \times \BP^1$ to the $\BP^1$ factors.
    We give a pictorial description of $S$ in \Cref{fig:surfaceQRT}.

\begin{exercise}
    Show that the surface $S$ in \Cref{prop:QRTspaceofICs} has $\dim |-K_S|=1$, with the divisors in $|-K_S|$  being the strict transforms under $\varepsilon = \pi_{b_1}\circ  \cdots\circ \pi_{b_8}$ of members of the pencil of biquadratic curves
    \begin{equation} \label{pencil}
    \Set{ \!V( \lambda_0(k^2 - x^2 -y^2 +2 t x y ) + \lambda_1 x^2 y^2) \subset \BP^1 \times \BP^1 ~|~  [\lambda_0,\lambda_1]\in \BP^1 }\cong\BP^1.
\end{equation}
\end{exercise}

\begin{exercise} \label{ex:pushforwardQRT}
    For the automorphism $\tilde{\varphi}$ of $S$ lifted from $\varphi$ in \Cref{eq:QRTphi}, compute the induced pushforward
    \begin{equation*}
        \tilde{\varphi}_* : \Pic(S) \rightarrow \Pic(S),
    \end{equation*}
    in terms of the basis $H_x,H_y,E_1,\dots,E_8$ in \Cref{eq:picbasisQRT}.
    \\\textbf{Hint:}
    Since $\tilde{\varphi}$ is an isomorphism, the pushforward of the class of a prime divisor is just the class of its image under $\tilde{\varphi}$. So, one can complete the exercise by computing the images of sufficiently many divisors under $\tilde{\varphi}$, using coordinates in \Cref{fig:blow-updataQRT}. It will be sufficient to consider the ones marked in red and blue in \Cref{fig:surfaceQRT}.
\end{exercise}
    
\begin{exercise}
    Construct a space of initial conditions along the same lines as in \Cref{prop:QRTspaceofICs} for the mapping
    \begin{equation} \label{eq:autqP1}
\varphi : \BP^1 \times \BP^1 \dashrightarrow \BP^1 \times \BP^1, \quad
(x,y) \mapsto (\bar{x}, \bar{y}) = \left( \frac{a+b x}{x^2y}, x\right)\!,
    \end{equation}
    where $a,b \in \BC\setminus \{0\}$.
\end{exercise}

\begin{exercise} \label{ex:qP1}
    Instead of the birational transformation $\varphi\in\Bir(\BP^1\times\BP^1)$ in \Cref{eq:autqP1}, consider the sequence of birational mappings $\varphi_n\in\Bir(\BP^1\times\BP^1)$, $n\in \BZ$, defined by 
\begin{equation} \label{eq:qP1} 
\varphi_n : \BP^1 \times \BP^1 \dashrightarrow \BP^1 \times \BP^1, \quad
(x,y) \mapsto (\bar{x}, \bar{y}) = \left( \frac{a+b q^n x}{x^2 y},  x \right)\!, 
\end{equation}
where $a,b \in \BC\setminus \{0\}$, and $q\in\BC$, $\abs{q}\neq1$.
Construct a space of initial conditions for this (sequence of) mapping(s). That is, find a (sequence of) birational morphisms $\varepsilon_n: S_n \rightarrow\BP^1\times\BP^1$, with $S_n$ a (sequence of) smooth projective rational surface(s), such that $\varphi_n$ lifts to an isomorphism $\tilde{\varphi}_n = \varepsilon_{n+1}^{-1}\circ \varphi_n\circ \varepsilon_n : S_n \rightarrow S_{n+1}$, with $S_n$ obtained via a sequences of eight blow-ups, each centred at a (possibly $n$-dependent) point. 
    \end{exercise}
\begin{remark}
    \Cref{ex:qP1} can be regarded similarly to how $\pain{I}$ was a non-autonomous analogue of the ODE \eqref{eq:weierstrassfirstorder}, and the mapping \eqref{eq:qP1} has a space of initial conditions formed of surfaces with only a single effective anti-canonical divisor, as anticipated  in \Cref{rem:p1surface}.
    In fact, \Cref{eq:qP1} is a discrete Painlev\'e equation of multiplicative ($q$-difference) type, sometimes called $q\pain{I}$, and \Cref{eq:autqP1} is its autonomous degeneration solved by elliptic functions.
\end{remark}

\subsection{Generalised Halphen surfaces}
\label{subsec:generalisedHalphensurfaces} 

 The starting point for the Sakai framework is a definition of surfaces which generalise those forming the spaces of initial conditions for Painlev\'e differential equations as constructed by Okamoto \cite{Okamoto1979}.
 For the remainder of this section, by surface we mean a smooth projective rational surface.

\begin{definition}[Generalised Halphen surface \cite{Sakai2001}]
\label{def:GHS}
    A surface $S$ is called a \emph{generalised Halphen surface} if it has an effective anti-canonical divisor $D \in |-K_S|$ such that if $D= \sum_i m_i D_i$, $m_i>0$ is its  decomposition as a linear combination of prime divisors then $D_i\cdot K_S = 0$ for all indices $i$.\footnote{This condition is referred to in \cite{Sakai2001} as $D$ being of canonical type.}
    \end{definition}
A generalised Halphen surface $S$ has $\dim |-K_S|$ equal to either $0$ or $1$. In the latter case $S$ is a Halphen surface of index 1, which is a type of rational elliptic surface,  see \Cref{rem:ellipticsurfaces}, with $|-K_S|$ providing its elliptic fibration.
 
We have seen examples of these in \Cref{example:wp} and \Cref{subsec:QRT}.
In the former case, $S$ has a unique effective anti-canonical divisor and  corresponds to the type of surface associated with discrete and differential Painlev\'e equations. We therefore make the following definition.
\begin{definition}[Sakai surface] \label{def:sakaisurface}
A \emph{Sakai surface} $S$ is a generalised Halphen surface with $\dim |-K_S|=0$.
\end{definition}

\begin{exercise} \label{ex:canonicaltype}
     Consider the surface $S=\Bl_{b_9}\cdots \Bl_{b_1}\BP^2$ constructed in \Cref{example:wp}.
     Verify that $D = \sum_{i=0}^{8} m_i D_i \in|-K_S|$ as given in \Cref{eq:Dwp} is such that $D_i.D=0$ for all $i=0,\ldots,8$, so $S$ is a generalised Halphen surface. 
\end{exercise} 
Sakai surfaces are associated with affine root systems, so we begin with some basic facts that will be used when we formalise this association in \Cref{subsec:classificationofSakaisurfaces}. 
We have the following from {\cite[Proposition 2]{Sakai2001}}, the proof of which we omit for brevity.
\begin{prop} \label{prop:rankPic10}
Let $S$ be a generalised Halphen surface. 
\begin{enumerate}
    \item The Picard group $\Pic(S)$ has $\rk \Pic(S) = 10$.  
    \item The surface $S$ admits $\BP^2$ as a minimal model, i.e. there exists a birational morphism $\pi : S \rightarrow \BP^2$.
\end{enumerate}
\end{prop}

\begin{remark} \label{rem:P2vsP1P1compactification}
  The result in \Cref{ex:canonicaltype} is true for $D\in \Div(S)$ as constructed from the autonomous system \eqref{eq:autonomousP1system}, but it also holds for the unique effective anti-canonical divisor on the surface constructed from the non-autonomous system \eqref{eq:hamP1} in \Cref{rem:p1surface}, equivalent to $\pain{I}$.
  This Sakai surface, with components of $D$ intersecting according to the $E_8^{(1)}$ diagram, is the only type in the classification which  also does not  admit $\BP^1 \times \BP^1$ as a minimal model. 
    If a surface $S$ admits a birational morphism to $\BP^2$ which can be written as a composition of blow-ups $\pi = \pi_1 \circ\pi_2\circ \cdots \circ \pi_{n-1} \circ \pi_n :S\rightarrow \BP^2$ where $\pi_{2}$ and $\pi_1$ contract curves onto distinct points, then $S$ also admits a birational morphism to $\BP^1\times\BP^1$.
    This can be achieved by replacing the last two contractions $\pi_{2}$ and $\pi_1$  onto, say, $p,q\in \BP^2$ with the contraction of a single curve onto $r \in \BP^1\times\BP^1$, as in \Cref{exa:2points}.
    This is not possible for a Sakai surface with  $E_8^{(1)}$ configuration of components of $D$, since to obtain such a surface by blow-ups of $\BP^2$, all of the points must be infinitely near.
\end{remark}

\begin{lemma} \label{prop:ACdivconnected}
Let $S$ be a surface with an effective anti-canonical divisor. Then, for any $D= \sum_i m_i D_i \in |-K_S|$, the support $\supp D:= \cup_{i} D_i$ is connected.
\end{lemma}
\begin{proof}
The idea, following \cite{Sakai2001}, is to suppose that $\supp D$ is not connected and obtain a contradiction. 
In order to do so, we will need two facts.
For $D \in |-K_S|$, we have 
\begin{enumerate}
\item $\dim H^1(D, \OO_D) = 1$,
\item $H^1(D',\OO_{D'}) = 0$, for any $D'\in\Eff (S)$ such that $D'\begin{tikzpicture}[baseline=-0.8ex]
    \node at (0,0) {$\preceq$};
    \draw (-0.108,-0.18)--(0.108,-0.18); 
    \draw (-0.108,-0.21)--(0.108,-0.1); 
\end{tikzpicture}D$.
\end{enumerate}

To prove (1), we consider the exact sequence of sheaves 
\begin{equation}
0 \longrightarrow \OO_S(-D) \longrightarrow \OO_S \longrightarrow \OO_D \longrightarrow 0.
\end{equation}
We recall briefly that $\OO_S(-D)$ is the sheaf of rational functions on $S$ that vanish along $D$, so the exactness of the sequence is immediate. From this, we obtain the long exact sequence in cohomology, which includes the following:
\begin{equation} \label{eq:connectedproofLES}
H^1(S,\OO_S ) \longrightarrow H^1(D,\OO_D) \longrightarrow H^2(S,\OO_S(-D) ) \longrightarrow H^2(S,\OO_S).
\end{equation}
Now as $S$ is rational, by \Cref{prop:brinva} we have that $H^1(S,\OO_S) = H^2(S,\OO_S) = 0$, since $H^1(\BP^2,\OO_{\BP^2})=H^2(\BP^2,\OO_{\BP^2}) = 0$.
Then the exact sequence \eqref{eq:connectedproofLES} gives
\begin{equation}
H^1(D, \OO_D) \cong H^2(S, \OO_S(-D) ).
\end{equation}
By Serre duality, see \Cref{thm:SErredual}, we have 
\begin{equation}
H^2(S, \OO_S(-D)) \cong H^0(S, \OO_S( K_S+D) )= H^0(S, \OO_S) \cong \BC,
\end{equation}
as $D$ is a representative of the anti-canonical divisor class. 
So $\dim H^1(D, \OO_D) =1$ as required. 

To prove (2), we assume $D'\in\Eff(S)$ and $D'\begin{tikzpicture}[baseline=-0.8ex]
    \node at (0,0) {$\preceq$};
    \draw (-0.108,-0.18)--(0.108,-0.18); 
    \draw (-0.108,-0.21)--(0.108,-0.1); 
\end{tikzpicture}D$, and use the same argument with $D'$ in place of $D$ to deduce 
\begin{equation}
H^1(D',\OO_{D'}) \cong H^0(S, \OO_S(K_S + D') ).
\end{equation}
Now as $D$ is an anti-canonical divisor and $D'$ is an effective divisor such that $D'\begin{tikzpicture}[baseline=-0.8ex]
    \node at (0,0) {$\preceq$};
    \draw (-0.108,-0.18)--(0.108,-0.18); 
    \draw (-0.108,-0.21)--(0.108,-0.1); 
\end{tikzpicture}D$, we have that $K_S + D' < 0$. 
Therefore, as a line bundle, $K_S + D'$ has no non-vanishing sections and 
\begin{equation}
H^0(S, \OO_S(K_S + D')) = 0,
\end{equation}
so $H^1(D',\OO_{D'})=0$ as required.

Using (1) and (2) we can prove by contradiction that $D$ is connected. 
Suppose there exist effective divisors $D', D''$ such that $D = D' + D''$, $D' \cap D'' = \varnothing$. Then $\OO_D \cong \OO_{D'} \oplus \OO_{D''}$, and 
\begin{equation}
H^1(D,\OO_D) \cong H^1(D',\OO_{D'}) \oplus H^1(D'',\OO_{D''}).
\end{equation} 
By (1) and (2), the left-hand side is isomorphic to $\BC$, while each component of the direct sum on the right-hand side is zero, so we have obtained a contradiction.
\end{proof}

\begin{prop} \label{prop:Sakaisurface-2curves}
    Let $S$ be a Sakai surface and let  $D = \sum_i m_i D_i\in |- K_S|$ be its unique effective anti-canonical divisor.  
    \begin{enumerate}
        \item If $D$ is irreducible, then it has arithmetic genus $p_a(D)=1$, so $D$ is an elliptic curve, a rational curve with a single nodal singularity, or a rational curve with a single cusp.
        \item If $D$ is not irreducible, then all of its irreducible components $D_i$ are rational curves with $D_i^2=-2$. 
    \end{enumerate}
\end{prop}

\begin{proof}
The statement (1) follows from the genus formula \eqref{eq:genusformula} applied to $D$, which by assumption has $D.K_S=D^2=0$.

To prove statement (2), we first show that $D_i^2 <0$. 
Since $D$ is of canonical type we have that 
\begin{equation}
0 = D_i.D =m_i D_i^2+ \sum_{k \neq i} m_k \left(D_k . D_i\right)\!.
\end{equation}
The connectedness of $D$, from \Cref{prop:ACdivconnected}, implies that each component $D_i$ has positive intersection with at least one other, which implies that $\sum_{k \neq i} m_k \left(D_k . D_i\right) > 0$. 
Therefore we obtain $m_i  D_i^2<0$ and, since $m_i > 0$, it means
\begin{equation}
D_i^2 < 0.
\end{equation}
By applying the genus formula to $D_i$ we get
\begin{equation*}
p_a(D_i) = 1 + \frac{1}{2} \left( D_i . D_i + K_S . D_i \right) = 1 + \frac{1}{2} D_i \cdot D_i.
\end{equation*}
Since $p_a(D_i)$ must be a non-negative integer and $D_i^2<0$, it follows that $D_i^2=-2$. 
Further, $p_a(D_i)=0$ so $D_i$ is a rational curve.
\end{proof}

\subsection{Classification of Sakai surfaces}
 \label{subsec:classificationofSakaisurfaces} 

Sakai surfaces are classified in terms of associated affine root systems realised in their Picard groups.
We have seen some preliminary facts leading to these structures in \Cref{prop:rankPic10} and \Cref{prop:Sakaisurface-2curves}, and we will now recall some facts about affine root systems and associated Weyl groups necessary to state the Sakai classification.

\subsubsection{Affine root systems and affine Weyl groups}
\label{subsubsec:affinerootsystems}

We now give an account of the theory of affine root systems and associated affine Weyl groups relevant to discrete Painlev\'e equations, adapting parts of \cite{KAC1990} to this specific situation.
This theory is related to infinite-dimensional complex Lie algebras generalising the classical finite-dimensional semi-simple ones, and much of what follows is motivated by the study of these (affine) \textit{Kac--Moody algebras} \cite{KAC1990}, including the associated Weyl groups.
This theory begins with the following definition, from which the associated root systems and Weyl groups are developed.

\begin{definition} \label{def:generalisedCartanmatrix}
A \emph{generalised Cartan matrix} is a square matrix $A = (A_{ij})_{i,j=1}^n$ of size $n$ with
\begin{itemize}
\item $A_{ii} = 2,$\quad  for \quad $i=1,\ldots n$;
\item $A_{ij}$ non-positive integers for $i \neq j$;
\item $A_{ij}=0 \quad \iff \quad A_{ji}=0$.
\end{itemize}
\end{definition}
We first introduce the Weyl group associated with a generalised Cartan matrix purely as an abstract Coxeter group as follows.

\begin{definition} \label{def:Weylgroup}
Let $A = (A_{ij})_{i,j=1}^n$ be a generalised Cartan matrix of size $n$. 
The \emph{Weyl group} of $A$, denoted by $W(A)$, is the free group generated by the symbols $r_i$, $i=1,\ldots, n$, subject to the following relations, in which $e$ indicates the identity element:
\begin{itemize}
\item $r_i^2 = e$, \quad for \quad $i=1 ,\ldots, n$,
\item $(r_i r_j)^{m_{ij}}=e$, \quad for $i \neq j, \quad \text{where } m_{ij} = \begin{cases} 2 & \text{if } A_{ij}A_{ji} = 0, \\ 3 & \text{if } A_{ij}A_{ji} = 1, \\ 4 & \text{if } A_{ij}A_{ji} = 2, \\ 6 & \text{if } A_{ij}A_{ji} = 3.\end{cases} $
\end{itemize}
In particular, when $A_{ij}A_{ji} \geq 4$ there are no relations between $r_i$ and $r_j$.
\end{definition}

The group $W(A)$ arises naturally in the study of the Kac--Moody algebra associated to $A$.
For our purposes, it will be sufficient to understand it in terms of automorphisms of a complex vector space defined by  reflections about hyperplanes, which can be done without introducing the algebra itself. 

\begin{definition} \label{def:realisation}
A realisation of a generalised Cartan matrix $A$ of size 
$n$ is a triple $(\mathfrak{h}, \Pi, \Pi^{\vee})$, where 
\begin{itemize}
\item $\mathfrak{h}$ is a vector space over $\BC$, 
\item $\Pi = \Set{ \alpha_1, \ldots, \alpha_n} \subset \mathfrak{h}^* = \Hom_{\BC}(\mathfrak{h}, \BC)$  is a linearly independent set, whose elements are called \emph{simple roots},
\item $\Pi^{\vee} = \Set{ \alpha^{\vee}_1, \ldots, \alpha^{\vee}_n } \subset \mathfrak{h}$ is a linearly independent set, whose elements are called \emph{simple coroots},
\end{itemize}
subject to the conditions 
\begin{itemize}
\item $\langle \alpha_i^{\vee} , \alpha_j \rangle = A_{ij}$,\quad for  $i,j=1, \ldots n$, 
\item $n - \rk A = \dim \mathfrak{h}  - n,$ 
\end{itemize}
where $\langle \,~~,~ \rangle : \mathfrak{h} \times \mathfrak{h}^* \rightarrow \BC$ is the evaluation pairing.
\end{definition}
Although in \Cref{def:Weylgroup} the Weyl group was introduced as an abstract group generated by symbols $r_i$, with a realisation of $A$ we have the following representation of $W(A)$.
\begin{definition} \label{def:simplereflection}
    Let $(\mathfrak{h}, \Pi, \Pi^{\vee})$ be a realisation  of a generalised Cartan matrix $A$. 
    For each $i=1,\dots,n$, the \emph{simple reflection} $r_i \in \GL(\mathfrak{h}^*)$ corresponding to $\alpha_i$ is defined by
\begin{equation}
r_i (\lambda) = \lambda - \langle\alpha_i^{\vee} , \lambda \rangle \alpha_i, \qquad \lambda\in\mathfrak{h}^*.
\end{equation} 
\end{definition}
It can be checked that \Cref{def:simplereflection} gives a faithful representation of the Weyl group $W(A)$ on $\mathfrak{h}^*$, and we abuse notation by writing $r_i$ both for the symbol as in \Cref{def:Weylgroup} and for the element of $\GL(\mathfrak{h}^*)$ as in \Cref{def:simplereflection}.

\begin{example} \label{ex:A2rootsystem}
    Consider the matrix 
    \begin{equation}\label{eq:matrixA}
        A = \left(\begin{array}{cc}
            2 & -1 \\
            -1 & 2
        \end{array}
        \right)\!,
    \end{equation}
    which is a generalised Cartan matrix, coinciding with the usual Cartan matrix of the complex simple Lie algebra $\mathfrak{sl}(3,\BC)$.
    The Weyl group of $A$ is
    \begin{equation}
        W(A)=\langle r_1, r_2 ~|~ r_1^2=r_2^2=e, ~r_1 r_2 r_1=r_2 r_1r_2\rangle \cong \mathfrak{S}_3,
    \end{equation} 
    where $\mathfrak{S}_3$ is the symmetric group on 3 symbols.
    A realisation of $A$ can be constructed as follows.
    Since $n=2$ and $\rk A=2$, we require $\dim\mathfrak{h}=2$.
    Take $\mathfrak{h}$ to be a two-dimensional vector space over $\BC$ and let $e_1,e_2\in\mathfrak{h}$ form a basis. 
    Let $\lambda_1,\lambda_2 \in \mathfrak{h}^*$ be the elements forming the dual basis to this, so $\langle e_i,\lambda_j\rangle=\delta_{i,j}$ for $i,j = 1,2$, where $\delta_{i,j}$  is the Kronecker delta.
    Then, we have a realisation $(\mathfrak{h},\Pi,\Pi^{\vee})$, where 
    \begin{equation}
    \begin{aligned}
        \Pi&=\Set{\alpha_1,\alpha_2}, \quad &\text{with}&\quad  \alpha_1 = \sqrt{2} \lambda_1, \quad\alpha_2= - \tfrac{1}{\sqrt{2}}\lambda_1 + \tfrac{\sqrt{3}}{\sqrt{2}}\lambda_2, \\
        \Pi^{\vee}&=\Set{\alpha^{\vee}_1,\alpha^{\vee}_2}, \quad &\text{with}&\quad  \alpha^{\vee}_1 = \sqrt{2} e_1, \quad\alpha^{\vee}_2= - \tfrac{1}{\sqrt{2}}e_1 + \tfrac{\sqrt{3}}{\sqrt{2}}e_2,
    \end{aligned}
    \end{equation}
    Then, the elements $r_i$, for  $i=1,2$, correspond,  as in \Cref{def:simplereflection}, to  reflections about the hyperplanes in $\mathfrak{h}^*$ orthogonal to $\alpha_1,\alpha_2$ with respect to the Hermitian inner product $(~~|~~)$ on $\mathfrak{h}^*$, given by $(\lambda_i|\lambda_j)=\delta_{i,j}$. 
    We illustrate the Weyl group action on $\mathfrak{h}^*$ restricted to $\BR \lambda_1+\BR\lambda_2\subset \mathfrak{h}^*$ in \Cref{fig:A2rootsystem}, which recovers the usual realisation in a 2-dimensional Euclidean space of the root system and Weyl group associated with the matrix $A$ in \eqref{eq:matrixA}.
\end{example}

\begin{figure}[htb]
    \centering
\scalebox{0.8}{\begin{tikzpicture} 
    \draw[dashed, thick, ->] (0,0) -- (2,0) node [right] {\small $\alpha_1$};
    \draw[dashed, thick, ->](0,0) -- (-1,1.73205) node [above] {\small $\alpha_2$};
    \draw[dashed, thick, ->](0,0) -- (1,1.73205) node [above] {\small $\alpha_1+\alpha_2$};
    \draw[dashed, thick, ->] (0,0) -- (-2,0) node [left] {\small $-\alpha_1$};
    \draw[dashed, thick, ->](0,0) -- (1,-1.73205) node [below] {\small $-\alpha_2$};
    \draw[dashed, thick, ->](0,0) -- (-1,-1.73205) node [below] {\small $-\alpha_1-\alpha_2$}; 
    \draw[dotted, thick, blue](0,-2.2) -- (0,2.2);
    \draw[blue] (0,2.5) node {\small $r_1$};
    \draw[blue, <->] (-.3,2.2) to [bend right=-30] (.3,2.2) ; 
    \draw[dotted, thick, blue,rotate=-60](0,-2.2) -- (0,2.2);
    \draw[blue,rotate=-60] (0,2.5) node {\small $r_2$};
    \draw[blue, <->,rotate=-60] (-.3,2.2) to [bend right=-30] (.3,2.2) ; 
    \draw[dotted, thick, blue,rotate=60](0,-2.2) -- (0,2.2);
    \draw[blue,rotate=60] (0,2.7) node {\small $r_1r_2r_1$};
    \draw[blue, <->,rotate=60] (-.3,2.2) to [bend right=-30] (.3,2.2) ;    \end{tikzpicture} }   
\caption{Weyl group of $A$ in \Cref{ex:A2rootsystem}.}
    \label{fig:A2rootsystem}
\end{figure}

If we have two matrices $A_1,A_2$ with realisations $(\mathfrak{h}_1, \Pi_1, \Pi^{\vee}_1),(\mathfrak{h}_2, \Pi_2, \Pi^{\vee}_2)$ respectively, we obtain a realisation of the block matrix
\begin{equation} \label{eq:GCMdirectsummatrix}
\left(\begin{array}{cc}A_1 & 0 \\0 & A_2\end{array}\right)\!,
\end{equation}
given by 
\begin{equation} \label{eq:GCMdirectsumrealisation}
\left( \mathfrak{h}_1 \oplus \mathfrak{h}_2 , (\Pi_1 \times \{0\}) \cup (\{0\} \times \Pi_2), (\Pi^{\vee}_1 \times \{0\}) \cup (\{0\} \times \Pi^{\vee}_2) \right)\!.
\end{equation}
\begin{definition} \label{def:decomposableGCM}
    If a generalised Cartan matrix and a realisation can be written as a non-trivial direct sum as in \eqref{eq:GCMdirectsummatrix} and \eqref{eq:GCMdirectsumrealisation}, possibly after a reordering of indices, then it is called \emph{decomposable}. If not, it is called \emph{indecomposable}.
\end{definition}
\begin{remark}
    Note that a realisation of $A$ is unique up to isomorphism \cite[Prop. 1.1]{KAC1990}, so we attribute the notions in \Cref{def:decomposableGCM} to generalised Cartan matrices without reference to any realisation.
\end{remark}

\begin{prop}[{\cite[Th. 4.3]{KAC1990}}]
An indecomposable generalised Cartan matrix $A$ belongs to one of three classes, which are referred to as of \emph{finite}, \emph{affine} and \emph{indefinite} types respectively. 
 These are defined as follows, where the matrix $A$ is of size $n \times n$ and, following \cite[Ch. 4]{KAC1990}, for $\mathbf{u} =(u_1,\ldots,u_n)\in \BR^n$, we write $\mathbf{u} > 0$ (respectively $\mathbf{u} \geq 0)$ if $u_i>0$ (respectively $u_i\geq 0$) for all $i=1,\dots,n$.
\begin{itemize}
\item[(Fin)] \begin{itemize}
    \item $\det A \neq 0$,
    \item there exists $\mathbf{u} > 0$ such that $A \mathbf{u} > 0$, and
    \item $A\mathbf{v} \geq 0$ implies $\mathbf{v}>0$ or $\mathbf{v}=0$.
\end{itemize} 
\item[(Aff)] 
\begin{itemize}
    \item $\operatorname{corank} A = 1$,
    \item there exists $\mathbf{u} > 0$ such that $A \mathbf{u} = 0$, and
    \item $A\mathbf{v} \geq 0$ implies $A \mathbf{v} = 0$.
\end{itemize} 
\item[(Ind)] 
\begin{itemize}
    \item there exists $\mathbf{u} > 0$ such that $A \mathbf{u} < 0$, and
    \item $A\mathbf{v} \geq 0$ for $v \geq 0$ implies $A \mathbf{v}=0$.
\end{itemize} 
\end{itemize}
\end{prop} 

The indecomposable matrices from the class (Fin) account for all Cartan matrices associated with finite-dimensional simple complex Lie algebras. 
Each of these matrices can be encoded in a \textit{Dynkin diagram}, which is a finite graph consisting of a node for each index $i \in \{1,\ldots , n \}$, with nodes corresponding to $i$ and $j$ connected by $|A_{ij} A_{ji}|$ edges, marked with arrows pointing towards $i$ if $|A_{ij}| > |A_{ji}|$, and non-oriented otherwise. The Dynkin diagrams for the class (Fin) are given in \Cref{DDfin}.

\begin{figure}[htb]
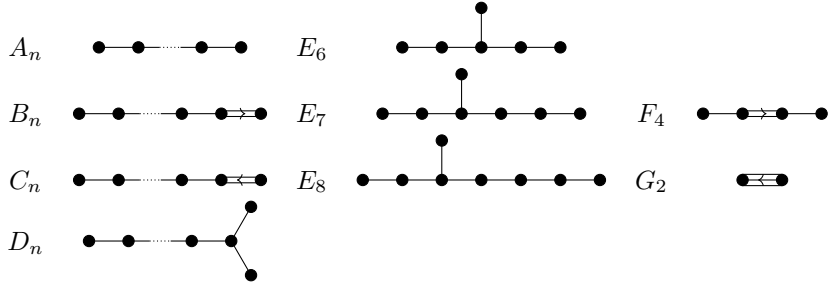

\begin{tabular}{ c c c c c c  }
 $A_n$ & \dynkin[scale=1.5]{A}{} & $E_6 $& \dynkin[scale=1.5]{E}{6} &	& \\ 
 $B_n$ & \dynkin[scale=1.5]{B}{} & $E_7 $& \dynkin[scale=1.5]{E}{7} &  $F_4$ & \dynkin[scale=1.5]{F}{4}\\  
 $C_n$ & \dynkin[scale=1.5]{C}{} & $E_8 $& \dynkin[scale=1.5]{E}{8} & $G_2$ & \dynkin[scale=1.5]{G}{2}\\
 $D_n$ & \dynkin[scale=1.5]{D}{} & 	&			&	&
\end{tabular}
\caption{Dynkin diagrams for indecomposable generalised Cartan matrices of finite type.}
\label{DDfin}
\end{figure}

In the Sakai framework for discrete Painlev\'e equations we are interested in the generalised Cartan matrices of affine type, and in particular those which are symmetric.
These correspond to (affine) Dynkin diagrams, sometimes known as extended Dynkin diagrams, which are \emph{simply laced} (meaning they only have single edges). 
From this point on we will focus on these matrices, which are relevant to discrete Painlev\'e equations, and present the corresponding Dynkin diagrams in \Cref{DDaff}. 

\begin{figure}[htb]
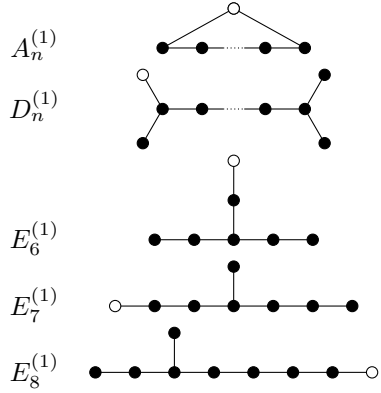

\begin{tabular}{ c c    }
$ A_n^{(1)} $& \dynkin[scale=1.5]{A}[1]{} \\
$ D_n^{(1)} $& \dynkin[scale=1.5]{D}[1]{}  \\
 $ E_6^{(1)} $ & \dynkin[scale=1.5]{E}[1]{6} \\
 $ E_7^{(1)} $& \dynkin[scale=1.5]{E}[1]{7} \\
 $ E_8^{(1)} $& \dynkin[scale=1.5]{E}[1]{8} 
\end{tabular}
\caption{Dynkin diagrams for symmetric generalised Cartan matrices of affine type.}
\label{DDaff}
\end{figure}
\begin{notation} \label{not:GCMaffinesimplylaced}
    For the remainder of this section, we let $A$ be a generalised Cartan matrix which is irreducible, of affine type, and symmetric, so with simply laced Dynkin diagram.
    We take $A = ( A_{ij} )_{i,j=0}^n$ to be of size $(n+1)\times(n+1)$ and choose a realisation $(\mathfrak{h},\Pi,\Pi^{\vee})$ with simple roots and coroots enumerated as $\Pi = \{ \alpha_0, \ldots, \alpha_n\}$, $\Pi^{\vee} = \{ \alpha^{\vee}_0, \ldots, \alpha^{\vee}_n\}$ respectively. 
    Importantly, by convention the index 0 corresponds to the non-filled node of the Dynkin diagram in \Cref{DDaff}.
    All of the constructions we will present in the remainder of this section depend, in principle, on the matrix $A$ and the choice of realisation, but we will suppress this dependence from the notation, e.g. writing $W=W(A)$, when there is no risk of confusion.  
\end{notation} 

\begin{definition} \label{def:setofroots}
    The \emph{root system} of $A$ is 
    $$\Phi \defeq W \,\Pi = \Set{ w (\alpha_i) ~|~ w \in W, \alpha_i\in \Pi}.$$
    Elements of $\Phi$ are called \emph{roots}. Strictly speaking, for $A$ as in \Cref{not:GCMaffinesimplylaced}, these are the \emph{real roots} in the sense of \cite{KAC1990}.
\end{definition}
\Cref{def:simplereflection} of the simple reflections ensures that $\Phi \subset \langle \alpha_0, \ldots, \alpha_n \rangle_{\BZ}\subset \mathfrak{h}^*$.
Next, as $A$ is of affine type, it is of corank $1$ and we have the following.
\begin{definition} \label{def:nullroot}
For $A$ as in \Cref{not:GCMaffinesimplylaced}, the \textit{null root} $\delta\in\mathfrak{h}^*$ is the element  
    \begin{equation} \label{nullroot}
    \delta = \sum_{i=0}^n m_i \alpha_i , \quad m_0=1,
    \end{equation}
where $(m_0,\ldots,m_n) \in \BZ^{n+1}$ is uniquely determined by the condition that $\sum_{j=0}^n m_j A_{ij}=0$ for all $i= 0,\ldots,n$  and $m_0=1$, so 
$\langle\alpha_i^{\vee}, \delta \rangle=0$
for $i=0,\ldots, n$.  
\end{definition}
The values of $m_i$ for all symmetric $A$ of affine type can be computed directly using its null space and are listed in \cite{KAC1990}.  
\begin{definition}
    \label{def:canonicalcentralelement}
    The \emph{canonical central element} $K \in \mathfrak{h}$ is the element
    \begin{equation} \label{canonicalcentralelement}
    K = \sum_{i=0}^n m^{\vee}_i \alpha^{\vee}_i, \quad m_0^{\vee}=1,
    \end{equation}
    where $(m^{\vee}_0,\ldots,m^{\vee}_n) \in \BZ^{n+1}$ is uniquely determined by the condition that $\sum_{i=0}^n m^{\vee}_i A_{ij}=0$ for all $j=0,\ldots n$,   and $m_0^{\vee}=1$, so 
    $\langle K , \alpha_i \rangle = 0$  
    for $i=0,\ldots, n$. 
\end{definition} 

Note that the definition of $(m^{\vee}_0,\ldots,m^{\vee}_n)$ for $A$ is the same as $(m_0,\ldots,m_n)$ for the transpose of $A$.
Since we are considering the case when $A$ is symmetric, we have $m^{\vee}_i = m_i$ for all $i=0,\ldots,n$.
The term canonical central element  comes from the fact that, when $\mathfrak{h}$ is considered as part of the Kac--Moody algebra associated with $A$, $K$ spans its centre, see \cite{KAC1990}. 
The conditions in \Cref{def:realisation} with our enumeration of $\Pi$ as in \Cref{not:GCMaffinesimplylaced} require that $\dim \mathfrak{h} = n+2$, so we extend $\Pi$, $\Pi^{\vee}$ to bases of $\mathfrak{h}^*$, $\mathfrak{h}$ as follows.

Fix a \emph{scaling element} $d \in \mathfrak{h}$ satisfying
\begin{equation}
\langle d, \alpha_i \rangle = \delta_{i,0} \quad \text{ for } i=0, \ldots, n. 
\end{equation}
The scaling element is determined up to addition of a constant multiple of $K$, and is linearly independent from $\Pi^{\vee}$.
This uniquely determines  an element $\Lambda_0 \in \mathfrak{h}^*$  by the conditions
\begin{equation}
\langle \alpha_i^{\vee}, \Lambda_0 \rangle = \delta_{i,0} \quad \text{ for } i=0, \ldots, n, \quad \langle d, \Lambda_0\rangle = 0.
\end{equation} 
\begin{lemma}
    The subsets $\Set{d, \alpha_0^{\vee}, \ldots, \alpha_n^{\vee}} \subset \mathfrak{h}$ and $\Set{\Lambda_0, \alpha_0, \ldots, \alpha_n} \subset \mathfrak{h}^*$ are bases over $\BC$.
\end{lemma}

We can now define a symmetric bilinear form on $\mathfrak{h}^*$, in terms of which the Weyl group can be described. 
\begin{definition}\label{def:bil}
    Define the symmetric bilinear form ${(\,\,\,|\,\,\,)} : \mathfrak{h}^*\times\mathfrak{h}^*\rightarrow \BC$ by 
    \begin{equation} \label{bilinformh}
        \begin{aligned}
        (\alpha_i | \alpha_j ) &= A_{ij}, \\
        (\alpha_i | \Lambda_0 ) &= \delta_{i,0}, \qquad \quad \qquad \text{ for } i,j,=0,\ldots, n,\\
        (\Lambda_0 | \Lambda_0) &= 0.
        \end{aligned}
    \end{equation} 
\end{definition}

\begin{lemma}
    The action of $W$ on $\mathfrak{h}^*$  defined by the simple reflections in \Cref{def:simplereflection} is written in terms of the symmetric bilinear form ${(\,\,\,|\,\,\,)}$ as
        \begin{equation} \label{eq:refformula}
        r_i(\lambda) = \lambda - ( \lambda | \alpha_i ) \alpha_i,
        \end{equation}
        for $\lambda\in \mathfrak{h}^*$.
    For any root $\alpha\in \Phi$, we have an element $r_{\alpha} \in W$, which acts on $\mathfrak{h}^*$ by the formula
    \begin{equation}
    r_{\alpha}(\lambda) = \lambda - (\lambda | \alpha ) \alpha,
    \end{equation}
    so in particular $r_i = r_{\alpha_i}$. 
    The element $r_{\alpha}$ is called the \emph{reflection associated to} $\alpha$. 
\end{lemma}

\begin{prop} \label{prop:weylgroupproperties}
    The Weyl group $W$ considered as a subgroup of $\GL(\mathfrak{h}^*)$ has the following properties:
    \begin{itemize}
        \item $w(\delta)=\delta$ for all $w\in W$, where $\delta$ is the null root as in \Cref{def:nullroot},
        \item $(w(\lambda_1) | w(\lambda_2) ) = (\lambda_1 | \lambda_2 ),$ for all $w \in W,~ \lambda_1, \lambda_2 \in \mathfrak{h}^*$,
        \item $r_{w(\alpha)} = w r_{\alpha} w^{-1}$, for all $w\in W$, $\alpha\in\Pi$.
    \end{itemize}
\end{prop}

\begin{definition}
    The \emph{root lattice} is the free abelian group
    \begin{equation}
    Q=  \langle \alpha_0, \ldots, \alpha_n \rangle_{\BZ} \subset \mathfrak{h}^*,
    \end{equation}
    equipped with the symmetric bilinear form ${(\,\,\,|\,\,\,)}$ defined in \Cref{def:bil}, which is $\BZ$-valued on $Q$. 
\end{definition} 
The Weyl group $W(A)$ associated with $A$ of affine type is of infinite order and contains a subgroup of translations, which corresponds to a sublattice of $Q$ associated with an underlying finite root system. 
This allows Kac's formalism to recover the classical definition of an affine Weyl group. 
In this formulation, one takes the Weyl group associated to a finite root system realised in a Euclidean space as reflections about hyperplanes through the origin orthogonal to the simple roots, and extends this to include reflections about certain affine hyperplanes \cite{HumphreysReflectionGroups}. This can also be seen as extending the finite Weyl group by translations corresponding to its root lattice, see \Cref{prop:affineWeylgroup} below.

\begin{lemma} 
    The matrix $\nought{A}= \left(A_{ij} \right)_{i,j=1}^n$ obtained by deleting the $0$-th row and column from $A$ is a generalised Cartan matrix of finite type.
    From the realisation of $A$, we obtain one for $\nought{A}$, denoted by $\left( \nought{\mathfrak{h}}, \nought{\Pi}, \nought{\Pi}^{\vee} \right)$, where 
    \begin{equation}
    \begin{aligned}
    \nought{\Pi} &= \Set{ \alpha_1 , \ldots , \alpha_n }\!,  &\quad \nought{\mathfrak{h}}^* &= \Span_{\BC}\Set{ \alpha_1, \ldots , \alpha_n  }\!, \\
    \nought{\Pi}^{\vee} &= \Set{ \alpha^{\vee}_1 , \ldots , \alpha^{\vee}_n }\!, &\quad \nought{\mathfrak{h}} &= \Span_{\BC}\Set{ \alpha^{\vee}_1, \ldots , \alpha^{\vee}_n  }\!.
    \end{aligned}
    \end{equation}
\end{lemma}

\begin{definition}\label{def:underlyingfiniteweylgroupandsetofrootsfinite}
    The subgroup 
    $$\nought{W} \defeq W(\nought{A}) =\langle r_1, \ldots, r_n \rangle  \subset W(A)$$ 
    is the \emph{underlying finite Weyl group} of $A$.  
    The \emph{underlying finite root system} of $A$, or the \emph{root system of} $\nought{A}$ is 
    $$\nought{\Phi} = \nought{W}\, \nought{\Pi} = \Set{ w(\alpha_i) ~|~ {w} \in \nought{W}, \alpha\in \nought{\Pi} }\!.$$
\end{definition}
The underlying finite Weyl group $\nought{W}$ is finite, as is the set $\nought{\Phi}$ of roots.
Once the enumeration of simple roots $\nought{\Pi}$ is fixed, any $\alpha \in \nought{\Phi}$ can be written as a linear combination of elements of $\nought{\Pi}$ with coefficients being either all non-negative integers or all non-positive integers, which gives us the decomposition
$$ \nought{\Phi} = \nought{\Phi}^+ \amalg \nought{\Phi}^-,$$
into the set $\nought{\Phi}^+$ of \emph{positive roots} and the set $\nought{\Phi}^-$ of \emph{negative roots}.

This induces a partial ordering on $\nought{\Phi}$ defined by 
    \begin{equation}\label{eq:orderingroots}
    \alpha\prec \alpha'\mbox{ if } \alpha'-\alpha \in \nought{\Phi}^+,
    \end{equation} 
which we remark is similar in spirit to the partial order on $\Div(X)$ defined in terms of $\Eff(X)$ in \Cref{eq:partialordereffectivedivisors}.

There is a unique \emph{highest root} $\theta$ of $\nought{\Phi}$ with respect to the ordering in \Cref{eq:orderingroots}, given by
\begin{equation}
\theta = \delta - \alpha_0 = \sum_{i=1}^l m_i \alpha_i,
\end{equation}
where $m_i$ are the same as in the expression \eqref{nullroot} for the null root. 
By composing reflections associated with $\theta$ and $\alpha_0$ one obtains
\begin{equation} \label{eq:r0rtheta}
r_0 r_{\theta}(\lambda) = \lambda + (\lambda | \delta) \theta - \left[ (\lambda | \theta) + (\lambda | \delta )\right] \delta,
\end{equation}
for $\lambda\in \mathfrak{h}^*$, where we have used the fact that $(\theta | \theta) = 2$, which follows from \Cref{prop:weylgroupproperties} and the definition of $\nought{\Phi}$ given in \Cref{def:underlyingfiniteweylgroupandsetofrootsfinite}. 
The element $r_0 r_{\theta}\in W$ is of infinite order, and the formula \eqref{eq:r0rtheta} motivates the following definition.

\begin{definition} \label{def:kactranslation}
    For $v \in \nought{\mathfrak{h}}^*$, the \emph{Kac translation} associated to $v$ is the element $T_v \in \GL(\mathfrak{h}^*)$ defined by the \emph{Kac translation formula}
\begin{equation} \label{kactranslation}
T_{v}(\lambda) = \lambda + (\lambda | \delta ) v - \left[ (\lambda | v) + (\lambda | \delta ) \frac{(v | v)}{2} \right] \delta.
\end{equation} 
\end{definition}

\begin{prop} \label{prop:kactranslationproperties}
    The Kac translation has the following properties:
    \begin{itemize}
        \item $T_u T_v = T_{u+v}$ for any $u,v \in \nought{\mathfrak{h}}^*$,
        \item $T_{w(v)} = w T_{v} w^{-1}$ for any $u, v \in \nought{\mathfrak{h}}^*$, $w \in \nought{W}$.
    \end{itemize}
\end{prop}

\begin{remark}
For $\beta \in \mathfrak{h}^*$ such that $(\beta | \delta) = 0$, we have
\begin{equation} \label{translationorth}
T_v(\beta) = \beta - (\beta | v) \delta,
\end{equation}
so the properties in \Cref{prop:kactranslationproperties} can be deduced on this part of $\mathfrak{h}^*$ using the $W$-invariance of the symmetric bilinear form and the fact that $\delta$ is fixed by all $w\in W$, as in \Cref{prop:weylgroupproperties}.
The extra terms in formula \eqref{kactranslation} ensure that these properties hold on the rest of $\mathfrak{h}^*$, which is spanned by $\Lambda_0$. 
Indeed, for $v \in \nought{\mathfrak{h}}^*$, we have
\begin{equation}
T_{v}(\Lambda_0) = \Lambda_0 + v - \frac{(v | v)}{2} \delta,
\end{equation}
where we have used $(v | \Lambda_0) = 0$ as $v \in \nought{\mathfrak{h}}^*$, and $(\Lambda_0 | \delta) = 1$. 
\end{remark}

As anticipated, the Kac translation allows us to describe $W$ as an extension of $\nought{W}$. 
The elements of $\nought{\mathfrak{h}}^*$ whose associated Kac translations will form this extension are given in the following definition.

\begin{definition}
    The root lattice of the underlying finite root system is 
    \begin{equation}
        \nought{Q} = \langle \alpha_1,\ldots,\alpha_n\rangle_{\BZ} \subset \nought{\mathfrak{h}}^*,
    \end{equation}
    equipped with the symmetric bilinear form ${(\,\,\,|\,\,\,)}$ restricted to $\nought{\mathfrak{h}}^*$.
\end{definition}

Because of the properties of Kac translations given in \Cref{prop:kactranslationproperties} and the fact that $T_{\theta} = r_0r_{\theta} \in W$, we get a normal subgroup
\begin{equation}
T_{\nought{Q}} = \Set{ T_v ~~|~~ v \in \nought{Q} } \lhd  W,
\end{equation}
and every $T_{v} \in T_{\nought{Q}}$ preserves the root lattice $Q$. 

\begin{prop} \label{prop:affineWeylgroup}
    We have an isomorphism
    \begin{equation}
W \cong \nought{W} \ltimes T_{\nought{Q}}, 
\qquad r_0 \mapsto r_{\theta} T_{-\theta}, \quad r_i\mapsto r_i, ~~i=1,\ldots,n.
\end{equation} 
\end{prop}
In the theory of discrete Painlev\'e equations, it will not be just the affine Weyl group $A$ that is relevant, but its extension by \emph{Dynkin diagram automorphisms}.
\begin{notation}
    From now on, we consider the action of $W$ restricted to $Q \subset \mathfrak{h}^*$ and denote the group of $\BZ$-module automorphisms of $Q$ preserving the symmetric bilinear form {$(\,\,\,|\,\,\,)$} by $\Aut(Q)$.
\end{notation}

\begin{definition} \label{def:dynkinauto}
    For matrix $A$ with realisation as in \Cref{not:GCMaffinesimplylaced}, a \emph{Dynkin diagram automorphism} is a graph automorphism of the Dynkin diagram associated to $A$. 
    In other words it is a permutation $\sigma$ of the indices $0,\ldots,n$ such that 
    \begin{equation}
        \langle \alpha_{\sigma(i)}^{\vee}, \alpha_{\sigma(j)} \rangle = \langle \alpha_{i}^{\vee}, \alpha_{j}\rangle,
    \end{equation}
    for all $i,j=0,\ldots,n$. 
    A Dynkin diagram automorphism $\sigma$ defines an element of $\Aut(Q)$, which we also denote by $\sigma$, defined by 
    $\sigma(\alpha_i)=\alpha_{\sigma(i)}$.
    We denote the group of Dynkin diagram automorphisms by $\Aut(A)$.
\end{definition}

Dynkin diagram automorphisms appear naturally when one notices that, for $Q$ to be preserved by a Kac translation $T_v$, we may choose $v$ from a lattice finer than $\nought{Q}$.
\begin{definition}
    The \emph{weight lattice} of the underlying finite root system is
\begin{equation}
\nought{P} = \langle \omega_1,\dots,\omega_n\rangle_{\BZ} \subset \nought{\mathfrak{h}}^*,
\end{equation}
    where $\omega_1, \ldots, \omega_n$ are the \emph{fundamental weights}, defined as the dual basis to $\nought{\Pi}^{\vee}$ with respect to the evaluation pairing, so $\langle \omega_i , \alpha_j^{\vee} \rangle = \delta_{i,j}$ for $i,j=1,\ldots n$. 
\end{definition}

Note that in the case when the generalised Cartan matrix $A$ is symmetric, as we are considering here, $\nought{P}$ is the maximal set of elements $v\in \nought{\mathfrak{h}}^*$  such that $\left(\beta | v \right) \in \BZ $ for all $\beta \in Q$.
The Kac translation associated with any $\omega \in \nought{P}$ preserves $Q$ and we get a normal subgroup
\begin{equation}
T_{\nought{P}} = \Set{ T_v ~~|~~ v \in \nought{P} } \lhd \Aut(Q).
\end{equation}

\begin{definition}
    The \emph{extended affine Weyl group} of $A$ is defined as 
    \begin{equation}
    \widetilde{W} = \nought{W} \ltimes T_{\nought{P}} \subset \Aut(Q),
    \end{equation}
    where the semi-direct product structure comes from the fact that $T_{w(v)}=w T_v w^{-1}$ as in \Cref{prop:kactranslationproperties}.
\end{definition}
\begin{prop} \label{prop:specialdynkinautos}
    There is an isomorphism 
    \begin{equation}
    \widetilde{W} \cong W \rtimes \Sigma,
    \end{equation}
    where $\Sigma \cong T_{\nought{P}}/T_{\nought{Q}}\cong \nought{P} / \nought{Q} \subset \Aut(A)$ is a subgroup of the group of  Dynkin diagram automorphisms.
    We list the groups $\Sigma$, for $A$ as in \Cref{not:GCMaffinesimplylaced}, in \Cref{Sigmatable}.
\end{prop}
    
\renewcommand{\arraystretch}{2}
\begin{table}[htb]
    \centering
\begin{tabular}{ c | c | c | c | c | c | c }
 		& $A_n^{(1)}$ & $D_{2n}^{(1)}$ & $D_{2n+1}^{(1)} $ & $E_6^{(1)}$ & $E_7^{(1)}$ & $E_8^{(1)}$ \\  \hline
 $\Sigma$     & $\BZ / (n+1)\BZ$ & $(\BZ / 2\BZ) \times (\BZ / 2\BZ)$ & $\BZ / 4 \BZ$ & $\BZ / 3\BZ$ & $\BZ / 2\BZ$ & $-$ 
 \end{tabular}
    \caption{Special Dynkin diagram automorphism groups for symmetric generalised Cartan matrices of affine type.}
 \label{Sigmatable}
\end{table} 

\renewcommand{\arraystretch}{1}

\begin{definition} \label{def:fullyextendedaffineWeylgroup}
    The \emph{(fully) extended Weyl group} of $A$ is 
    \begin{equation}
        \widehat{W}(A) = W(A) \rtimes \Aut(A),
    \end{equation}
    where the semi-direct product structure is defined by $r_{\sigma(i)} = \sigma \,r_i \, \sigma^{-1}$, where $\sigma\in \Aut(A)$.
\end{definition}
To summarise, we have the following inclusions and isomorphisms among subgroups of $\Aut(Q)$:
    \begin{equation}
        \nought{W} \subset W \cong \nought{W}\ltimes T_{\nought{Q}} \subset\nought{W}\ltimes T_{\nought{P}} \cong W \rtimes \Sigma = \widetilde{W} \subset W \rtimes \Aut(A) = \widehat{W}.
        \end{equation}

\subsubsection{Surface and symmetry root sublattices of $Q(E_8^{(1)})$ in $\Pic(S)$}

We are now ready to describe the way that Sakai surfaces are associated to affine root systems.
We begin, following \cite{Sakai2001}, with the observation of the root lattice $Q(E_8^{(1)})$ naturally appearing in relation to generalised Halphen surfaces. 

\begin{prop}[{\cite[Prop.7, Prop. 8]{Sakai2001}}] \label{prop:E8latticeinPic}
    Let $S$ be a generalised Halphen surface. 
    Then, the Picard group of $S$ is isomorphic, when equipped with the symmetric bilinear form ${(F_1 | F_2)}= - F_1.F_2$, to the Lorentzian lattice of rank 10:
    \begin{equation} \label{eq:lorentzianlattice}
        \Lambda_{10} := \langle v_0,v_1,\ldots,v_9\rangle_{\BZ}, \qquad \left( v_0 | v_0 \right)= -1, ~  \left( v_i|v_0 \right)= 0, ~ \left( v_i|v_j \right)= \delta_{i,j},
    \end{equation}
    for $i,j=1,\ldots,9$. 
    
    Further, $K_S^{\perp} \defeq \{ F\in \Pic(S)~|~ F.K_S=0 \} \subset \Pic(S)$ is  isomorphic to the root lattice $Q(E_8^{(1)})$.
    Further, the isomorphism can be chosen such that $-K_S$ is identified with the null root $\delta\in Q(E_8^{(1)})$.
\end{prop}
\begin{proof}
    From \Cref{prop:rankPic10} we know that $S$ can be obtained from $\BP^2$ through a sequence of 9 blow-ups each centred at a point. Then, by \Cref{lemma:eltransf} we have that 
    \begin{equation}
        \Pic(S) \cong \langle H, E_1, E_2, E_3, E_4, E_5, E_6,E_7, E_8, E_9 \rangle_{\BZ}, 
    \end{equation}
    and from \Cref{th:blow-upformula} we have 
    \begin{equation}
        -K_S = 3 H - E_1 - E_2 - E_3 -E_4 -E_5 -E_6 - E_7 -E_8 - E_9.
    \end{equation}
    Then, we have an isomorphism $\Pic(S)\cong \Lambda_{10}$, $H\mapsto v_0$, $E_i\mapsto v_i$ for $i=1,\ldots,9$.
Consider the enumeration of the nodes in the $E_8^{(1)}$ Dynkin diagram given in \Cref{fig:placeholder2}.
    \begin{figure}[htb]
        \centering \begin{tikzpicture}[elt/.style={circle,draw=black!100,thick, inner sep=0pt,minimum size=2mm},scale=0.8]
 		\path 	(-2,0) 	node 	(d1) [elt, label={[xshift=0pt, yshift = -21 pt] $1$} ] {}
                    (-1,0) 	node 	(d2) [elt, label={[xshift=0pt, yshift = -21 pt] $2$} ] {}
                    (0,0) 	node 	(d3) [elt, label={[xshift=0pt, yshift = -21 pt] $3$} ] {}
                    (1,0) 	node 	(d4) [elt, label={[xshift=0pt, yshift = -21 pt] $4$} ] {}
                    (2,0) 	node 	(d5) [elt, label={[xshift=0pt, yshift = -21 pt] $5$} ] {}
                    (3,0) 	node 	(d6) [elt, label={[xshift=0pt, yshift = -21 pt] $6$} ] {}
                    (4,0) 	node 	(d7) [elt, label={[xshift=0pt, yshift = -21 pt] $7$} ] {}
                    (5,0) 	node 	(d0) [elt, label={[xshift=0pt, yshift = -21 pt] $0$} ] {}
                    (0,1) 	node 	(d8) [elt, label={[xshift=10pt, yshift = -10 pt] $8$} ] {}
 		       ;
 		\draw [black,line width=1pt ] (d1) -- (d2) -- (d3) -- (d4) -- (d5) -- (d6) -- (d7) -- (d0)  (d3) -- (d8);
 	\end{tikzpicture}
        \caption{An enumeration of the nodes in the $E_8^{(1)}$ Dynkin diagram.}
        \label{fig:placeholder2}
    \end{figure}
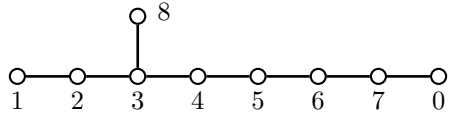 
    An isomorphism between $K_S^{\perp} \subset  \Pic(S)$ and $Q(E_8^{(1)}) = \langle\alpha_0, \alpha_1,\ldots,\alpha_8\rangle_{\BZ}$ can be chosen to be that given by the identification
    \begin{equation} \label{eq:E8simplerootbasis}
        \alpha_0 = E_8-E_9, \quad \alpha_i = E_i - E_{i+1} ~~\text{ for } ~~ i=1,\ldots, 7, \quad \alpha_8 = H - E_1 -E_2 -E_3,
    \end{equation}
    which in particular identifies $-K_S$ with the null root for $E_8^{(1)}$:
    \[
    -K_S = 3 H - E_1-\cdots-E_9 = \sum_{i=0}^{8} m_i \alpha_i =\delta, 
    \]
    with $(m_0,m_1,\ldots,m_8)=(1,2,4,6,5,4,3,2,3)$.
    Note that this isomorphism is unique only up to automorphisms of the lattice $Q(E_8^{(1)})$, which form the group $\pm W(E_8^{(1)})$\footnote{This is the group of automorphisms of $Q(E_8^{(1)})$ formed of $W(E_8^{(1)})$ as well as compositions of the action of $w\in W(E_8^{(1)})$ with the automorphism $Q(E_8^{(1)})\to Q(E_8^{(1)})$, $\alpha\mapsto-\alpha$.}, see \cite[Ex. 5.8]{KAC1990}. 
\end{proof}

For a Sakai surface $S$ with unique effective anti-canonical divisor $D=\sum_{i} m_i D_i \in |-K_S|$, we can already see from \Cref{prop:ACdivconnected} and \Cref{prop:Sakaisurface-2curves} that, when $D$ is not irreducible, the matrix having entries $D_i.D_j$ is a generalised Cartan matrix of affine type.
We again emphasise that we denote by the same symbol the divisor $D_i \in \Div(S)$ and the corresponding element $D_i \in \Pic(S)$.

\begin{definition}[Surface root lattice] \label{def:surfacerootlattice}
Let $S$ be a Sakai surface with unique effective anti-canonical divisor $D=\sum_i m_i D_i \in |-K_S|$.
Then, the $\BZ$-$\Span$ of the classes of components $D_i\in \Pic(S)$ is isomorphic, when equipped with the same symmetric bilinear form {$(\,\,\,|\,\,\,)$} as in \Cref{prop:E8latticeinPic},
to the root lattice of an affine root system of some type $\mathcal{R}$.  
We call this free $\BZ$-module the 
\emph{surface root lattice} $Q(\mathcal{R}) = \sum_{i} \BZ D_i \subset K_S^{\perp} \subset \Pic(S)$. 
Here $\mathcal{R}$ indicates the type of the corresponding generalised Cartan matrix.
The components $D_i$ define a basis $\Pi$ of simple roots for $Q(\mathcal{R})$, which we call the \emph{surface root basis}.
Note that when $D$ is irreducible, the type $\mathcal{R}=A_0^{(1)}$ is assigned to the lattice $\BZ D$.
When no confusion is possible, we will sometimes write $Q=Q(\mathcal{R})$. 
\end{definition}

\begin{definition}[Symmetry root lattice] \label{def:symmetryrootlattice}
   Let $S$ be a Sakai surface with unique effective anti-canonical divisor $D=\sum_i m_i D_i \in |-K_S|$. 
   The orthogonal complement 
    $$Q(\mathcal{R}^{\perp}) \defeq \Set{F\in\Pic(S)~|~ F.D_i=0 ~\text{ for all } i\!}\!,$$
    of $Q(\mathcal{R})$ in $K_S^{\perp} \subset \Pic(S)$ is isomorphic, when equipped with {$(\,\,\,|\,\,\,)$},
    to a root lattice of an affine type $\mathcal{R}^{\perp}$, which we call the \emph{symmetry root lattice}. 
    A choice of basis of simple roots $\Pi=\Set{\alpha_0,\dots,\alpha_n} \subset Q(\mathcal{R}^{\perp})$ 
    for the root lattice $Q(\mathcal{R}^{\perp})$ is called a \emph{symmetry root basis}.
    When the symmetry root lattice contains only multiples of $D$, i.e. $Q(\mathcal{R}^{\perp})=\BZ D$, it is assigned the type $\mathcal{R}^{\perp}=A_0^{(1)}$.
    We will sometimes write $Q^{\perp}=Q(\mathcal{R}^{\perp})$.
\end{definition}

Given the result of \Cref{prop:E8latticeinPic}, the possible pairs $(\mathcal{R},\mathcal{R}^{\perp})$ for Sakai surfaces are dictated by the possible complementary root sublattices in $Q(E_8^{(1)})$.
These are given in \Cref{Rtable} and \Cref{Rperptable} respectively, with the types corresponding to Painlev\'e differential equations, as in \Cref{table:surfacetypes}, indicated by boxes.  

\begin{remark}
    In some cases, namely when $\mathcal{R}=A_6^{(1)}, A_7^{(1)'}, D_7^{(1)}$, the symmetry root lattice $Q(\mathcal{R}^{\perp})$ is realised in $\Pic(S)$ with roots of non-standard lengths, which are indicated in the notation for the type $\mathcal{R}^{\perp}$ with ${|\alpha|^2 = (\alpha | \alpha)}$ being the squared length of roots.
    The arrows in \Cref{Rtable} and \Cref{Rperptable} represent surface degenerations in the sense of \textit{Rains} \cite{rains}.
\end{remark}

\begin{figure}[htb] 
\centering
\begin{tikzcd}[row sep=normal, column sep=small, every matrix/.append style={nodes={font=\small}},
/tikz/execute at end picture={
    \node (large) [rectangle, draw, fit=(D4) (D5) (D6) (D7) (D8) (E8) (E7) (E6) ] {};
  }] 
A_0^{(1)} \arrow[d] & & & & & & &  A_7^{(1)'} \arrow[rdd]& \\ 
A_0^{(1) *} \arrow[r] \arrow[rd] & A_1^{(1)} \arrow[r] \arrow[rd]&  A_2^{(1)} \arrow[r] \arrow[rd]&  A_3^{(1)} \arrow[r] \arrow[rd]& A_4^{(1)} \arrow[r] \arrow[rd]& A_5^{(1)} \arrow[r] \arrow[rd]\arrow[rdd] & A_6^{(1)} \arrow[r] \arrow[rd]\arrow[rdd] \arrow[ru]& A_7^{(1)} \arrow[r] \arrow[rd] \arrow[rdd] &  A_8^{(1)} \\ 
 & A_0^{(1) * *} \arrow[r] &  A_1^{(1) *} \arrow[r] & A_2^{(1) *} \arrow[r] & |[alias=D4]| D_4^{(1)} \arrow[r] & |[alias=D5]| D_5^{(1)} \arrow[r] \arrow[rd]& |[alias=D6]| D_6^{(1)} \arrow[r] \arrow[rd]& |[alias=D7]| D_7^{(1)} \arrow[r]  \arrow[rd]& |[alias=D8]| D_8^{(1)} \\ 
 & & & & & & |[alias=E6]| E_6^{(1)} \arrow[r] & |[alias=E7]| E_7^{(1)} \arrow[r] & |[alias=E8]| E_8^{(1)}
\end{tikzcd}
\caption{Surface types $\mathcal{R}$ for Sakai surfaces.} 
\label{Rtable}
\end{figure}
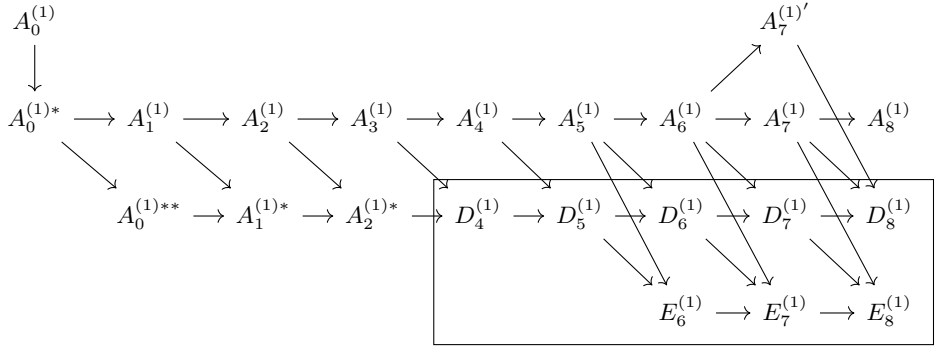

\begin{figure}[htb] 
\centering
\begin{tikzcd}[row sep=small, column sep=tiny, every matrix/.append style={nodes={font=\small}},
/tikz/execute at end picture={
    \node (large) [rectangle, draw, fit=(D4) (D5) (D6) (D7) (D8) (E8) (E7) (E6) ] {};
  }]
E_8^{(1)} \arrow[d] & & & & & & &  \underset{|\alpha|^2=8}{A_1^{(1)}} \arrow[rdd]& \\
E_8^{(1)} \arrow[r] \arrow[rd] & E_7^{(1)} \arrow[r] \arrow[rd]&  E_6^{(1)} \arrow[r] \arrow[rd]&  D_5^{(1)} \arrow[r] \arrow[rd]& A_4^{(1)} \arrow[r] \arrow[rd]&(A_2+A_1)^{(1)} \arrow[r] \arrow[rd]\arrow[rdd] & ( \underset{|\alpha|^2=14}{A_1}+A_1)^{(1)} \arrow[r] \arrow[rd]\arrow[rdd] \arrow[ru]&A_1^{(1)} \arrow[r] \arrow[rd] \arrow[rdd] & A_0^{(1)} \\
 & E_8^{(1)} \arrow[r] &  E_7^{(1)} \arrow[r] & E_6^{(1)} \arrow[r] & |[alias=D4]| D_4^{(1)} \arrow[r]& |[alias=D5]| A_3^{(1)} \arrow[r] \arrow[rd]& |[alias=D6]| (A_1+A_1)^{(1)}\arrow[r] \arrow[rd]& 
 |[alias=D7]| \underset{|\alpha|^2=4}{A_1^{(1)}} \arrow[r]  \arrow[rd]& |[alias=D8]| A_0^{(1)} \\
 & & & & & & |[alias=E6]| A_2^{(1)} \arrow[r] & |[alias=E7]| A_1^{(1)} \arrow[r] & |[alias=E8]|A_0^{(1)}
\end{tikzcd}
\caption{Symmetry types $\mathcal{R}^{\perp}$ for Sakai surfaces.} 
\label{Rperptable}
\end{figure}
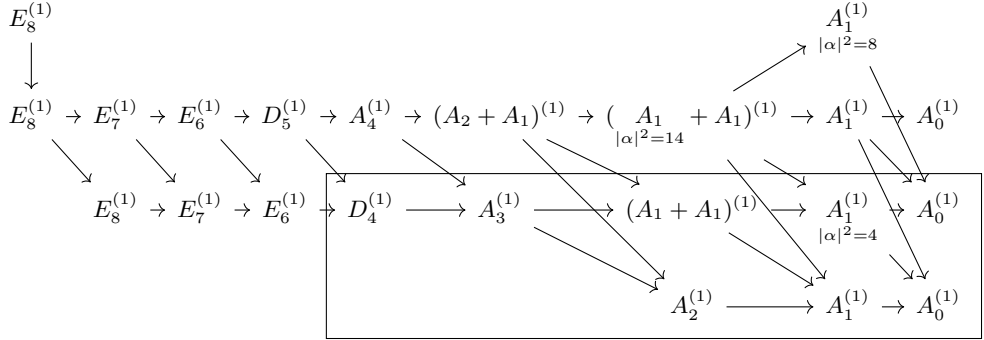

\subsubsection{Additive, multiplicative and elliptic types}

In \Cref{def:surfacerootlattice}, for a Sakai surface $S$ we introduced the surface type $\mathcal{R}$ as the type of the root lattice spanned by the components $D_i$ of $D\in|-K_S|$. 
We now, as in the classification in \cite{Sakai2001}, enrich the surface type $\mathcal{R}$ with the information of $h_1(D_{\red},\BZ) \defeq \rk H_1(D_{\red},\BZ)$, where we denote the support of $D$, following \cite{Sakai2001}, by $D_{\red}=\supp D$. 
This leads to the surface types as shown in \Cref{Rtable}.
From the classification, for any Sakai surface we have $h_1(D_{\red},\BZ) \in\Set{0,1,2}$. These three possibilities distinguish surfaces associated with discrete Painlev\'e equations of additive, multiplicative, and elliptic type, respectively. 
\begin{definition} \label{def:additivemultiplicativeelliptic}
    Let $S$ be a Sakai surface, $D = \sum_i m_i D_i \in|-K_S|$ its unique effective anti-canonical divisor with support $D_{\red}=\cup_iD_i$. 
    Then, the surface $S$ is of
    \begin{itemize}
        \item \emph{additive type} if $h_1(D_{\red},\BZ)=0$,
        \item \emph{multiplicative type} if $h_1(D_{\red},\BZ)=1$,
        \item \emph{elliptic type} if $h_1(D_{\red},\BZ)=2$.        
    \end{itemize}
\end{definition}

\begin{example}
    The surface types $A_0^{(1)}$, $A_0^{(1)*}$ and $A_0^{(1)**}$ (in the notation of \cite{Sakai2001} as in \Cref{Rtable}) refer to the cases when $D$ is an elliptic curve ($h_1(D_{\red},\BZ)=2$), a rational curve with a node ($h_1(D_{\red},\BZ)=1$) and a rational curve with a cusp ($h_1(D_{\red},\BZ)=0$), respectively. 
    In all these cases, the type of the root lattice as in \Cref{def:surfacerootlattice} is $A_0^{(1)}$.
\end{example}

\subsection{Symmetries of Sakai surfaces and discrete Painlev\'e equations}\label{subsec:symmetries}  

Discrete Painlev\'e equations are defined in terms of symmetries of Sakai surfaces. 
The symmetries in question form (fully) extended affine Weyl groups associated with the symmetry type $\mathcal{R}^{\perp}$.
We will present two descriptions of such symmetries and the discrete Painlev\'e equations they define.

Before this, we note that discrete Painlev\'e equations are \emph{non-autonomous} systems. 
Rather than being defined, for example, by a single birational transformation $\varphi\in \Bir(\BP^2)$, they should be understood as a pair $(\varphi_{\boldsymbol{a}}, \boldsymbol{a}\mapsto \bar{\boldsymbol{a}})$ consisting of a parametric family of birational transformations $\varphi_{\boldsymbol{a}} : \BP^2 \dashrightarrow \BP^2$ and a \emph{parameter evolution} $\boldsymbol{a}\mapsto \bar{\boldsymbol{a}}$.
Note that the perhaps more familiar notion of a non-autonomous difference equation defined by a sequence $\varphi_n$ of mappings indexed by $n \in \BZ$ fits into this framework, with parameter evolution induced by $n \mapsto n+1$.

The first description we will outline makes use of a single Sakai surface $S$ with the extra data of a \emph{blowing-down structure}, i.e. a way to blow-down $S$ to $\BP^2$ via a birational morphism $\pi : S \rightarrow \BP^2$. 
A symmetry in this setting is then described at the level of $\Pic(S)$ as a \emph{Cremona isometry}, defining a change of blowing-down structure. 
Given two blowing-down structures for $S$ with morphisms $\pi$, $\pi'$, the difference between the two ways to blow-down $S$ to $\BP^2$ gives a discrete Painlev\'e equation as a birational map $\pi' \circ \pi^{-1} : \BP^2 \dashrightarrow \BP^2$. The non-autonomous nature of the discrete Painlev\'e equations comes from how the \emph{root variables} $\boldsymbol{a}$, which are defined by a choice of simple roots for $Q(\mathcal{R}^{\perp})$, change under the Cremona isometry. We will explain the definition of root variables below, but at this point it is sufficient to understand them as entries $a_i$ of a tuple $\boldsymbol{a}$ of parameters, taking values in either $\BC$ for additive type, $\BC  \mod 2 \pi i\BZ$ for multiplicative type, or $\BC  \mod \BZ+\tau \BZ$ for elliptic type.

The second description considers a realisation of a \emph{family} of Sakai surfaces $S_{\boldsymbol{a}}$ of a given type, parametrised by root variables $\boldsymbol{a}$.
In this framework, a symmetry is defined as an automorphism of the family, consisting of an action $\boldsymbol{a} \rightarrow \bar{\boldsymbol{a}}$ on the parameters together with a family of isomorphisms $S_{\boldsymbol{a}} \rightarrow S_{\bar{\boldsymbol{a}}}$ between the corresponding surfaces. This determines a parametric family of transformations $\varphi_{\boldsymbol{a}}$ of $\BP^2$, represented by the pair $(\varphi_{\boldsymbol{a}}, \boldsymbol{a} \rightarrow \bar{\boldsymbol{a}})$.

\subsubsection{Cremona isometries and changes of blowing-down structures}\label{subsec:changeblowstruc}

The first description of symmetries of Sakai surfaces requires the following.

\begin{definition}[Geometric basis; blowing-down structure]\label{def:geobasis}
    Let $S$ be a Sakai surface. 
    A $\BZ$-basis $\mathcal{E}=(H,E_1,\ldots,E_{9})$ of $\Pic(S)$ is called a \emph{geometric basis} if there exists a birational morphism $\varepsilon : S \rightarrow \BP^2$ written as a composition   blow-ups $\varepsilon = \pi_{1} \circ \cdots \circ \pi_{9}$ each centred at a point,  such that $H$ is the pull-back under $\varepsilon$ of the class of a hyperplane in $\BP^2$ and $E_{i}$, for $1\leq i\leq 9$, is the class of the exceptional divisor contracted by $\pi_i$.
    We call $\mathcal{E}$ the geometric basis for $\Pic(S)$ associated to $\varepsilon$, and we call the pair $(\varepsilon,\mathcal{E})$ a \emph{blowing-down structure} for $S$.
\end{definition}

It is worth mentioning that in \cite{Sakai2001} the data carried by a geometric basis is called a \emph{strict geometric marking} of $\Pic(S)$, whereas the term geometric basis is used in \cite{Mase}. 
A change of blowing-down structure for a Sakai surface $S$ can be described in terms of an automorphism of $\Pic(S)$ that relates the corresponding geometric bases. 
We recall the following definition, the terminology for which we take from \cite{Sakai2001,dolgachevortland,dolgachevweylgroup, looijenga}.

\begin{definition}[Cremona isometry] \label{def:Cremonaisometry}
For a surface $S$, a \emph{Cremona isometry} is an invertible $\BZ$-linear map $\sigma : \Pic(S)\rightarrow \Pic(S)$ such that 
\begin{itemize}
    \item it preserves the intersection form, i.e. $\sigma(F_1).\sigma(F_2)= F_1 .F_2$ for all $F_1,F_2\in \Pic(S)$, 
    \item it preserves the canonical class, i.e. $\sigma(K_S)=K_S$,
    \item it preserves effectiveness of divisor classes, i.e. if $F \in \Eff(S)$ is effective, then also $\sigma(F)\in\Eff(S)$.
\end{itemize}
The Cremona isometries for a surface $S$ form a group, which we denote by $\operatorname{Cr}(S)$. 
\end{definition}

\begin{definition} \label{def:changeofblowingdownstructure}
Suppose a Sakai surface $S$ has two blowing-down structures $(\varepsilon,\mathcal{E})$ and $(\varepsilon',\mathcal{E}')$.
Then the \emph{change of blowing-down structure} from $(\varepsilon,\mathcal{E})$ to $(\varepsilon',\mathcal{E}')$ is the data of the birational transformation $\varepsilon' \circ \varepsilon^{-1}  : \BP^2 \dashrightarrow \BP^2$ and the lattice automorphism $\sigma$ of $\Pic(S)$ such that $\sigma(\mathcal{E})=\mathcal{E}'$, see \Cref{fig:changeofblowingdownstructure}.
\end{definition}

\begin{figure}[htb]
    \centering
    \begin{tikzcd}
        & S \arrow[dl, swap, "\varepsilon"] \arrow[dr, "\varepsilon'"] & \\
        \BP^2 \arrow[rr, dashed, "\varepsilon' \circ \varepsilon^{-1} "] & & \BP^2.
    \end{tikzcd}
    \caption{Change of blowing-down structure.}
    \label{fig:changeofblowingdownstructure}
\end{figure}
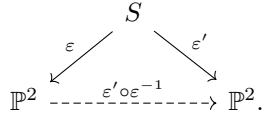 
We made \Cref{def:changeofblowingdownstructure} for Sakai surfaces, but it can be given in more general settings, e.g. \cite{dolgachevortland}.
See \Cref{exa:CREMONA} for an example of a change of blowing-down structure.

For a Sakai surface $S$ of surface type $\mathcal{R}$, Sakai gave a description  of $\operatorname{Cr}(S)$ in terms of the affine Weyl group of the root system of $Q(\mathcal{R}^{\perp})$, and constructed corresponding changes of blowing-down structures.

However, the group $\Cr(S)$ is not the same for all surfaces of a given type $\mathcal{R}$, and it depends on whether there are \emph{nodal curves} on the surface, which we describe now. These are related to special solutions of differential and discrete Painlev\'e equations that exist for particular parameter values, given in terms of classical special functions of hypergeometric type, see \cite{KNY} and \cite{SaitoTerajimiNodalCurves}.
\begin{definition}[Nodal curves] \label{def:nodalcurve}
    Let $S$ be a Sakai surface with unique effective anti-canonical divisor $D=\sum_i m_i D_i \in |-K_S|$. 
    A \emph{nodal curve} on $S$ is a smooth rational curve of self-intersection $-2$ which is not an irreducible component of $D$. 
    The set of classes of nodal curves is denoted by $\Delta^{\operatorname{nod}}\subset \Pic(S)$. 
\end{definition}

We now state part of Sakai's description of $\Cr(S)$.

\begin{theorem}{\cite[Th. 26]{Sakai2001}} \label{thm:cremonaisometries}
    For a Sakai surface $S$ of surface type $\mathcal{R}\neq A_6^{(1)}, A_7^{(1)}, A_7^{(1)'}, A_8^{(1)}, D_7^{(1)}, D_8^{(1)}$ as shown in \Cref{Rtable}, 
    \begin{equation} \label{eq:Cr(S)}
        \Cr(S) \cong \left( W(\mathcal{R}^{\perp})\rtimes \Aut(\mathcal{R}^{\perp})\right)_{\Delta^{\operatorname{nod}}}.
    \end{equation} 
    \end{theorem}
   Let us explain the right-hand side of \Cref{eq:Cr(S)}. The group
   $W(\mathcal{R}^{\perp})$ is the Weyl group generated by the simple reflections $r_i$ corresponding to a symmetry root basis formed of $\alpha_i\in Q(\mathcal{R}^{\perp})$ as in \Cref{def:symmetryrootlattice}.
   For the surface types here, these roots have standard lengths, i.e. satisfy $\alpha_i.\alpha_i=-2$. They act on $\Pic(S)$ by 
   \begin{equation} \label{eq:refformulaPic}
      r_i (F)= F - (F | \alpha_i)\alpha_i = F + (F.\alpha_i )\alpha_i. 
   \end{equation}
   The Dynkin diagram automorphisms $\operatorname{Aut}(\mathcal{R}^{\perp})$ act by permutations of the symmetry root basis when restricted to $Q(\mathcal{R}^{\perp})$, and extend to $\Pic(S)$ according to the explicit expressions in \cite{Sakai2001}. 
   The subscript indicates the stabiliser, in the semi-direct product of these subgroups of lattice automorphisms of $\Pic(S)$, of the set $\Delta^{\operatorname{nod}}$ of classes of nodal curves.

\begin{remark}
    For surfaces $S$ of types $\mathcal{R}$ not accounted for in \Cref{thm:cremonaisometries}, the group $\Cr(S)$ can still be described in terms of $W(\mathcal{R}^{\perp})$ and its extension by Dynkin diagram automorphisms, but the realisation at the level of $\Pic(S)$ becomes subtle due to, for example, non-standard root lengths, see \cite{Sakai2001}. In order to ease the exposition, we refer to these groups also as $W(\mathcal{R}^{\perp})\rtimes\Aut(\mathcal{R}^{\perp})$, but note that the actions on $\Pic(S)$ may be slightly different to that explained above.
\end{remark}

\subsubsection{Period map and root variables}

In order for a change of blowing-down structure to define a non-autonomous discrete system as a pair $(\varphi_{\boldsymbol{a}}, \boldsymbol{a}\mapsto \bar{\boldsymbol{a}})$ along the lines discussed at the beginning of \Cref{subsec:symmetries}, we need to specify parameters and their evolution. 
For this we will introduce the \emph{root variables} for a Sakai surface, which are parameters for the set of isomorphism classes of Sakai surfaces of a fixed type, in a sense which is made precise in \cite[Th. 25]{Sakai2001} through a Torelli-type theorem.
Root variables are defined in terms of a kind of \emph{period map} for $S$, the construction of which is due to \textit{Looijenga} \cite{looijenga}. 
Looijenga defined  the period map in the case of a rational surface with an anti-canonical divisor given by a sum of rational curves whose intersection/dual  graph, meant in the same sense as in \Cref{ex:E8config}, consists of a cycle, which accounts for all Sakai surfaces of multiplicative type.

Let $S$ be a generalised Halphen surface, let $D= \sum_i m_i D_i\in|-K_S|$ be an effective anti-canonical divisor and denote by $D_{\red} $ its support $D_{\red}= \supp D$.  
Take a rational $2$-form $\omega$ on $S$ such that $- \divi \omega = D$, so $\omega$ defines a holomorphic symplectic form on $S \setminus D_{\red}$. 
This gives the mapping
\begin{equation} \label{eq:chihat} 
\begin{tikzcd}[row sep=tiny]
     H_2(S \,\setminus D_{\red} , \BZ)\arrow[r,"\hat{\chi} "]&\BC\\ \Gamma\arrow[r,mapsto]&\int_{\Gamma} \omega.
\end{tikzcd} 
\end{equation} 
Let $Q^{\perp}$ be the orthogonal complement in $\Pic(S)$ of the components of $D$ with respect to the intersection pairing.
The \emph{period map} for $S$ with respect to the pair\footnote{
Since not all generalised Halphen surfaces $S$ are Sakai surfaces, the choice of $D$ has been incorporated into the construction of the period map. 
    This is necessary to use the characterisation of generalised Halphen surfaces with $\dim |-K_S|=1$ as those with the value of the period map on $-K_S$ being zero, see \cite[Prop. 23]{Sakai2001}.} $(D,\omega)$ is the  function
\begin{equation}
\chi : Q^{\perp} \rightarrow \BC ~\mod \hat{\chi}(H_1(D_{\red},\BZ)),
\end{equation}
defined in terms of $\hat{\chi}$ using the long exact sequence in relative singular homology for the pair $(S, S \setminus D_{\red})$.
For the sake of brevity, we will not give the details of the construction of the period map.
Instead, we will outline how the period map is computed in practice, which proceeds along the same lines as in \cite[Chapter I, Section 5]{looijenga} and \cite[Lemma 21]{Sakai2001}. 

For an element $\alpha \in Q^{\perp}$  such that $\alpha.\alpha=-2$ and $\alpha \not\in \Delta^{\operatorname{nod}}$, express this as
\begin{equation}
\alpha = C^1 - C^0,
\end{equation}
where $C^1,C^0\in \Pic(S)$ correspond to exceptional curves in the sense of \Cref{def:exceptionalcurve}. 
This is guaranteed to be possible for all Sakai surfaces by \cite[Lem. 21 \& App. A]{Sakai2001}.
Then the computation of $\chi(\alpha)$ is done as follows.
\begin{itemize}
\item Find the unique $D_k$ among the irreducible components of $D$ such that 
\begin{equation}
C^1.D_k = C^0 . D_k = 1, \quad \text{ and } \quad C^1 . D_j = C^0 . D_j = 0\quad  \text{ for } j \neq k.
\end{equation}
\item  
The value of $\chi(\alpha)$ is then computed using the residue formula as
\begin{equation} \label{residueintegral}
\chi(\alpha) = 2 \pi i \int_{D_k \cap C^0}^{D_k \cap C^1} \operatorname{Res}_{D_k} \omega. 
\end{equation}
\end{itemize}
\begin{exercise}
  Show the existence and uniqueness of $D_k$ using the fact that $D$ is an effective anti-canonical divisor and $C^1$, $C^0$ are exceptional curves in the sense of \Cref{def:exceptionalcurve}, e.g. by using the genus formula \eqref{eq:genusformula}.
\end{exercise}

\begin{definition}[Root variables] \label{def:rootvariables}
    Let $S$ be a generalised Halphen surface and $D=\sum_i m_i D_i \in|-K_S|$. Fix a basis $\Pi = \Set{\alpha_0, \dots, \alpha_n}$ of simple roots for $Q^{\perp}$ and a choice of $\omega$ with $\divi \omega = - D$.
    The \emph{root variables} for the basis $\Pi$ are the values of the period map $\chi$ on $S$ with respect to $D$ and $\omega$ on the elements $\alpha_i$, i.e. 
    \begin{equation*}
         \chi(\alpha_i) \in \BC \mod  \hat{\chi}(H_1(D_{\red},\BZ)).
    \end{equation*}
\end{definition} 
The root variables provide a parametrisation of a family of surfaces of a given type via locations of points to be blown-up, or data for representing the surface as a gluing of affine open subsets, see \cite[Sec. 5]{Sakai2001}.
These parameters, denoted by $a_i$, are related to the values of the period map in one of the following ways depending on the type of $S$ as in \Cref{def:additivemultiplicativeelliptic}.
    \begin{description}
        \item[Additive] 
        In this case, $H_1(D_{\red},\BZ)$ is trivial and the root variables are
            \begin{equation} \label{eq:rootvaradd}
                a_i =  \chi(\alpha_i).
            \end{equation} 
            For a family of surfaces with $\dim |-K_S|=0$, usually the choice of $\omega$ is normalised such that $\chi(-K_S)=1$.
        \item[Multiplicative]
        In this case, $H_1(D_{\red},\BZ)$ has a generator, say $\gamma$, determining an orientation of $D_{\red}$.
        The period map gives $\chi : Q^{\perp}\rightarrow \BC \mod \hat{\chi}(\BZ \gamma)$, and usually one normalises $\hat{\chi}(\gamma) = 2 \pi i $. So, the root variable parameters are written as
            \begin{equation} \label{eq:rootvarmult}
                a_i =  e^{\chi(\alpha_i)}.
            \end{equation}
        One often denotes $q=e^{\chi(-K_S)}$, which becomes the shift parameter for $q$-difference Painlev\'e equations derived from a Sakai surface of multiplicative type.
        \item[Elliptic]    
            In this case, $D_{\red}$ is an elliptic curve so $H_1(D_{\red},\BZ)$ has rank two.
            Fixing a basis $(\gamma_0,\gamma_1)$ of $H_1(D_{\red},\BZ)$ and setting $\hat{\chi}(\gamma_0)=1$, $\hat{\chi}(\gamma_1) = \tau$ such that $\operatorname{Im}(\tau)>0$,
            the period map gives $\chi : Q^{\perp} \rightarrow \BC \mod \BZ+\BZ\tau$.
            Then, $\chi$ takes values in a torus, 
            and the parameters controlling the locations of points to be blown-up to obtain $S$ from $\BP^2$ or $\BP^1 \times \BP^1$ are related to $\chi(\alpha_i)$ through elliptic functions parametrising $D_{\red}$. 
            We will illustrate this in \Cref{ex:rootvarsweierstrass}.
    \end{description}
\begin{remark}
    In the literature, the term `root variables' sometimes refers to the parameters $a_i\in \BC$ and sometimes to the values $\chi(\alpha_i)$  of the period map as elements of $\BC \mod \hat{\chi}(H_1(D_{\red},\BZ))$.
    For the sake of precision, in these notes we call $a_i\in \BC$ \textit{root variable parameters}, and reserve the term root variables for $\chi(\alpha_i)$.
    For examples of explicit calculation of the period map leading to parameters in the form of \Cref{eq:rootvaradd,eq:rootvarmult}, see \cite[Sec. 5]{Sakai2001} or, e.g. \cite[Sec. 2.2]{DzhamayTakenawa} for the additive case and \cite[Sec. 3.1.3]{DzhamayKnizel} for the multiplicative case. We will illustrate the elliptic case in \Cref{ex:rootvarsweierstrass}, following \cite[Sec. 5]{Sakai2001}.
    For the elliptic case, we have used the Weierstrass parametrisation as is done in \cite{Sakai2001,MurataSakaiYoneda}, but it is possible to use other elliptic functions, see \cite{YamadaElliptic,KNY,CarsteaDzhamayTakenawa2017}.
\end{remark}
\begin{example} \label{ex:rootvarsweierstrass}
    Let $S$ be a Sakai surface of elliptic type $\mathcal{R}=A_0^{(1)}$ and assume without loss of generality\footnote{
    From \Cref{prop:rankPic10}, there always exists a morphism $\pi : S \rightarrow \BP^2$. Then $D$ will be the strict transform of a cubic curve passing through $b_1,\dots,b_9$, which can always be put into Weierstrass form via a change of coordinate.} that $S=\Bl_{b_9}\cdots\Bl_{b_1} \BP^2$, where $b_1,\dots,b_9$ are points\footnote{ Some points could be infinitely near, but in such case they must lie on the strict transform of the cubic in order to give $S$ of elliptic type.} on the Weierstrass cubic curve given, for some $g_2,g_3\in\BC$, by
    \begin{equation} \label{eq:weierstrasscubicexample}
         V(4 x_1^3 - x_0x_2^2-g_2 x_0^2x_1 - g_3 x_0^3 )\subset \BP^2,
    \end{equation}
    not lying simultaneously on any other cubic. 
    The curve \eqref{eq:weierstrasscubicexample} can be parametrised by the Weierstrass $\wp$ function according to $[x_0:x_1:x_2]=\left[1:\wp(z;g_2,g_3):\wp'(z;g_2,g_3)\right]$, so the nine points $b_1,\dots,b_9$ can be written as
    \begin{equation} \label{eq:pointlocationselliptic}
        b_i : [x_0:x_1:x_2]=\left[1:\wp(\theta_i;g_2,g_3):\wp'(\theta_i;g_2,g_3)\right]\!,
    \end{equation}
    for some $\theta_1,\dots,\theta_9\in\BC/\BZ+\BZ \tau$.

    The symmetry root lattice $Q^{\perp}$ is of type $E_8^{(1)}$, and we can take the basis of simple roots $\alpha_0,\dots,\alpha_8$ given by \Cref{eq:E8simplerootbasis} in the proof of \Cref{prop:E8latticeinPic}.
    For the simple roots of the form $\alpha_i=E_i-E_{i+1}$, $i=1,\dots,7$, we have
    \begin{equation}
        \chi(\alpha_i)=\theta_i-\theta_{i+1} \mod \BZ+\BZ\tau,
    \end{equation}
     where $\BZ+\BZ\tau$ is the period lattice of $\wp(z;g_2,g_3)$.
     Let us explain this.
    The two-form $\omega$ on $S$ giving $D=-\divi \omega$ can be chosen to be that given in coordinates $x=\frac{x_1}{x_0}$, $y=\frac{x_2}{x_0}$ by
    \begin{equation*}
        \frac{dx \wedge dy}{y^2-4 x^3+g_2 x + g_3}.
    \end{equation*}
    The values of $\hat{\chi}$ on $Q^{\perp}$ are computed by integrating the holomorphic 1-form on $D$, given by 
    \begin{equation}
        \operatorname{Res}_D \omega = \frac{1}{2\pi i}\frac{dx}{y}.
    \end{equation}        
    Via the parametrisation of the curve \eqref{eq:weierstrasscubicexample}, the holomorphic 1-form pulls back under the isomorphism $\BC/\BZ+\BZ\tau \simeq D$ to $\frac{1}{2 \pi i}dz$. 
    Further, $E_i\cap D$ and $E_{i+1}\cap D$ correspond to $z=\theta_i$ and $z=\theta_{i+1}$ respectively.
    Then, formula \eqref{residueintegral} gives
    \begin{align*}
    \chi(\alpha_i)&=\chi(E_i-E_{i+1}) =2 \pi i \int_{D \cap E_{i+1}}^{D \cap E_i} \operatorname{Res}_{D} \omega \\
    &= \int_{\theta_{i+1}}^{\theta_i} d z \mod \BZ+\BZ\tau 
    = \theta_i-\theta_{i+1}\mod \BZ+\BZ\tau.  
     \end{align*}
    Root variables corresponding to roots not of the form $E_i-E_j$ can be calculated similarly. 
    In particular, for the only simple root from the basis in \Cref{eq:E8simplerootbasis} which is not of the form $E_i-E_j$, namely $\alpha_8=H-E_1-E_2-E_3$, we have
    $\chi(\alpha_8)=\theta_1+\theta_2+\theta_3 \mod \BZ+\BZ \tau$.
    Recalling the expressions \eqref{eq:pointlocationselliptic} for the locations of the points $b_i$ in terms of $\theta_j$'s, we note that these are related to $\chi(\alpha_i)$ through the Weierstrass $\wp$ function.
\end{example}

\subsubsection{Families of Sakai surfaces and Cremona action} \label{subsec:familiesofSakaisurfacesandCremonaaction}

In order to formally define a discrete Painlev\'e equation, we consider the root variables as parameters for a \emph{family} of Sakai surfaces of surface type $\mathcal{R}$. 
This family will be parametrised by the space  of values of root variables $\mathscr{A}$. This is  a complex manifold given as a subset $\mathscr{A}\subset \left(\BC\mod \hat{\chi}(H_1(D_{\red},\BZ))\right)^{\times n}$, where $n$ and $\hat{\chi}(H_1(D_{\red},\BZ))$ are determined by the surface type $\mathcal{R}$, see \Cref{def:rootvariables}.

In \cite{Sakai2001}, the construction of discrete Painlev\'e equations from families of Sakai surfaces was described as taking a Sakai surface $S$ of surface type $\mathcal{R}$, then constructing an action of $\Cr(S)$ by isomorphisms on a family $\mathcal{S}$ of Sakai surfaces of the same type, i.e. isomorphisms between different surfaces in the family. 

One way to view this action as `realising' $\Cr(S)$ via pushforwards of isomorphisms, is to specify an identification of all of the Picard groups of surfaces in the family with that of $S$. 
This is often done, sometimes implicitly,   in the literature, and we will take this opportunity to present one way of spelling out the details.

\begin{definition} \label{def:familyofSakaisurfaces}
    For a given surface type $\mathcal{R}$, a family of Sakai surfaces parametrised by root variables in $\mathscr{A}$   is a family 
    \begin{equation}
        \mathcal{S} \rightarrow \mathscr{A},
    \end{equation}
    with fibre $S_{\boldsymbol{a}}$ over $\boldsymbol{a} \in \mathscr{A}$ being a Sakai surface of type $\mathcal{R}$, with additional data of 
    \begin{itemize}
        \item morphisms $\varepsilon_{\boldsymbol{a}}:S_{\boldsymbol{a}} \rightarrow \BP^2$, for each $\boldsymbol{a}\in\mathscr{A}$, varying holomorphically with respect to $\boldsymbol{a}$;
        \item the geometric basis $\mathcal{E}_{\boldsymbol{a}}$ for $\Pic(S_{\boldsymbol{a}})$ associated with $\varepsilon_{\boldsymbol{a}}$, which we write as
        \begin{equation*}
             \Pic(S_{\boldsymbol{a}}) = \langle H^{(\boldsymbol{a})}, E^{(\boldsymbol{a})}_1,\ldots,E^{(\boldsymbol{a})}_9 \rangle_{\BZ};
        \end{equation*} 
        \item an identification, provided by the geometric bases coming from the $\varepsilon_{\boldsymbol{a}}$, of all the Picard groups $\Pic(S_{\boldsymbol{a}})$ for $\boldsymbol{a}\in \mathscr{A}$, into the single lattice
        \begin{equation*}
            \Pic_{\mathcal{S}}:= \langle H, E_1,\ldots,E_9 \rangle_{\BZ},
        \end{equation*}
        which is isomorphic to the Lorentzian lattice $\Lambda_{10}$, see \Cref{prop:E8latticeinPic}.
        This identification is via the isomorphisms
        \begin{equation} 
            \begin{tikzcd}[row sep=tiny]
                 \Pic_{\mathcal{S}} \arrow[r,"\iota_{\boldsymbol{a}} "]&\Pic(S_{\boldsymbol{a}})\\ 
                  H  \arrow[r,mapsto]     &H^{(\boldsymbol{a})} \\
                  E_i \arrow[r,mapsto]   &E_i^{(\boldsymbol{a})}.
            \end{tikzcd} 
        \end{equation}
        Note that $\Pic_{\mathcal{S}}$ is equipped with a symmetric bilinear form inherited from the intersection form on $\Pic(S_{\boldsymbol{a}})$ via $\iota_{\boldsymbol{a}}$, which we denote by the same symbol;
        \item elements $D_i\in \Pic_{\mathcal{S}}$ such that $D_i^{(\boldsymbol{a})} = \iota_{\boldsymbol{a}}(D_i)$, with $D^{(\boldsymbol{a})}=\sum_i m_i D^{(\boldsymbol{a})}_i\in|-K_{S_{\boldsymbol{a}}}|$ being the unique effective anti-canonical divisor of $S_{\boldsymbol{a}}$;
        \item a basis $\Pi=\Set{\alpha_0,\dots,\alpha_n} \subset \Pic_{\mathcal{S}}$ of simple roots for $$Q^{\perp} = \Set{\! F \in \Pic_{\mathcal{S}} ~|~ F.D_i=0 \text{ for all } i\!}\!;$$
        \item rational two-forms $\omega_{\boldsymbol{a}}$ on $S_{\boldsymbol{a}}$ such that $-\divi \omega_{\boldsymbol{a}} = D^{(\boldsymbol{a})}\in|-K_{S_{\boldsymbol{a}}}|$, varying holomorphically  with respect to $\boldsymbol{a}$, defining period maps $\chi_{\boldsymbol{a}}$ on $\iota_{\boldsymbol{a}}(Q^{\perp})\subset \Pic(S_{\boldsymbol{a}})$;
    \end{itemize}
    such that the root variables of $S_{\boldsymbol{a}}$ for the basis $\Set{\!\iota_{\boldsymbol{a}}(\alpha_0),\dots,\iota_{\boldsymbol{a}}(\alpha_n)\!}$ of $\iota_{\boldsymbol{a}}(Q^{\perp})$ with respect to the period map $\chi_{\boldsymbol{a}}$ are $\boldsymbol{a} \in \mathscr{A}$.
\end{definition}

\begin{remark} \label{rem:extraparameter}
    For surface types $\mathcal{R}$ associated with differential Painlev\'e equations, there is an \emph{extra parameter} which plays the role of the independent variable in the differential Painlev\'e equation, see \Cref{rem:p1surface}.  
    In such cases we consider this among the parameters in $\mathscr{A}$, in addition to the values of the period map on simple roots.
\end{remark}

The following is a rephrasing of collected results of \cite{Sakai2001}, with details spelt out in line with \Cref{def:familyofSakaisurfaces}.
\begin{theorem}[\cite{Sakai2001}]  \label{thm:Cremonaaction}
    For each surface type $\mathcal{R}$, there is a family $\mathcal{S}\rightarrow \mathscr{A}$ of Sakai surfaces of   type $\mathcal{R}$, together with an action of $W(\mathcal{R}^{\perp})\rtimes\Aut(\mathcal{R}^{\perp})$, such that an element $w\in W(\mathcal{R}^{\perp})\rtimes\Aut(\mathcal{R}^{\perp})$ acts as follows
    \begin{itemize}
        \item on $\Pic_{\mathcal{S}}$ by $w: \Pic_{\mathcal{S}} \rightarrow \Pic_{\mathcal{S}}$ defined in the same way 
        \footnote{For cases when $\mathcal{R}^{\perp}$ has standard root lengths, simple reflections $r_i$ act according to the reflection formula $r_i(\lambda) = \lambda + (\lambda.\alpha_i)\alpha_i$, where $\lambda\in\Pic_{\mathcal{S}}$,  and $\Aut(\mathcal{R}^{\perp})$ acts as specified on a case-by-case basis in \cite{Sakai2001}.}
        as in \Cref{thm:cremonaisometries},
        \item on $\mathcal{S}$ by isomorphisms $\tilde{\varphi}_{\boldsymbol{a}}^{(w)} : S_{\boldsymbol{a}}\rightarrow S_{\bar{\boldsymbol{a}}}$, such that $\left(\tilde{\varphi}_{\boldsymbol{a}}^{(w)}\right)^* \omega_{\bar{\boldsymbol{a}}}=\omega_{\boldsymbol{a}}$, inducing  $w$ by pullback, i.e. for every $\boldsymbol{a}\in\mathscr{A}$, the following diagram commutes:
        \[
        \begin{tikzcd}
            \Pic_{\mathcal{S}} \arrow[d, "\iota_{\boldsymbol{a}}"']  & \Pic_{\mathcal{S}} \arrow[d, "\iota_{\bar{\boldsymbol{a}}}"] \arrow[l, "w"'] \\
            \Pic(S_{\boldsymbol{a}}) &  \Pic(S_{\bar{\boldsymbol{a}}}), \arrow[l,swap, "\left(\tilde{\varphi}_{\boldsymbol{a}}^{(w)}\right)^* "']
        \end{tikzcd}
        \]
        \item on $\mathscr{A}$ by transformations $\boldsymbol{a}\mapsto \bar{\boldsymbol{a}}$ given by
    \begin{equation} \label{eq:rootvarevol}
        \chi_{\bar{\boldsymbol{a}}}\left( \iota_{\bar{\boldsymbol{a}}}(\alpha_i)\right) = \chi_{\boldsymbol{a}} \left( \iota_{\boldsymbol{a}} \left(w(\alpha_i)\right)\right).
    \end{equation}
    \end{itemize} 
    This action is called the \emph{Cremona action} of $W(\mathcal{R}^{\perp})\rtimes\Aut(\mathcal{R}^{\perp})$.
\end{theorem}
The Cremona action realises the groups $\Cr(S_{\boldsymbol{a}})$, for $\boldsymbol{a}\in\mathscr{A}$, in the following way \cite[Sec. 6.2]{Sakai2001}.
We can regard $S_{\boldsymbol{a}}$ and $S_{\bar{\boldsymbol{a}}}$ as the same surface, but with root variables defined using $\Set{\iota_{\boldsymbol{a}}(\alpha_0),\dots,\iota_{\boldsymbol{a}}(\alpha_n)}$ and $\Set{\iota_{\boldsymbol{a}}\circ w(\alpha_0),\dots,\iota_{\boldsymbol{a}}\circ w(\alpha_n)}$, respectively. These are $\Pi$ written in the geometric bases $\mathcal{E}_{\boldsymbol{a}}$ and $w(\mathcal{E}_{\boldsymbol{a}})$.
Then, we have a change of blowing-down structure from $\mathcal{E}_{\boldsymbol{a}}$ to $w(\mathcal{E}_{\boldsymbol{a}})$, given by $\varphi_{\boldsymbol{a}}^{(w)} = \varepsilon_{\bar{\boldsymbol{a}}}\circ\tilde{\varphi}_{\boldsymbol{a}}^{(w)}\circ \varepsilon_{\boldsymbol{a}}^{-1} \in \Bir (\BP^2)$:
\begin{equation} \label{eq:changeofbdsdiagram}
\begin{tikzcd}
        & S_{\boldsymbol{a}} \arrow[dl, swap, "\varepsilon_{\boldsymbol{a}}"] \arrow[dr, "\varepsilon_{\bar{\boldsymbol{a}}} \circ \tilde{\varphi}_{\boldsymbol{a}}^{(w)}"] &  \\
        \BP^2 \arrow[rr, dashed, "\varphi_{\boldsymbol{a}}^{(w)}"] & & \BP^2. 
    \end{tikzcd}
\end{equation}

\begin{remark}
    The parameter evolution in \Cref{eq:rootvarevol} is dictated by $w$, which can be seen either in terms of $\bar{\boldsymbol{a}}$ being root variables for $\Pi$ written in the basis $w(\mathcal{E}_{\boldsymbol{a}})$, or at the level of the isomorphisms $S_{\boldsymbol{a}}\rightarrow S_{\bar{\boldsymbol{a}}}$ as follows, see \cite[Rem. 3.1]{DzhamayTakenawa}.
    Suppose $\psi : S\rightarrow S'$ is an isomorphism between two generalised Halphen surfaces $S,S'$ with rational 2-forms $\omega, \omega'$ such that $-\divi \omega = D\in |-K_S|$ and $-\divi \omega' = D'\in|-K_{S'}|$, and period maps $\chi_{S}$, $\chi_{S'}$ defined with respect to $(D,\omega)$ and $(D',\omega')$ respectively.
    Then, if $\psi^*\omega'=\omega$, the definition of the period map guarantees that when $\alpha\in Q^{\perp}\subset \Pic(S)$ is such that $\psi_*(\alpha) =\alpha'$, we have $\chi_{S'}(\alpha') = \chi_S(\alpha)$.
    By applying this to the situation in \Cref{thm:Cremonaaction}, we have for $\alpha_i\in Q^{\perp}\subset \Pic_{\mathcal{S}}$ that
    \begin{equation}
    \begin{aligned}        
         \chi_{\bar{\boldsymbol{a}}} \left( \iota_{\bar{\boldsymbol{a}}}(\alpha_i)\right) 
        &= \chi_{\bar{\boldsymbol{a}}} \left( (\iota_{\bar{\boldsymbol{a}}}\circ w^{-1}) \left(w(\alpha_i)\right)\right) \\
        &= \chi_{\bar{\boldsymbol{a}}} \left( \left( (\tilde{\varphi}_{\boldsymbol{a}}^{(w)})_*\circ \iota_{\boldsymbol{a}}\right) \left(w(\alpha_i)\right)\right)
        = \chi_{\boldsymbol{a}} \left( \iota_{\boldsymbol{a}} \left(w(\alpha_i)\right)\right)\!.
    \end{aligned}
    \end{equation}
\end{remark}

\begin{prop} \label{prop:rootvarevolution}
        With the notation of \Cref{def:familyofSakaisurfaces} and \Cref{thm:Cremonaaction}, let $M\in \GL(n+1,\BZ)$ be the matrix representing the restriction of $w$ to $Q^{\perp}$ with respect to the basis $\Set{\alpha_0,\dots,\alpha_n}$, so $w(\alpha_i)=\sum_{j=0}^n M_{ij} \alpha_j$.  Then, the evolution of root variable parameters, in the additive and multiplicative cases, can be written as 
\begin{align*}
    &\bar{a}_i  = \chi_{\bar{\boldsymbol{a}}} \left( \iota_{\bar{\boldsymbol{a}}}(\alpha_i)\right) = \chi_{\boldsymbol{a}} ( \iota_{{\boldsymbol{a}}} (\textstyle\sum_{j=0}^{n} M_{ij} \alpha_j) ) = \sum_{j=0}^n M_{ij}a_j &&(\text{additive}),  \\
    &\bar{a}_i  = e^{\chi_{\bar{\boldsymbol{a}}} \left( \iota_{\bar{\boldsymbol{a}}}(\alpha_i)\right)} = e^{\chi_{\boldsymbol{a}} ( \iota_{{\boldsymbol{a}}} (\sum_{j=0}^n M_{ij} \alpha_j) )} = \prod_{j=0}^n a_j^{M_{ij}} &&(\text{multiplicative}).  
\end{align*}
In the elliptic case, the evolution is described in terms of the group law on the elliptic curve, or of the  addition on a torus, see \cite{CarsteaDzhamayTakenawa2017,EguchiTakenawa,KMMOY}.
This covariant correspondence between the actions on $\Pic_{\mathcal{S}}$ and on the root variables is the reason why we have required $w$ to be induced by pullback rather than pushforward in \Cref{thm:Cremonaaction}. However, the opposite convention is sometimes used. 

\end{prop}

We say a family $\mathcal{S}\rightarrow \mathscr{A}$, as in \Cref{def:familyofSakaisurfaces}, is \textit{universal} if it exhausts all isomorphism classes, i.e. any Sakai surface of surface type $\mathcal{R}$ is isomorphic to $S_{\boldsymbol{a}}$ for some $\boldsymbol{a}\in\mathscr{A}$.
Note that by the Torelli-type Theorem \cite[Th. 25]{Sakai2001}, $S_{\boldsymbol{a}}$ and $S_{\boldsymbol{a}'}$ from a universal family are isomorphic if and only if $\boldsymbol{a},\boldsymbol{a}'\in \mathscr{A}$ are related by the action of $W(R^{\perp})\rtimes\operatorname{Aut}(R^{\perp})$ as in \Cref{thm:Cremonaaction}, and in this case the isomorphism is unique. 
This is one interpretation of why root variables are sometimes referred to in the discrete Painlev\'e literature as being `gauge-invariant' parameters \cite{DzhamayTakenawa}.

\begin{remark}
    It is important to note that the families constructed in \cite{Sakai2001} and referred to in \Cref{thm:Cremonaaction} admit the maximal possible symmetry group for their surface type, in the sense of realising all Cremona isometries.
    In various contexts, there appear families with constrained root variable parameters.
    Such families have restricted symmetry groups, i.e. they do not admit the action of the whole of $W(\mathcal{R}^{\perp})\rtimes\Aut(\mathcal{R}^{\perp})$. 
\end{remark}

\subsubsection{Discrete Painlev\'e equations}

We are now ready to define discrete Painlev\'e equations, using the setup in \Cref{def:familyofSakaisurfaces} and \Cref{thm:Cremonaaction}.

First note that, from \Cref{subsubsec:affinerootsystems}, we have a subgroup of translations 
\[
T_{\nought{P}(\mathcal{R}^{\perp})} \subset W(\mathcal{R}^{\perp}) \rtimes \Aut(\mathcal{R}^{\perp}),
\]
where $\nought{P}(\mathcal{R}^{\perp})$ is the weight lattice of the underlying finite root system of the symmetry type $\mathcal{R}^{\perp}$.
Discrete Painlev\'e equations were initially said in \cite{Sakai2001} to arise from the Cremona action of translations, but it is now common to define discrete Painlev\'e equations as corresponding to elements of infinite order, not just translations. 

\begin{definition}[Discrete Painlev\'e equation]
    Given a family $\mathcal{S}\rightarrow \mathscr{A}$ of Sakai surfaces of type $\mathcal{R}$ as in \Cref{def:familyofSakaisurfaces} with Cremona action of $W(\mathcal{R}^{\perp})\rtimes \operatorname{Aut}(\mathcal{R}^{\perp})$ as in \Cref{thm:Cremonaaction}, a \emph{discrete Painlev\'e equation} is a family of pairs $(\varphi_{\boldsymbol{a}}, \boldsymbol{a}\mapsto \bar{\boldsymbol{a}})$ parametrised by $\mathscr{A}$, coming from the Cremona action of an element $w$ of \emph{infinite order}, where:
    \begin{itemize}
        \item the association $\boldsymbol{a}\mapsto \bar{\boldsymbol{a}}$ is the action of $w$ on $\mathscr{A}$,
        \item the birational transformation $\varphi_{\boldsymbol{a}} \in \Bir(\BP^2)$ is as in the diagram  \eqref{eq:changeofbdsdiagram}, omitting the element $w$ from the notation.
    \end{itemize} 
 \end{definition}
    
\begin{remark}
    Since $T_{\nought{P}(\mathcal{R}^{\perp})}$ is a finite index subgroup of $W(\mathcal{R}^{\perp})\rtimes \Aut(\mathcal{R}^{\perp})$, any element of infinite order must become a translation after some finite number of iterations.  Elements which are of infinite order but which are not translations are sometimes referred to as \emph{quasi-translations} \cite{yangtranslations}.
    For all types of Sakai surfaces which give rise to discrete Painlev\'e equations, there are infinitely many non-conjugate elements of infinite order. This means that there are infinitely many inequivalent discrete Painlev\'e equations.
    This is implicit in the original definition from Sakai's paper, and was also noticed at the level of affine Weyl groups in \cite{dPsinfinite}.
\end{remark}
 
\begin{remark} \label{rem:translations}
We give some remarks related to the action of translation elements of $W(\mathcal{R}^{\perp})\rtimes \operatorname{Aut}(\mathcal{R}^{\perp})$ on $\Pic_{\mathcal{S}}$.
In the framework of \Cref{subsubsec:affinerootsystems}, the translations are associated with elements of the weight lattice $\nought{P}(\mathcal{R}^{\perp})$ of the underlying finite root system.
This is via the Kac translation formula \eqref{kactranslation} in \Cref{def:kactranslation}, which gives,  for $v \in\nought{P}(\mathcal{R}^{\perp})$, the action of $T_v$ on $\mathfrak{h}^*$.
The restriction of the Kac translation formula to the orthogonal complement of the null root is given by \Cref{translationorth}. 
In particular, this provides sthe action of $T_v$ on the root lattice.

However, it is important to note that, in general\footnote{For $\mathcal{R}^{\perp}$ with standard root lengths, it is still possible to use the formula for the action on $\Pic_{\mathcal{S}}$ of translations associated to elements of the root lattice $\nought{Q}(\mathcal{R}^{\perp})\subset \nought{P}(\mathcal{R}^{\perp})$, which can be expressed in terms of only simple reflections, without the need for Dynkin diagram automorphisms.}, the action of a translation on the whole of $\Pic_{\mathcal{S}}$ will not be given by the Kac translation formula.
Instead, it is computed by first writing the translation as a composition of generators of $W(\mathcal{R}^{\perp})\rtimes\Aut(\mathcal{R}^{\perp})$, and then using their actions on $\Pic_{\mathcal{S}}$ as in \Cref{thm:Cremonaaction}. 

Nevertheless, we still regard these elements as translations associated to elements of $\nought{P}(\mathcal{R}^{\perp})$ according to their action on $Q^{\perp}$. Explicitly, we have
\begin{equation} \label{eq:kactransQ}
    T_v(\lambda) = \lambda - (\lambda.v)\delta, \qquad \lambda \in Q^{\perp}, ~v\in \nought{P}(\mathcal{R}^{\perp}),
\end{equation}
where we have changed $v$ to $-v$ from the assignment of $T_v$ to $v$ in \Cref{def:kactranslation} in order to neaten some expressions when we illustrate translations explicitly in \Cref{sec:KNY}.
\end{remark}

\subsection{Example: discrete Painlev\'e equations on $D_5^{(1)}$ Sakai surfaces} 
\label{sec:KNY}

We now give an explicit example of a family of Sakai surfaces parametrised by root variables, as well as some associated discrete Painlev\'e equations.  
Throughout \Cref{subsec:familiesofSakaisurfacesandCremonaaction} we made definitions using $\BP^2$, but all Sakai surfaces of type from which discrete Painlev\'e equations can be constructed (i.e. $\mathcal{R} \neq A_8^{(1)},D_8^{(1)},E_8^{(1)}$) admit $\BP^1\times\BP^1$ as a minimal model, see \Cref{rem:P2vsP1P1compactification}. 
The framework of \Cref{subsec:familiesofSakaisurfacesandCremonaaction} can be translated to the alternative choice of $\BP^1\times\BP^1$, and we will implicitly do this in this section.

\subsubsection{Family of surfaces}

We construct a family of Sakai surfaces of type $\mathcal{R}=D_5^{(1)}$, $\mathcal{R}^{\perp}=A_3^{(1)}$, parametrised by root variables, following \cite[Section 8.2.18]{KNY}.
Begin with $\BP^1\times\BP^1$ with the atlas as in \Cref{rmk:coorchart} with $q,p$ in place of $x,y$, and perform eight blow-ups of points $b_1,\dots, b_8$ as given in \Cref{fig:KNYblowupdata}, depending on parameters in 
$$\mathscr{A} = \Set{\! \boldsymbol{a} = (a_0,a_1,a_2,a_3;t)\in\BC^4\times \mathcal{T} ~|~ a_{0} + a_{1} + a_{2} + a_{3} = 1 \!}\!,$$ 
where we have included the `extra parameter' $t\in\mathcal{T}=\BC\setminus\{0\}$ in the root variable space $\mathscr{A}$, see \Cref{rem:extraparameter}. This plays the role of the independent variable for $\pain{V}$, which also corresponds to this surface type.
For $\boldsymbol{a}\in\mathscr{A}$, we denote the resulting surface $S_{\boldsymbol{a}}$, with morphism $\varepsilon_{\boldsymbol{a}}:S_{\boldsymbol{a}}\rightarrow\BP^1\times\BP^1$. 
We give a representation of $S_{\boldsymbol{a}}$ in Figure \ref{fig:KNY-soic-5}.

\begin{table}[htb]
    \centering
    {\small 
    \begin{tabular}{c|c}
        $b_i$ & $\pi_{b_i}$ \\\hline 
        $\mathcal{U}_{1,0} \ni b_1:(Q,p)=(0,-t)$\hspace{0.5cm} & $\begin{tikzcd}[row sep =tiny]
            \CW_1^{(0)} \ni (U_1,V_1)\arrow[r,mapsto]& (V_1,-t+ U_1 V_1)\in\mathcal{U}_{1,0}\\[-1em]
            \CW_1^{(1)} \ni(u_1 , v_1)\arrow[r,mapsto]& (u_1 v_1,-t+ v_1)\in\mathcal{U}_{1,0}
        \end{tikzcd}$ \\\hline 
        $E_1 \ni b_2:(U_1,V_1)=(-a_0,0)$\hspace{0.5cm}  & $\begin{tikzcd}[row sep =tiny]
            \CW_2^{(0)} \ni (U_2,V_2)\arrow[r,mapsto]& (-a_0+V_2, U_2 V_2)\in\CW_1^{(0)}\\[-1em]
            \CW_2^{(1)} \ni(u_2 , v_2)\arrow[r,mapsto]& (-a_0+u_2 v_2, v_2)\in\CW_1^{(0)}
        \end{tikzcd}$ \\\hline 
        $\mathcal{U}_{1,0} \ni b_3:(Q,p)=(0,0)$\hspace{0.5cm}  & $\begin{tikzcd}[row sep =tiny]
            \CW_3^{(0)} \ni (U_3,V_3)\arrow[r,mapsto]& (V_3, U_3 V_3)\in\mathcal{U}_{1,0}\\[-1em]
            \CW_3^{(1)} \ni(u_3 , v_3)\arrow[r,mapsto]& (u_3 v_3, v_3)\in\mathcal{U}_{1,0}
        \end{tikzcd}$ \\\hline 
        $E_3 \ni b_4:(U_3,V_3)=(-a_2,0)$\hspace{0.5cm}  & $\begin{tikzcd}[row sep =tiny]
            \CW_4^{(0)} \ni (U_4,V_4)\arrow[r,mapsto]& (-a_2+V_4, U_4 V_4)\in\CW_3^{(0)}\\[-1em]
            \CW_4^{(1)} \ni(u_4 , v_4)\arrow[r,mapsto]& (-a_2+u_4 v_4, v_4)\in\CW_3^{(0)}
        \end{tikzcd}$ \\\hline 
        $\mathcal{U}_{0,1} \ni b_5:(q,P)=(0,0) $\hspace{0.5cm}  & $\begin{tikzcd}[row sep =tiny]
            \CW_5^{(0)} \ni (U_5,V_5)\arrow[r,mapsto]& (V_5,U_5 V_5)\in\mathcal{U}_{0,1}\\[-1em]
            \CW_5^{(1)} \ni(u_5 , v_5)\arrow[r,mapsto]& (u_5 v_5, v_5)\in\mathcal{U}_{0,1}
        \end{tikzcd}$ \\\hline 
        $E_5 \ni b_6:(u_5,v_5)=(a_1,0)$\hspace{0.5cm}  & $\begin{tikzcd}[row sep =tiny]
            \CW_6^{(0)} \ni (U_6,V_6)\arrow[r,mapsto]& (a_1+V_6, U_6 V_6)\in\CW_5^{(1)}\\[-1em]
            \CW_6^{(1)} \ni(u_6 , v_6)\arrow[r,mapsto]& (a_1+u_6 v_6, v_6)\in\CW_5^{(1)}
        \end{tikzcd}$ \\\hline 
        $\mathcal{U}_{0,1} \ni b_7:(q,P)=(1,0)$\hspace{0.5cm}  & $\begin{tikzcd}[row sep =tiny]
            \CW_7^{(0)} \ni (U_7,V_7)\arrow[r,mapsto]& (1+V_7,U_7 V_7)\in\mathcal{U}_{0,1}\\[-1em]
            \CW_7^{(1)} \ni(u_7 , v_7)\arrow[r,mapsto]& (1+u_7 v_7, v_7)\in\mathcal{U}_{0,1}
        \end{tikzcd}$ \\\hline 
        $E_7 \ni b_8:(u_7,v_7)=(a_3,0)$\hspace{0.5cm}  & $\begin{tikzcd}[row sep =tiny]
            \CW_8^{(0)} \ni (U_8,V_8)\arrow[r,mapsto]& (a_3+V_8, U_8 V_8)\in\CW_7^{(1)}\\[-1em]
            \CW_8^{(1)} \ni(u_8 , v_8)\arrow[r,mapsto]& (a_3+u_8 v_8, v_8)\in\CW_7^{(1)}
        \end{tikzcd}$  
    \end{tabular}
    }
    \caption{Blow-up data for the family of $D_5^{(1)}$ surfaces.}
    \label{fig:KNYblowupdata}
\end{table}

\begin{figure}[ht]
	\scalebox{0.8}{\begin{tikzpicture}[>=stealth,basept/.style={circle, draw=red!100, fill=red!100, thick, inner sep=0pt,minimum size=1.2mm}]
	\begin{scope}[xshift=0cm,yshift=0cm]
	\draw [black, line width = 1pt] (-0.4,0) -- (2.9,0)	node [pos=0,left] {
 \small $p=0$ 
 } node [pos=1,right] {\small $p=0$};
	\draw [black, line width = 1pt] (-0.4,2.5) -- (2.9,2.5) node [pos=0,left] {
 \small $P=0$ 
 } node [pos=1,right] {\small $P=0$};
	\draw [black, line width = 1pt] (0,-0.4) -- (0,2.9) node [pos=0,below] {
  \small $q=0$ 
 } node [pos=1,above] {\small $q=0$};
	\draw [black, line width = 1pt] (2.5,-0.4) -- (2.5,2.9) node [pos=0,below] {
  \small $Q=0$ 
 } node [pos=1,above] {\small $Q=0$};
	\node (p3) at (2.5,0) [basept,label={[xshift = -8pt, yshift=-15pt] \small $b_{3}$}] {};
	\node (p4) at (3,0.5) [basept,label={[yshift=-2pt] \small $b_{4}$}] {};
	\node (p1) at (2.5,1) [basept,label={[xshift = -8pt, yshift=-15pt] \small $b_{1}$}] {};
	\node (p2) at (3,1.5) [basept,label={[yshift=-2pt] \small $b_{2}$}] {};
	\node (p5) at (0,2.5) [basept,label={[xshift = 8pt, yshift=-2pt] \small $b_{5}$}] {};
	\node (p6) at (-.5,2) [basept,label={[yshift=-17pt] \small $b_{6}$}] {};
	\node (p7) at (1.5,2.5) [basept,label={[xshift = 8pt, yshift=-2pt] \small $b_{7}$}] {};
	\node (p8) at (1,2) [basept,label={[yshift=-17pt] \small $b_{8}$}] {};
	\draw [red, line width = 0.8pt, ->] (p2) -- (p1);
	\draw [red, line width = 0.8pt, ->] (p4) -- (p3);
	\draw [red, line width = 0.8pt, ->] (p6) -- (p5);
	\draw [red, line width = 0.8pt, ->] (p8) -- (p7);	
	\end{scope}
	\draw [->] (5.5,1)--(3.5,1) node[pos=0.5, below] {$\pi_{b_1}\circ \cdots \circ \pi_{b_{8}}$};
	\begin{scope}[xshift=7.5cm,yshift=0cm]
	\draw [red, line width = 1pt] (-0.4,0) -- (3.5,0)	node [pos=0, left] {\small $H_{p}-E_{3}$};
	\draw [red, line width = 1pt] (0,-0.4) -- (0,2.4) node [pos=0, below] {\small $H_{q}-E_{5}$};
	\draw [blue, line width = 1pt] (-0.2,1.8) -- (0.8,2.8) node [pos=0, left] {\small $E_{5}-E_{6}$};
	\draw [red, line width = 1pt] (-0.1,2.7) -- (0.4,2.2) node [pos=0, above] {\small $E_{6}$};
	\draw [blue, line width = 1pt] (1.2,1.8) -- (2.2,2.8) node [pos=0, xshift=-14pt, yshift=-5pt] {\small $E_{7}-E_{8}$};
	\draw [red, line width = 1pt] (1.6,2.4) -- (2.1,1.9) node [pos=1, below] {\small $E_{8}$}; 
	\draw [blue, line width = 1pt] (0.3,2.6) -- (4.2,2.6) node [pos=1,right] {\small $H_{p} - E_{5} - E_{7}$};
	\draw [blue, line width = 1pt] (3,-0.2) -- (4,0.8) node [pos=1,right] {\small $E_{3} - E_{4}$};
	\draw [red, line width = 1pt] (3.4,0.4) -- (3.9,-0.1) node [pos=1, below] {\small $E_{4}$};
	\draw [blue, line width = 1pt] (3.8,0.3) -- (3.8,3) node [pos=1, above] {\small $H_{q}-E_{1} - E_{3}$};	
	\draw [blue, line width = 1pt] (3,1) -- (4,2) node [pos=1,right] {\small $E_{1} - E_{2}$};
		\draw [red, line width = 1pt] (3.1,1.9) -- (3.6,1.4) node [pos=0, above] {\small $E_{2}$};
	\draw [red, line width = 1pt] (-0.4,1.2) -- (3.5,1.2)	node [pos=0, left] {\small $H_{p}-E_{1}$}; 
	\draw [red, line width = 1pt] (1.4,-0.4) -- (1.4,2.4) node [pos=0, below] {\small $H_{q}-E_{7}$};
	\end{scope}
	\end{tikzpicture}}
	\caption{Configuration of exceptional divisors for the family of $D_5^{(1)}$ surfaces.}
	\label{fig:KNY-soic-5}
\end{figure}
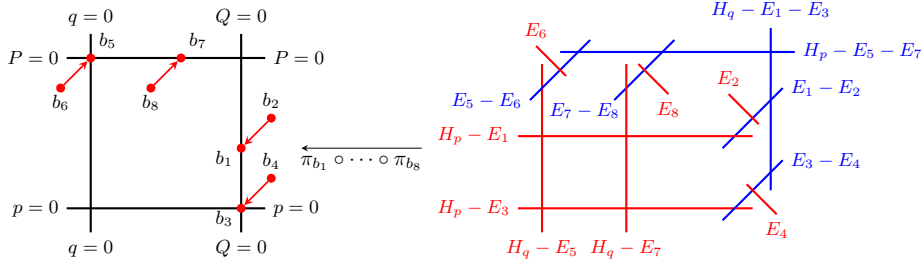

\subsubsection{Root data} \label{subsubsec:rootdata}

With the geometric bases coming from $\varepsilon_{\boldsymbol{a}}$, all the Picard groups $\Pic(S_{\boldsymbol{a}})$, for $\boldsymbol{a}\in\mathscr{A}$ are identified with the lattice 
\begin{equation}
    \Pic_{\mathcal{S}} = \langle H_q,H_p,E_1,\dots,E_8\rangle_{\BZ},
\end{equation}
which is isomorphic to the Lorentzian lattice as in \Cref{prop:E8latticeinPic} with basis $H_q=v_0-v_1$, $H_p=v_0-v_2$, $E_1=v_0-v_1-v_2$, $E_i=v_{i+1}$ for $i=2,\dots,8$.
We take surface and symmetry root bases as in \Cref{fig:d-roots-d5-KNY}	
and \Cref{fig:a-roots-a3-KNY} respectively, and denote the root lattice associated with the symmetry roots by 
\begin{equation} \label{eq:QperpD5example}
    Q^{\perp} = \langle \alpha_0, \alpha_1,\alpha_2,\alpha_3\rangle_{\BZ} \subset \Pic_{\mathcal{S}}.
\end{equation}
The parameters $\boldsymbol{a}\in\mathscr{A}$ are the root variables of $S_{\boldsymbol{a}}$ for the symmetry root basis in \Cref{fig:a-roots-a3-KNY}, using the rational 2-form given in coordinates by $dq\wedge dp$, together with the extra parameter $t$ as in \Cref{rem:extraparameter}.

\begin{figure}[htb]
\begin{equation*}\label{eq:d-roots-d51}			
	\raisebox{-22.1pt}{\scalebox{0.7}{\begin{tikzpicture}[
			elt/.style={circle,draw=black!100,thick, inner sep=0pt,minimum size=2mm}]
		\path 	(-1,1) 	node 	(d0) [elt, label={[xshift=-10pt, yshift = -10 pt] $D_{0}$} ] {}
		        (-1,-1) node 	(d1) [elt, label={[xshift=-10pt, yshift = -10 pt] $D_{1}$} ] {}
		        ( 0,0) 	node  	(d2) [elt, label={[xshift=-10pt, yshift = -10 pt] $D_{2}$} ] {}
		        ( 1,0) 	node  	(d3) [elt, label={[xshift=10pt, yshift = -10 pt] $D_{3}$} ] {}
		        ( 2,1) 	node  	(d4) [elt, label={[xshift=10pt, yshift = -10 pt] $D_{4}$} ] {}
		        ( 2,-1) node 	(d5) [elt, label={[xshift=10pt, yshift = -10 pt] $D_{5}$} ] {};
		\draw [black,line width=1pt ] (d0) -- (d2) -- (d1)  (d2) -- (d3) (d4) -- (d3) -- (d5);
	\end{tikzpicture}}} \qquad
			\begin{alignedat}{2}
			D_{0} &= E_{1} - E_{2}, &\qquad  D_{3} &= H_{p} - E_{5} - E_{7},\\
			D_{1} &= E_{3} - E_{4}, &\qquad  D_{4} &= E_{5} - E_{6},\\
			D_{2} &= H_{q} - E_{1} - E_{3}, &\qquad  D_{5} &= E_{7} - E_{8}.
			\end{alignedat}
\end{equation*}
	\caption{Surface root basis for the family of $D_5^{(1)}$ surfaces.}
	\label{fig:d-roots-d5-KNY}	
\end{figure}
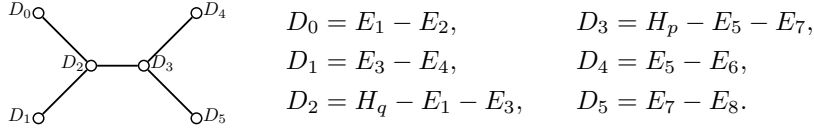

\begin{figure}[ht]
\begin{equation*}\label{eq:a-roots-a31}			
	\raisebox{-22.1pt}{\scalebox{0.7}{\begin{tikzpicture}[
			elt/.style={circle,draw=black!100,thick, inner sep=0pt,minimum size=2mm}]
		\path 	(-1,1) 	node 	(a0) [elt, label={[xshift=-10pt, yshift = -10 pt] $\alpha_{0}$} ] {}
		        (-1,-1) node 	(a1) [elt, label={[xshift=-10pt, yshift = -10 pt] $\alpha_{1}$} ] {}
		        ( 1,-1) node  	(a2) [elt, label={[xshift=10pt, yshift = -10 pt] $\alpha_{2}$} ] {}
		        ( 1,1) 	node 	(a3) [elt, label={[xshift=10pt, yshift = -10 pt] $\alpha_{3}$} ] {};
		\draw [black,line width=1pt ] (a0) -- (a1) -- (a2) --  (a3) -- (a0); 
	\end{tikzpicture}}} \qquad
			\begin{alignedat}{2}
			\alpha_{0} &= H_{p} - E_{1} - E_{2}, &\qquad  \alpha_{3} &= H_{q} - E_{7} - E_{8},\\
			\alpha_{1} &= H_{q} - E_{5} - E_{6}, &\qquad  \alpha_{2} &= H_{p} - E_{3} - E_{4}. 
			\end{alignedat}
\end{equation*}
	\caption{Symmetry root basis of type $A_3^{(1)}$ for the family of $D_5^{(1)}$ surfaces.}
	\label{fig:a-roots-a3-KNY}	
\end{figure}
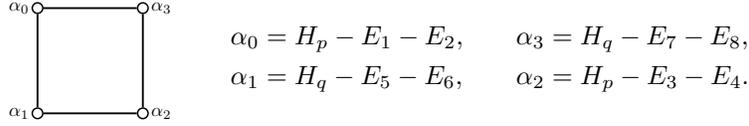

\subsubsection{Extended affine Weyl group $W(A_3^{(1)})\rtimes \Aut(A_3^{(1)})$}

The affine Weyl group of type $\mathcal{R}^{\perp}=A_3^{(1)}$ is 
\begin{equation}
    W(A_3^{(1)}) = \langle r_0, r_1,r_2,r_3\rangle,
\end{equation}
where the simple reflection $r_i$ associated to $\alpha_i$ acts on $\Pic_{\mathcal{S}}$ via the formula \eqref{eq:refformulaPic}.
For the Dynkin diagram automorphisms, we write $\Aut(A_3^{(1)})$ in terms of the generators $\sigma_1$ and $\sigma_2$, which as permutations of simple roots are given by 
\begin{equation}
		\sigma_{1} = (\alpha_{0}\alpha_{3})(\alpha_{1}\alpha_{2}),\quad 
		\sigma_{2} = (\alpha_{0}\alpha_{2})
\end{equation}
and act on $\Pic_{\mathcal{S}}$ by
\begin{equation}
    \sigma_1 = (H_q H_p)(E_1 E_7)(E_2 E_8)(E_3E_5)(E_4 E_6), \quad 
    \sigma_2 = (E_1 E_3)(E_2 E_4),
\end{equation}
where we have again used cycle notation for permutations.
The subgroup of special Dynkin diagram automorphisms, c.f. \Cref{prop:specialdynkinautos}, is  
\begin{equation}
    \Sigma = \langle \rho := \sigma_1 \sigma_2\rangle \cong \BZ / 4\BZ,    
\end{equation} 
which acts on $\Pic_{\mathcal{S}}$ according to $\rho=(H_qH_p)(E_1E_5 E_3 E_7)(E_2 E_6 E_4 E_8)$, and on $Q^{\perp}$ by $\rho=(\alpha_0\alpha_1\alpha_2\alpha_3)$. The weight lattice of the underlying finite root system is given by
\begin{equation}
    \nought{P} = \langle \omega_1,\omega_2,\omega_3\rangle \subset Q^{\perp}\otimes \BQ,
\end{equation}
where the fundamental weights $\omega_{i}$ are given by 
\begin{equation*}\label{A3w}
    \omega_1 = \frac{3}{4} \alpha_1 + \frac{1}{2} \alpha_{2} + \frac{1}{4} \alpha_3, \quad
\omega_2 = \frac{1}{2} \alpha_1 +  \alpha_{2} + \frac{1}{2} \alpha_3, 
\quad
\omega_3 = \frac{1}{4} \alpha_1 + \frac{1}{2} \alpha_{2} + \frac{3}{4} \alpha_3,
\end{equation*}
so $\alpha_i.\omega_j=-\delta_{i,j}$.
The associated translations are given, in terms of the simple reflections and the generator $\rho$ of $\Sigma$, by
\begin{equation}
         T_{\omega_1} = \rho^3 r_{2} r_{3} r_{0}, 
         \quad 
         T_{\omega_2} = \rho^2 r_0 r_3 r_1 r_0, 
         \quad 
         T_{\omega_3} = \rho r_2 r_1 r_0.
\end{equation}
Note that these translations are associated with weights that are not roots, so their action on the whole of $\Pic_{\mathcal{S}}$ cannot be written using the Kac translation formula, as explained in \Cref{rem:translations}.
Their actions on $\Pic_{\mathcal{S}}$ are computed by composing those of the generators.
Nevertheless, their actions on $Q^{\perp}$ are given by the translation formula \eqref{eq:kactransQ}. Explicitly, we have
\begin{equation*}
    T_{\omega_i}(\alpha_0) = \alpha_0 - \delta, \quad T_{\omega_i}(\alpha_i)=\alpha_i+\delta, \quad T_{\omega_i}(\alpha_j)=\alpha_j, \quad \text{for} \quad j\neq i.
\end{equation*}

\subsubsection{Cremona action}
 In what follows, we give the Cremona action of ${W}(A_{3}^{(1)})\rtimes \Aut(A_3^{(1)})$ on the family of surfaces. We express it in the format $w : (\boldsymbol{a},(q,p)) \mapsto (\bar{\boldsymbol{a}},(\bar{q},\bar{p}))$, with $\boldsymbol{a} \mapsto \bar{\boldsymbol{a}}$ being the action on $\mathscr{A}$ and $(q,p)\mapsto(\bar{q},\bar{p})$ being the map $\varphi^{(w)}_{\boldsymbol{a}}=\varepsilon_{\bar{\boldsymbol{a}}}\circ \tilde{\varphi}_{\boldsymbol{a}}^{(w)}\circ \varepsilon_{\boldsymbol{a}}^{-1}:\BP^1\times\BP^1\dashrightarrow\BP^1\times \BP^1$ written in the affine chart $\mathcal{U}_{0,0}$.

\begin{lemma}\label{thm:bir-weyl-d5} 
	The actions of simple reflections $r_{i}\in W(A_3^{(1)})$ are given by
	\begin{alignat*}{2}
		r_{0}&: 
		\left(\begin{matrix} a_{0} & a_{1} \\ a_{2} & a_{3} \end{matrix}\ ;\ t\ ;
		\begin{matrix} q \\ p \end{matrix}\right) 
		&&\mapsto 
		\left(\begin{matrix} -a_{0} & a_{0} + a_{1} \\ a_{2} & a_{0} + a_{3} \end{matrix}\ ;\ t\ ;
		\begin{matrix} \displaystyle q + \frac{a_{0}}{p + t} \\ p \end{matrix}\right)\!, \\
		r_{1}&: \left(\begin{matrix} a_{0} & a_{1} \\ a_{2} & a_{3} \end{matrix}\ ;\ t\ ;
		\begin{matrix} q \\ p \end{matrix}\right)
		&&\mapsto 
		\left(\begin{matrix} a_{0} + a_{1} & -a_{1} \\ a_{1} + a_{2} & a_{3} \end{matrix}\ ;\ t\ ;
		\begin{matrix}  q \\ \displaystyle p - \frac{a_{1}}{q} \end{matrix}\right)\!, \\
		r_{2}&: 
		\left(\begin{matrix} a_{0} & a_{1} \\ a_{2} & a_{3} \end{matrix}\ ;\ t\ ;
		\begin{matrix} q \\ p \end{matrix}\right) 
		&&\mapsto 
		\left(\begin{matrix} a_{0} & a_{1} + a_{2} \\ -a_{2} & a_{2} + a_{3} \end{matrix}\ ;\ t\ ;
		\begin{matrix} \displaystyle q + \frac{a_{2}}{p}\\ p \end{matrix}\right)\!, \\
		r_{3}&: 
		\left(\begin{matrix} a_{0} & a_{1} \\ a_{2} & a_{3} \end{matrix}\ ;\ t\ ;
		\begin{matrix} q \\ p \end{matrix}\right) 
		&&\mapsto 
		\left(\begin{matrix} a_{0}+a_{3} & a_{1} \\ a_{2}+a_{3} & -a_{3} \end{matrix}\ ;\ t\ ;
		\begin{matrix} q \\ \displaystyle  p - \frac{a_{3}}{q-1} \end{matrix}\right)\!.
	\end{alignat*}	
    The actions of the generators $\sigma_1,\sigma_2$ of $\Aut(A_3^{(1)})$ are given by 
	\begin{alignat*}{2}
		\sigma_{1}&: 
		\left(\begin{matrix} a_{0} & a_{1} \\ a_{2} & a_{3} \end{matrix}\ ;\ t\ ;
		\begin{matrix} q \\ p \end{matrix}\right) 
		&&\mapsto 
		\left(\begin{matrix} a_{3} & a_{2} \\ a_{1} & a_{0}  \end{matrix}\ ;\ -t\ ;
		\begin{matrix} \displaystyle -\frac{p}{ t} \\ q t \end{matrix}\right)\!, \\
		\sigma_{2}&: 
		\left(\begin{matrix} a_{0} & a_{1} \\ a_{2} & a_{3} \end{matrix}\ ;\ t\ ;
		\begin{matrix} q \\ p \end{matrix}\right) 
		&&\mapsto 
		\left(\begin{matrix} a_{2} & a_{1} \\ a_{0} & a_{3}  \end{matrix}\ ;\ -t\ ;
		\begin{matrix} q \\ p + t \end{matrix}\right)\!.
	\end{alignat*}

\end{lemma}

\subsubsection{Examples of discrete Painlev\'e equations}

\begin{example} \label{example:dPKNY}
From the Cremona action of the translation $T_{\omega_1} T_{\omega_2}^{-1} T_{\omega_3}\in W(A_3^{(1)}) \rtimes \Aut(A_3^{(1)})$, we obtain the discrete Painlev\'e equation
\begin{equation*} \label{eq-KNY-dP}
	\left\{
	\begin{aligned}
		\bar{q} 	&= 1 - q - \frac{a_0}{p+t} - \frac{a_2}{p}, \\
		\bar{p} 	&= - p - t + \frac{a_1}{\bar{q}} + \frac{a_3-1}{\bar{q}-1}, 
	\end{aligned}
	\qquad 
	\begin{aligned}
		\bar{a}_0 &= a_0 + 1, &&\bar{a}_1 = a_1 - 1, \\
		\bar{a}_2 &= a_2 + 1, &&\bar{a}_3 = a_3 - 1.
	\end{aligned}
	\right.
\end{equation*} 
\end{example} 
\begin{example} \label{example:dPSakai}
From the Cremona action of the translation $T_{\omega_3}\in W(A_3^{(1)}) \rtimes \Aut(A_3^{(1)})$
we obtain the discrete Painlev\'e equation
\begin{equation*} \label{eq-Sakai-dP}
	\left\{
	\begin{aligned}
		\bar{q} 	&= \frac{p+t}{t} \left( 1 - \frac{a_0 + a_1}{ a_0 + (p+t)q} \right) \!, \\
		\bar{p} 	&= \frac{1}{1-\bar{q}}\left( a_2 + \frac{(a_0+a_1)p}{t+p - t \bar{q}} \right)\!,  
	\end{aligned}
	\qquad 
	\begin{aligned}
		\bar{a}_0 &= a_0 + 1, &&\bar{a}_1 = a_1, \\
		\bar{a}_2 &= a_2, &&\bar{a}_3 = a_3 - 1.
	\end{aligned}
	\right.
\end{equation*}
\end{example}

\begin{exercise}
    In both \Cref{example:dPKNY,example:dPSakai}, verify that the map 
    \[
    \begin{tikzcd}[row sep =tiny]
        \BP^1 \times \BP^1 \arrow[r,dashed,"\varphi_{\boldsymbol{a}} "]& \BP^1 \times \BP^1 \\
        (q,p) \arrow[r,mapsto]&  (\bar{q},\bar{p}),
    \end{tikzcd}
    \]
    with the parameter evolution $\boldsymbol{a}\mapsto\bar{\boldsymbol{a}}$, lifts to an isomorphism $\tilde{\varphi}_{\boldsymbol{a}} = \varepsilon_{\bar{\boldsymbol{a}}}\circ \varphi_{\boldsymbol{a}}\circ\varepsilon_{\boldsymbol{a}}^{-1} : S_{\boldsymbol{a}}\rightarrow S_{\bar{\boldsymbol{a}}}$.
    Compute the action on $\Pic_{\mathcal{S}}$ induced by $(\tilde{\varphi}_{\boldsymbol{a}})^*$, and its restriction to $Q^{\perp}$ in terms of the basis given in \Cref{fig:a-roots-a3-KNY}. Confirm in each case that this corresponds to the specified translation element of $W(A_3^{(1)})\rtimes\Aut(A_3^{(1)})$.
\end{exercise}


\end{document}